\documentclass[11pt]{article}
\usepackage[a4paper,margin=1in]{geometry}
\usepackage{amsmath,amssymb,amsthm,mathtools,mathrsfs}
\usepackage{enumitem}
\usepackage{microtype}
\usepackage[hidelinks]{hyperref}
\usepackage[T1]{fontenc}
\usepackage[utf8]{inputenc}

\newtheorem{theorem}{Theorem}[section]
\newtheorem{proposition}[theorem]{Proposition}
\newtheorem{lemma}[theorem]{Lemma}
\newtheorem{corollary}[theorem]{Corollary}
\newtheorem{definition}[theorem]{Definition}
\newtheorem{remark}[theorem]{Remark}
\newtheorem{example}[theorem]{Example}

\newcommand{\R}{\mathbb{R}}
\newcommand{\N}{\mathbb{N}}
\newcommand{\Dp}{\mathfrak D_p}
\newcommand{\dpair}{\mathfrak d_p}
\newcommand{\V}{\mathsf V}
\newcommand{\Def}{\mathsf R}
\newcommand{\RelDef}{\mathcal D}

\newcommand{\cE}{\mathcal E}
\newcommand{\cB}{\mathcal B}
\newcommand{\supp}{\operatorname{supp}}
\newcommand{\Tdyad}{\mathbb T}
\newcommand{\Rough}{\mathscr R}

\title{$p$-roughness of paths and \\
invariance of $p$-th variation}
\author{Rama Cont\\
Mathematical Institute, University of Oxford.}
\date{2026}

\begin{document}
\maketitle

\begin{abstract}
We introduce an intrinsic notion of $p$-roughness  for continuous paths,  for  $p>1$, defined by the uniform convergence of 
 discrete $p$-energies over all shifted sufficiently fine  uniform grids.  We prove that this self-averaging property is equivalent to mesoscopic cancellation of the coarse-graining error for discrete $p$-energy, yielding a characterization that can be verified on a single uniform multiresolution. $p-$roughness refines the finite $p-$th variation property and implies the invariance of $p$-th variation across  a class of partition sequences.  
 
We prove that Brownian motion is almost surely 2-rough and that fractional Brownian motion with Hurst parameter $H$ is almost surely $1/H$-rough. We also derive   criteria for $p$-roughness based  on Faber-Schauder coefficients.   
These results yield partition-robust formulations of higher-order pathwise calculus and energy occupation measures.

Finally, we interpret coarse-graining as a renormalization flow for the $p$-energy and show that the class of strictly $p$-rough paths is stable under critical time–amplitude scaling, with linear $p$-energy profiles as the nonzero fixed profiles.
\end{abstract}
MSC Classification: 60H05, 28A80, 26A45, 60E05, 60L99
\newpage
\tableofcontents
\newpage

\section{Introduction}
\label{sec:introduction}

F\"ollmer's pathwise It\^o calculus \cite{Follmer1981} and its extensions \cite{AnanovaCont2017,ChiuCont2022,CF10B,ContJin2024,ContPerkowski2019,DavisOblojSiorpaes2018,hirai2023} use as  a starting point the class of functions with finite $p-$th order variation along a prescribed sequence of partitions.  For $x\in C([0,T],\mathbb{R}), $ the $p$-th variation along a partition sequence  $\pi_n=(0=t^n_0<...<t^n_i<t^n_{i+1}<..)$ is defined as
$$ [x]^p_\pi(t)=\mathop{\lim}_{n\to\infty}\sum_{\pi_n\cap [0,t]}|x(t_{i+1}^n)-x(t^n_i)|^p.$$
This limit,  when it exists,  may depend on the sequence of partitions $\pi=(\pi_n)_{n\geq 1}$ \cite{DavisOblojSiorpaes2018,delaVega1974,Freedman1983} .
 The resulting constructions --
pathwise (F\"ollmer) integrals,  change of variable formulas and local times   \cite{ContPerkowski2019,DavisOblojSiorpaes2018}-- depend therefore  on the underlying sequence of partitions \cite{DavisOblojSiorpaes2018,Freedman1983} and are not  intrinsic to the path in general. 
This partition-dependence is well documented in the quadratic case \cite{ChiuCont2018,ContDas2023,DavisOblojSiorpaes2018}, with  constructive examples of partition sequences leading to different limits for the same path.  

On the other hand,  the $p$-th variation is known to be almost surely invariant across a large class of partition sequences for many classes of stochastic processes,  such as Gaussian processes  \cite{Dudley1973,qian2019} or semimartingales \cite{jacod2011}.  This suggests the existence of a large class of irregular functions, containing typical sample paths of these stochastic processes,  for which $p-$th variation is well-defined and invariant across a range of partition sequences.   

Previous studies have attempted to identify such a class of irregular functions by studying invariance of $p$-th variation, for integer $p$,  through    cancellation of cross-products of fine increments grouped according to a reference partition sequence  \cite{ContDas2023,Avilez2021}.  These approaches defined `roughness'  relative to a reference partition and are limited to integer $p$.

\subsection{Overview}
We define here an {\it intrinsic} concept of  $p$-roughness for continuous paths,  which does not involve a reference partition and does not require $p>1$ to be an integer. 
As in \cite{ContDas2023},  the idea is to   control how the discrete `$p$-energy' changes when the partition is coarse-grained.
Using  an iterative decomposition of this coarse-graining error for any $p>1$, we extend the  cross-increment cancellation mechanism from the integer case to arbitrary $p>1$ and  define the \(p\)-roughness of a path (Definition \ref{def:p-roughness-intrinsic} ) by examining the fine resolution behavior of the discrete \(p\)-energy  along all sufficiently fine shifted uniform grids. 
 

Multiresolution representations have been extensively used to characterize  pointwise regularity and local oscillatory behaviour of functions in terms of wavelet coefficients \cite{jaffard1991,jaffard2006,jaffard1996,triebel2006}, leading to regularity results governed by the {\it amplitude} of Faber-Schauder or wavelet coefficients. 
Here we focus, by contrast, on an {\it irregularity} property.  We show,  in Sections \ref{sec:intrinsic-roughness} and \ref{sec:schauder-representation} and through examples in Sec. \ref{sec:randomsign}, that $p-$roughness is governed by cancellation effects in the {\it signs} (or 'phase' in Jaffard's terminology) of Faber-Schauder coefficients.
A key result is Proposition~\ref{prop:microscope-free-characterization},  which  characterizes   $p$-roughness
as equivalent to the uniform cancellation of coarse-graining errors over  {\it mesoscopic} blocks of a dyadic partition.  This cancellation criterion is then expressed in terms of Faber-Schauder coefficients  in Section \ref{sec:schauder-representation}. Section \ref{sec:phase} explains why the {\it signs} of Faber-Schauder coefficients,  not the modulus alone, determine roughness (Proposition \ref{prop:TL-phase-dependence} and Theorem \ref{thm:rademacher-schauder-rough}), echoing Jaffard's constructions of pointwise irregular functions \cite{jaffard2006}.

Theorem \ref{thm:p-rough-invariance} then shows that the $p$-roughness property  implies invariance of $p-$th variation across   perturbations of uniform grid sequences satisfying a condition linked to the modulus of continuity of the path.  This invariance does not require H\"older regularity,  and can accommodate a range of moduli of continuity for the path.

\begin{equation*}
\boxed{
\begin{gathered}
\text{\large Uniform self-averaging}
\quad\Longleftrightarrow\quad
\begin{gathered}
\text{\large Mesoscopic cancellation of}\\[-0.5mm]
\text{\large coarse-graining errors}
\end{gathered}
\quad\Longrightarrow\quad
\begin{gathered}
\text{\large Partition}\\[-0.5mm]
\text{\large invariance}
\end{gathered}
\end{gathered}}
\label{eq:intro-intrinsic-logical-structure}
\end{equation*}
We then establish in Section~\ref{sec:examples} the $p$-roughness property for some important examples of irregular functions and stochastic processes.   We first show that Brownian motion is  \(2\)-rough almost surely,  with a quantitative maximal estimate of its roughness (Theorem \ref{thm:brownian-intrinsic-roughness}) and that  fractional Brownian motion with Hurst exponent \(H\) is   almost-surely \(1/H\)-rough (Theorem \ref{thm:fbm-intrinsic-p-roughness}).  Finally,  we give a criterion for $p$-roughness for random Faber-Schauder expansions.
These examples show that  the class of $p$-rough functions is large enough to contain  typical sample paths of   these processes. 

We then use these results to establish an intrinsic formulation of the pathwise F\"ollmer-Ito calculus \cite{Follmer1981,CF10B} and its higher-order extension \cite{ContJin2024,ContPerkowski2019}, study $p$-roughness for vector-valued paths and discuss the link with $\rho-$irregularity \cite{CatellierGubinelli2016,GaleatiGubinelli2024} and local time properties. 

Finally, the definition of $p$-roughness  through successive coarse-graining operations has a direct link with the renormalization group,  which we explore in Section \ref{sec:renormalization}.

\subsection{Outline}
The paper is organized in three parts: definition and  characterization of $p$-roughness (Sections \ref{sec:pth-variation}, \ref{sec:intrinsic-roughness} and \ref{sec:schauder-representation}), its implications for partition invariance (Section \ref{sec:invariance}) and applications to stochastic processes,  pathwise calculus and fine properties of paths.

Section~\ref{sec:pth-variation} introduces the   coarse-graining error and its
  block decomposition. Section~\ref{sec:intrinsic-roughness}
defines $p$-roughness, proves its characterization in terms of mesoscopic cancellation (Proposition~\ref{prop:microscope-free-characterization})
 and establishes stability of $p$-roughness under smooth 
transformations, perturbations with vanishing $p$-energy and pathwise integration.
Section~\ref{sec:schauder-representation}   characterizes $p-$roughness in  terms of Faber--Schauder
coordinates  and
shows that $p-$roughness is a fine structure property.

Section~\ref{sec:invariance} derives the principal consequence of $p$-roughness: the invariance of $p$-th variation across  perturbations of shifted
uniform grids. We  establish a stability estimate for nearby partitions and then derive the partition-invariance theorem. A separate   criterion characterizes equality of $p$-th variation along two prescribed partition sequences. We also discuss the case of balanced partitions and Hölder-continuous paths (Section \ref{subsec:holder}).

Section~\ref{sec:examples} contains the main probabilistic results and counterexamples.
Brownian motion is shown to be $2$-rough with a quantitative estimate
uniform in mesh, phase and time (Theorem~\ref{thm:brownian-intrinsic-roughness}), and fractional Brownian
motion with Hurst parameter $H\in(0,1)$ is shown to be $1/H$-rough
almost surely (Theorem~\ref{thm:fbm-intrinsic-p-roughness}). 
Random Faber-Schauder series are used to 
construct examples and counterexamples of $p$-roughness. Independent Rademacher Schauder coefficients
produce $2$-rough paths almost surely (Theorem~\ref{thm:rademacher-schauder-rough}), whereas
Example~\ref{ex:rademacher-p4} proves almost-sure failure of
$4$-roughness for the corresponding critical random-sign series.

The remaining sections examine consequences and ramifications of these results.  
Section~\ref{sec:fourier}  defines occupation-measure stability and studies Fourier roughness and
 local-time properties.
Section~\ref{sec:applications} discusses applications to higher-order and
fractional F\"ollmer--It\^o calculus and
$p$-roughness of vector-valued paths through
 $p$-th variation tensors.  Finally, Section~\ref{sec:renormalization} links the coarse-graining flow with the renormalization group, shows the invariance of the $p$-roughness class under the renormalization flow and describes the action of
critical time--amplitude scaling on the limiting $p$-energy profiles.
\section{$p$-th variation and coarse-graining of partitions}
\label{sec:pth-variation}

\subsection{$p$-th variation along a partition sequence}

Fix \(T>0\).  Consider a   partition 
$\lambda_n
=
\{0=u_0^n<u_1^n<\cdots<u_{N(\lambda_n)}^n=T\}$ of \([0,T]\), with
vanishing mesh \begin{equation}
|\lambda_n|
:=
\max_{0\le i<N(\lambda_n)}(u_{i+1}^n-u_i^n) \longrightarrow0.
\label{eq:vanishing-mesh}
\end{equation}
For \(x\in C([0,T];\R)\), define
\begin{equation}
\omega_x(h)
:=
\sup\{|x(t)-x(s)|:s,t\in[0,T],\ |t-s|\le h\}.
\label{eq:modulus}
\end{equation}
Then \(\omega_x(h)\to0\) as \(h\downarrow0\).
For \(p>1\), define the discrete  measure
\begin{equation}
\mu_{\lambda_n}^{p,x}
:=
\sum_{i=0}^{N(\lambda_n)-1}
|x(u_{i+1}^n)-x(u_i^n)|^p\,\delta_{u_i^n}.
\label{eq:discrete-time-energy}
\end{equation}
and, for $t\in [0,T]$,
\begin{equation}
\V^p_{\lambda_n}(x;t)
:=
\sum_{i=0}^{N(\lambda_n)-1}
|x(u_{i+1}^n\wedge t)-x(u_i^n\wedge t)|^p.
\label{eq:stopped-energy}
\end{equation}
We define the $p-$th variation of $x$ on $[0,t]$ along the sequence $\lambda=(\lambda_n)_{n\geq 1}$ as the limit
\begin{equation}
[x]^p_\lambda(t):= \lim_{n\to\infty}\sum_{\lambda_n}|x(u_{i+1}^n\wedge t)-x(u^n_i\wedge t)|^p
\label{eq:pth-variation-function}
\end{equation}
if it exists \cite{ContPerkowski2019}: 
\begin{definition}[Finite \(p\)-th variation along a partition sequence]
\label{def:pth-variation}
\(x\in C([0,T];\R)\) is said to have finite  \(p\)-th variation along the sequence $\lambda=(\lambda_n)_{n\geq 1}$ 
if \(\mu_{\lambda_n}^{p,x}\) converges weakly to a finite nonatomic Borel measure \(\mu_\lambda^{p,x}\).  In that case
\begin{equation}
[x]^p_\lambda(t):=\mu_\lambda^{p,x}([0,t]),
\qquad t\in[0,T]
\label{eq:pth-variation-function}
\end{equation}
and we denote $x\in V_p(\lambda).$
\end{definition}

\paragraph{Stopped-sum characterization.}
We shall use the following characterization \cite[Lemma~1.3]{ContPerkowski2019}:
\begin{equation}
x\in V_p(\lambda)
\quad\Longleftrightarrow\quad
\V^p_{\lambda_n}(x;t)\mathop{\longrightarrow}^{n\to\infty} A(t)
\text{ for every }t\in[0,T],
\label{eq:stopped-convergence}
\end{equation}
for some continuous nondecreasing \(A\) with \(A(0)=0\).  In this case \(A=[x]^p_\lambda\), and the convergence is uniform on \([0,T]\).  See \cite[Lemma~1.3]{ContPerkowski2019} and  \cite{ContJin2024} for $p>1$.
\begin{remark}[Relation with $p$-variation]
    Finite \(p\)-th variation along a partition sequence is distinct from finite \(p\)-variation in the Wiener-Young sense. Many examples of interest, such as fractional Brownian motion, have in fact infinite $p-$variation almost-surely while having finite $p-$th variation along a large family of partition sequences \cite{ContPerkowski2019,delaVega1974,Dudley1973,qian2019,Pratelli2011,Taylor1972}.
\end{remark}

\subsection{Coarse-graining errors}
The comparison of discrete $p-$energies along different partitions involves the computation of {\it coarse-graining error} terms involving terms of the type $|\Delta_1x+...+\Delta_mx|^p-|\Delta_1x|^p-...-|\Delta_mx|^p$ where $\Delta_jx$ represent different fine increments of $x$. As we shall see, such error terms admit a simple algebraic decomposition.
\begin{definition}[Coarse-graining error]
\label{def:coarse-graining-error}
For \(p>1\), define
\begin{eqnarray}
\dpair(a,b)
:=|a+b|^p-|a|^p-|b|^p,
\label{eq:two-increment-error}\\
\forall m\geq 2,\qquad \Dp(a_1,\ldots,a_m)
:=
\left|\sum_{j=1}^m a_j\right|^p-
\sum_{j=1}^m|a_j|^p.
\label{eq:block-coarse-graining-error}
\end{eqnarray}
with \(\Dp(\varnothing)=0\).
\end{definition}

\begin{lemma}[Decomposition of the coarse-graining error]
\label{lem:coarse-graining-error-decomposition}
Let \(S_k=a_1+\cdots+a_k\).  Then, for every \(m\ge2\),
\begin{equation}
\Dp(a_1,\ldots,a_m)
=
\sum_{k=2}^{m}\dpair(S_{k-1},a_k).
\label{eq:coarse-graining-error-decomposition}
\end{equation}
\end{lemma}

\begin{proof}
By \eqref{eq:two-increment-error},
$\dpair(S_{k-1},a_k)
=|S_k|^p-|S_{k-1}|^p-|a_k|^p.$
Summing  over $k=2..m$ yields \eqref{eq:coarse-graining-error-decomposition}.
\end{proof}

The identity is valid for every real \(p>1\) and requires no polynomial expansion.  At \(p=2\),
\begin{equation}
\Dp(a_1,\ldots,a_m)
=2\sum_{1\le i<j\le m}a_i a_j,
\qquad
\dpair(P,\delta)=2P\delta.
\label{eq:quadratic-coarse-graining-error}
\end{equation}
The following elementary estimate follows  from the fundamental theorem of calculus applied to \(z\mapsto |z|^p\), together with \((u+v)^{p-1}\le C_p(u^{p-1}+v^{p-1})\):
\begin{lemma}
\label{lem:power-increment}
For \(p>1\), there exists \(c_p<\infty\) such that
\begin{equation}
\left||a+b|^p-|a|^p\right|
\le
c_p\bigl(|a|^{p-1}|b|+|b|^p\bigr)
\label{eq:power-increment}
\end{equation}
for all \(a,b\in\R\).  Consequently,
\begin{equation}
|\dpair(a,b)|
\le
C_p\bigl(|a|^{p-1}|b|+|b|^p\bigr).
\label{eq:two-increment-error-bound}
\end{equation}
\end{lemma}
We shall also use repeatedly the inequality
\begin{equation}
|a-b|^p\le2^{p-1}(|a|^p+|b|^p).
\label{eq:elementary-p-sum}
\end{equation}


Let $
\lambda=\{0=u_0<u_1<\cdots<u_M=T\}$
be a finite partition, and let
\begin{equation}
\sigma=\{u_{i_0},u_{i_1},\ldots,u_{i_K}\}\subseteq\lambda,
\qquad
0=i_0<i_1<\cdots<i_K=M,
\label{eq:finite-coarsening}
\end{equation}
be a coarsening of $\lambda$.  For \(u_j\in\lambda\), define  
\begin{equation}
\Delta_jx=x(u_{j+1})-x(u_j),\qquad\tau_\sigma(u_j)
:=
\max\{v\in\sigma:v\le u_j\}.
\label{eq:lookback-foot}
\end{equation}
 \begin{lemma}[$p$-energy perturbation estimate]
\label{lem:energy-perturbation}
Let \(p>1\). There exists a constant \(c_p<\infty\) such that, for every
finite index set \(I\) and every two families
\((a_k)_{k\in I},(e_k)_{k\in I}\subset\mathbb{R}\),
\begin{equation}
\left|
\sum_{k\in I}|a_k+e_k|^p-\sum_{k\in I}|a_k|^p
\right|
\le
c_p\left(
A^{(p-1)/p}F^{1/p}+F
\right),
\label{eq:energy-perturbation}
\end{equation}
where
\[
A:=\sum_{k\in I}|a_k|^p,
\qquad
F:=\sum_{k\in I}|e_k|^p.
\]
\end{lemma}

\begin{proof}
By Lemma~\ref{lem:power-increment},
\[
\left||a_k+e_k|^p-|a_k|^p\right|
\le
c_p\left(
|a_k|^{p-1}|e_k|+|e_k|^p
\right).
\]
Summing over \(k\in I\) and applying H\"older's inequality gives
\[
\begin{aligned}
\left|
\sum_{k\in I}|a_k+e_k|^p-\sum_{k\in I}|a_k|^p
\right|
&\le
c_p\left(
\sum_{k\in I}|a_k|^{p-1}|e_k|
+\sum_{k\in I}|e_k|^p
\right)
\\
&\le
c_p\left[
\left(\sum_{k\in I}|a_k|^p\right)^{(p-1)/p}
\left(\sum_{k\in I}|e_k|^p\right)^{1/p}
+\sum_{k\in I}|e_k|^p
\right],
\end{aligned}
\]
which is \eqref{eq:energy-perturbation}.
\end{proof}

\begin{proposition}[Lookback representation of the block coarse-graining error]
\label{prop:lookback-representation}
For every coarsening \(\sigma\subseteq\lambda\),
\begin{equation}
\V^p_\sigma(x;T)-\V^p_\lambda(x;T)
=
\sum_{j=0}^{M-1}
\dpair\bigl(x(u_j)-x(\tau_\sigma(u_j)),\Delta_jx\bigr).
\label{eq:lookback-representation}
\end{equation}
The same identity holds for stopped sums.
\end{proposition}

\begin{proof}
Apply Lemma~\ref{lem:coarse-graining-error-decomposition} separately to the fine increments contained in each \(\sigma\)-block and sum over blocks.  Within the block starting at \(\tau_\sigma(u_j)\), the partial sum preceding \(\Delta_jx\) is exactly \(x(u_j)-x(\tau_\sigma(u_j))\).  Inserting \(t\) gives the stopped version.
\end{proof}
\section{\(p\)-roughness}
\label{sec:intrinsic-roughness}
Equation~\eqref{eq:lookback-representation}  shows that partition-dependence is entirely encoded by accumulated two-increment interactions. This motivates a definition of   roughness in terms of the {\it cancellation} of these two-increment interactions   across scales and phases.
The block geometry enters only through the left   endpoint; the nonlinear observable itself is always the   two-increment interaction \(\dpair(P,\delta)\).
\subsection{Definition and    relation with Besov regularity}
For \(0<\ell\le T\) and \(0\le a<\ell\), define the shifted uniform grid
\begin{equation}
\Pi(\ell,a)
:=
\{0,T\}\cup
\{a+k\ell:k\in\mathbb Z,\ 0<a+k\ell<T\}.
\label{eq:shifted-uniform-grid}
\end{equation}

\begin{definition}[\(p\)-roughness]
\label{def:p-roughness-intrinsic}
Let \(p>1\).  A path \(x\in C([0,T])\) is called \emph{\(p\)-rough} if
\begin{equation}
\boxed{
\Omega_p^A(x;\delta)
:=
\sup_{0<\ell\le\delta}
\sup_{0\le a<\ell}
\sup_{t\in[0,T]}
\left|
\V^p_{\Pi(\ell,a)}(x;t)-A(t)
\right|\mathop{\longrightarrow}^{\delta \to 0}0.}
\label{eq:intrinsic-p-roughness}
\end{equation}
for some continuous nondecreasing function \(A:[0,T]\to[0,\infty)\) with $A(T)>A(0)=0$.
We denote the class of such paths by \(\mathscr R_p([0,T])\). 
\end{definition}

The function \(A\) in Definition~\ref{def:p-roughness-intrinsic} is unique.  Indeed, if \(A\) and \(B\) both satisfy Equation~\eqref{eq:intrinsic-p-roughness}, then for every \(\delta>0\), every \(0<\ell\le\delta\), and every phase \(a\),
\begin{equation}
\|A-B\|_\infty
\le
\Omega_p^A(x;\delta)+\Omega_p^B(x;\delta) \mathop{\to}^{\delta \to 0} 0.
\label{eq:intrinsic-energy-uniqueness}
\end{equation}
 We call $A$ the \emph{intrinsic $p$-th variation} (or \(p\)-energy)  of \(x\) and denote
\begin{equation}
[x]^p(t):=A(t),
\qquad
\mu^{p,x}([0,t]):=[x]^p(t),
\label{eq:intrinsic-energy-measure}
\end{equation}
where \(\mu^{p,x}\) is the limit measure \eqref{eq:pth-variation-function} associated with \(A\).
This notation is compatible with the notation of Definition~\ref{def:pth-variation}: for every sequence \(\lambda_n=\Pi(\ell_n,a_n)\) with \(\ell_n\to0\) and \(0\le a_n<\ell_n\), Definition~\ref{def:p-roughness-intrinsic} gives
\begin{equation}
\sup_{t\in[0,T]}
\left|
\V^p_{\lambda_n}(x;t)-[x]^p(t)
\right|
\longrightarrow0.
\label{eq:all-shifted-uniform-sequences}
\end{equation}
Hence the stopped-sum characterization \eqref{eq:stopped-convergence} yields
\begin{equation}
x\in V_p(\lambda),
\qquad
[x]^p_\lambda=[x]^p.
\label{eq:intrinsic-compatible-partitionwise}
\end{equation}
In particular, \([x]^p_{\Tdyad}=[x]^p\) for the dyadic sequence.  However Definition~\ref{def:p-roughness-intrinsic} contains no reference to the dyadic or any other specific partition sequence.
We further denote
\begin{eqnarray}
  \mathscr R^+_p([0,T]) &=&\{ x\in {\mathscr R}_p([0,T]),\qquad \forall t\in (0,T],\qquad [x]^p(t)>0\} \label{def.strictlyrough}\\
   \mathscr R^{++}_p([0,T]) &=&\{ x\in {\mathscr R}_p([0,T]),\qquad  [x]^p:[0,T]\mapsto \mathbb{R}_+\ {\rm strictly\ increasing}\} \label{def.nondegenerate}
\end{eqnarray}
We   refer to paths in $\mathscr R^+_p([0,T])$ as {\it strictly} $p$-rough and paths in $\mathscr R^{++}_p([0,T])$ as non-degenerate $p$-rough.

 $p$-roughness  is distinct from the notion of `true roughness' used in rough-path theory \cite{frizhairer}, which is a local nondegeneracy condition on fine increments of a path.   $p$-roughness is a self-averaging property of the discrete $p$-energy, requiring its uniform convergence   over shifted fine grids.

Rosenbaum \cite{rosenbaum2009} relates Besov regularity with the finite $p-$th variation property. The following lemma identifies the critical Besov regularity implied by $p$-roughness:
\begin{lemma}[Critical Besov regularity]
\label{lem:p-rough-besov}
\[
\mathscr R_p([0,T])
\subset B^{1/p}_{p,\infty}([0,T])
\subset B^{1/p}_{1,\infty}([0,T])\qquad{\rm and}\qquad{\mathscr R}_p([0,T])\cap B^{1/p}_{p,q}([0,T])=\emptyset\quad{\rm for}\quad 1\leq q <\infty.
\]
\end{lemma}
\begin{proof} Recall, for
\(0<\alpha<1\), the Besov--Nikolskii seminorm on \([0,T]\) \cite{triebel2006}:
\[
\|x\|_{B^\alpha_{1,\infty}}
:=
\sup_{0<h<T}
h^{-\alpha}
\int_0^{T-h}|x(s+h)-x(s)|\,ds.
\] Let \(x\in\mathcal R_p([0,T])\).
For \(0<h<T\), averaging the terminal \(p\)-energy of the
shifted grids \(\Pi(h,a)\) over \(a\in[0,h)\) gives
\begin{equation}
\int_0^{T-h}|x(s+h)-x(s)|^p\,ds
\le
\int_0^h \V^p_{\Pi(h,a)}(x;T)\,da.
\label{eq:phase-average-energy}
\end{equation}
To see this, note that each interval \([s,s+h]\subset[0,T]\)
appears exactly once, up to a null set of phases, as a complete
length-\(h\) cell of one of the shifted grids; the two boundary cells
only contribute additional nonnegative terms.
From Definition~\ref{def:p-roughness-intrinsic}
\[
\sup_{0\le a<h}
\V^p_{\Pi(h,a)}(x;T)
\longrightarrow [x]^p(T)
\qquad (h\downarrow0).
\]
Hence the right-hand side of
\eqref{eq:phase-average-energy} is bounded by \(C_x h\) for all
sufficiently small \(h\), which gives
$$
\int_0^{T-h}|x(s+h)-x(s)|^p\,ds
\le C_x h $$
 at small scales. Enlarging
\(C_x\), if necessary, gives the estimate for all \(0<h\le T\).
Thus
\[
\sup_{0<h\le T}
h^{-1/p}
\left(
\int_0^{T-h}|x(s+h)-x(s)|^p\,ds
\right)^{1/p}
<\infty,
\]
so \(x\in B^{1/p}_{p,\infty}([0,T])\).  H\"older's inequality then leads to
\[
\int_0^{T-h}|x(s+h)-x(s)|\,ds
\le
T^{1-1/p}
\left(
\int_0^{T-h}|x(s+h)-x(s)|^p\,ds
\right)^{1/p}
\le C_x h^{1/p},
\]
so $ x\in B^{1/p}_{1,\infty}([0,T]).$
Since
the difference between the phase-averaged grid energy and $L^p$ norm of 
the integrated increment consists precisely of the
two boundary-cell contributions,
\[
  0\le
  \frac1h\int_0^h V_{\Pi(h,a)}^p(x;T)\,da
  -\frac1h\int_0^{T-h}|x(s+h)-x(s)|^p\,ds
  \le 2 \omega_x(h)^p.
\]
The first term converges to $[x]^p(T)$ by the uniformity
in $a$ in Definition~\ref{def:p-roughness-intrinsic}, so
\begin{equation}
  \lim_{h\downarrow0}
  \frac1h\int_0^{T-h}|x(s+h)-x(s)|^p\,ds
  =[x]^p(T)>0.
  \label{eq:roughness-besov-energy-limit}
\end{equation}
For $1\le q<\infty$, the finite-difference Besov seminorm
at smoothness $1/p$ is
\[
  \|x\|_{B^{1/p}_{p,q}}^q
  =
  \int_0^T
  \left(
    \frac1h\int_0^{T-h}|x(s+h)-x(s)|^p\,ds
  \right)^{q/p}\frac{dh}{h}.
\]
 As the inner term converges to 
$[x]^p(T)>0$ by \eqref{eq:roughness-besov-energy-limit}, the integral  diverges  therefore $x\notin B^{1/p}_{p,q}([0,T])$ for $1\le q<\infty.$
\end{proof}
 \subsection{Cancellation mechanism and intrinsic characterization}
Underlying the roughness property is a {\it cancellation} mechanism across fine increments.
To expose this mechanism, consider the dyadic partition sequence
\begin{equation}
\Tdyad_m
:=
\{u_j^m=jh_m:0\le j\le2^m\},
\qquad
h_m:=T2^{-m}.
\label{eq:dyadic-grid}
\end{equation}
For \(0\le j<2^m\), define
\begin{equation}
\Delta_{m,j}x
:=x(u_{j+1}^m)-x(u_j^m),
\qquad
Z_{m,j}(x):=h_m^{-1/p}\Delta_{m,j}x.
\label{eq:normalized-dyadic-increments}
\end{equation}
Then
\begin{equation}
\V^p_{\Tdyad_m}(x;T)
=
h_m\sum_{j=0}^{2^m-1}|Z_{m,j}(x)|^p.
\label{eq:normalized-energy}
\end{equation}

For integers \(1\le b\le2^m\) and \(0\le r<b\), define the regular block coarsening
\begin{equation}
\Tdyad_m^{b,r}
:=
\{0,T\}
\cup
\{u_j^m:1\le j<2^m,\ j\equiv r\pmod b\}.
\label{eq:regular-block-coarsening}
\end{equation}
The first and last blocks may be shorter than \(b\) fine intervals; every interior block contains exactly \(b\) fine intervals.  Let \(\tau_m^{b,r}\) denote its left block endpoint map as in Equation~\eqref{eq:lookback-foot}.  Define
\begin{equation}
\mathcal C_m^p(x;b,r;t)
:=
\V^p_{\Tdyad_m^{b,r}}(x;t)
-
\V^p_{\Tdyad_m}(x;t).
\label{eq:mesoscopic-error-observable}
\end{equation}
By Proposition~\ref{prop:lookback-representation},
\begin{equation}
\mathcal C_m^p(x;b,r;t)
=
\sum_{u_j^m<t}
\dpair\left(
 x(u_j^m)-x(\tau_m^{b,r}(u_j^m)),
 x(u_{j+1}^m\wedge t)-x(u_j^m)
\right).
\label{eq:mesoscopic-lookback-formula}
\end{equation}
Equivalently, define
\begin{equation}
\beta_{b,r}(j)
:=
\max\left\{0,\ r+b\left\lfloor\frac{j-r}{b}\right\rfloor\right\},
\label{eq:regular-block-foot-index}
\end{equation}
and
\begin{equation}
Z_{m,j}(x;t)
:=h_m^{-1/p}\bigl(x(u_{j+1}^m\wedge t)-x(u_j^m\wedge t)\bigr).
\label{eq:stopped-normalized-increment}
\end{equation}
Then
\begin{equation}
\mathcal C_m^p(x;b,r;t)
=
h_m\sum_{u_j^m<t}
\dpair\left(
 \sum_{i=\beta_{b,r}(j)}^{j-1}Z_{m,i}(x),
 Z_{m,j}(x;t)
\right).
\label{eq:mesoscopic-normalized-lookback}
\end{equation}
The {\it mesoscopic} regime corresponds to
\begin{equation}
b\longrightarrow\infty,
\qquad
bh_m\longrightarrow0.
\label{eq:mesoscopic-regime}
\end{equation}
Define for \(L\ge2\),  the 'roughness modulus'
\begin{equation}
\mathfrak R_{p,L}(x)
:=
\limsup_{m\to\infty}
\sup_{\substack{L\le b\le \lfloor2^m/L\rfloor\\0\le r<b}}
\sup_{t\in[0,T]}
|\mathcal C_m^p(x;b,r;t)|.
\label{eq:roughness-modulus}
\end{equation}
$\mathfrak R_{p,L}$ measures the maximal asymptotic change in $p-$energy caused by coarse-graining over blocks with at least $L$ fine increments, over intervals at most $T/L$.
The dyadic mesoscopic cancellation condition is
\begin{equation}
\lim_{L\to\infty}\mathfrak R_{p,L}(x)=0
\label{eq:dyadic-block-cancellation}
\end{equation}

 \begin{lemma}[Finite-level stability of nearby shifted uniform grids]
\label{lem:finite-shifted-grid-stability}
Let \(0<\delta\le T\), and consider
$
\lambda=\Pi(\ell,a),
\lambda'=\Pi(\ell',a')$
with
\begin{equation}
\frac{\delta}{2}\le \ell,\ell'\le2\delta.
\label{eq:nearby-grid-shell}
\end{equation}
Assume   there exist subpartitions
\[
 \widetilde\lambda
=
\{0=\widetilde u_0<\widetilde u_1<\cdots<
\widetilde u_N=T\}\subseteq\lambda,
\qquad
\widetilde\lambda'
=
\{0=\widetilde u'_0<\widetilde u'_1<\cdots<
\widetilde u'_N=T\}\subseteq\lambda',
\]
such that
\begin{equation}
\begin{aligned}
&\lambda\setminus\widetilde\lambda
\subset (0,2\delta)\cup(T-2\delta,T),\qquad
\lambda'\setminus\widetilde\lambda'
\subset (0,2\delta)\cup(T-2\delta,T),
\end{aligned}
\label{eq:boundary-edit-windows}
\end{equation}
with at most one point deleted from each of the two boundary regions
for each partition, and $\widetilde\lambda,\widetilde\lambda'$    
  satisfy the endpoint-fixed matching condition
\begin{equation}
|\widetilde u_k-\widetilde u'_k|
\le\varepsilon,
\qquad
0\le k\le N.
\label{eq:edited-grid-matching}
\end{equation}
If
\begin{equation}
\sup_{t\in[0,T]}
\V^p_{\lambda'}(x;t)\le M,
\label{eq:finite-grid-energy-bound}
\end{equation}
then
\begin{equation}
\begin{aligned}
\sup_{t\in[0,T]}
\left|
\V^p_{\lambda}(x;t)-\V^p_{\lambda'}(x;t)
\right|
\le C_{p,T}\Bigg[
&M^{(p-1)/p}
\left(
\frac{\omega_x(2\varepsilon)^p}{\delta}
\right)^{1/p}
+\frac{\omega_x(2\varepsilon)^p}{\delta}
+\omega_x(4\delta)^p
\Bigg].
\end{aligned}
\label{eq:shifted-grid-deterministic-stability}
\end{equation}
\end{lemma}

\begin{proof}
Deleting one interior point merges two adjacent cells. Since
\[
|u+v|^p\le2^{p-1}\bigl(|u|^p+|v|^p\bigr),
\]
and at most two interior points are deleted from \(\lambda'\),
Equation~\eqref{eq:finite-grid-energy-bound} implies
\begin{equation}
\sup_{t\in[0,T]}
\V^p_{\widetilde\lambda'}(x;t)
\le C_p M.
\label{eq:edited-reference-energy-bound}
\end{equation}

For \(0\le k<N\), define the corresponding stopped increments
\[
a_k(t)
:=
x(\widetilde u'_{k+1}\wedge t)
-
x(\widetilde u'_k\wedge t),
\]
\[
b_k(t)
:=
x(\widetilde u_{k+1}\wedge t)
-
x(\widetilde u_k\wedge t),
\qquad
e_k(t):=b_k(t)-a_k(t).
\]
Since \(r\mapsto r\wedge t\) is \(1\)-Lipschitz,
Equation~\eqref{eq:edited-grid-matching} gives
\[
|e_k(t)|
\le
2\omega_x(\varepsilon)
\le
2\omega_x(2\varepsilon).
\]
By Equation~\eqref{eq:nearby-grid-shell}, each of the original grids has
at most \(C_T/\delta\) intervals, and hence the same is true of the
edited grids. Therefore, uniformly in \(t\),
\begin{equation}
F(t):=\sum_{k=0}^{N-1}|e_k(t)|^p
\le
C_{p,T}\frac{\omega_x(2\varepsilon)^p}{\delta}.
\label{eq:finite-grid-error-energy}
\end{equation}

Apply Lemma~\ref{lem:energy-perturbation} with
\(a_k=a_k(t)\) and \(e_k=e_k(t)\). Using
\eqref{eq:edited-reference-energy-bound} and
\eqref{eq:finite-grid-error-energy}, we obtain
\[
\begin{aligned}
\sup_{t\in[0,T]}
\left|
\V^p_{\widetilde\lambda}(x;t)
-
\V^p_{\widetilde\lambda'}(x;t)
\right|
\le C_{p,T}\Bigg[
&M^{(p-1)/p}
\left(
\frac{\omega_x(2\varepsilon)^p}{\delta}
\right)^{1/p}
 +
\frac{\omega_x(2\varepsilon)^p}{\delta}
\Bigg].
\end{aligned}
\]

It remains to compare each original grid with its edited version.
A deleted point merges two adjacent cells, each of length at most
\(2\delta\), so all increments affected by that deletion are supported
on an interval of length at most \(4\delta\). This remains true for
stopped increments: if \(t\) lies inside the merged cell, both the
original and merged stopped increments are increments of \(x\) over
subintervals of an interval of length at most \(4\delta\).
Consequently each affected term is bounded by
\(\omega_x(4\delta)^p\). Since only a bounded number of cells is
affected,
\[
\begin{aligned}
&
\sup_{t\in[0,T]}
\left|
\V^p_{\lambda}(x;t)
-
\V^p_{\widetilde\lambda}(x;t)
\right|
 +
\sup_{t\in[0,T]}
\left|
\V^p_{\lambda'}(x;t)
-
\V^p_{\widetilde\lambda'}(x;t)
\right|
\le
C_p\,\omega_x(4\delta)^p.
\end{aligned}
\]
Combining the last two estimates proves
\eqref{eq:shifted-grid-deterministic-stability}.
\end{proof}

\begin{proposition}[Intrinsic characterization of $p$-roughness]
\label{prop:microscope-free-characterization}
Let \(p>1\) and \(x\in C([0,T])\).  Then 
\(x\in\mathscr R_p([0,T])\) in the sense of Definition~\ref{def:p-roughness-intrinsic} if and only if
\begin{equation}
 x\in V_p(\Tdyad),
\qquad
[x]^p_{\Tdyad}(T)>0,
\qquad
\lim_{L\to\infty}\mathfrak R_{p,L}(x)=0.
\label{eq:dyadic-characterization-condition}
\end{equation}
When these conditions hold,
\begin{equation}
[x]^p=[x]^p_{\Tdyad},
\label{eq:intrinsic-equals-dyadic-energy}
\end{equation}
and for every \(L\ge2\),
\begin{equation}
\lim_{\delta\downarrow0}\ \Omega_p^{[x]^p_{\Tdyad}}(x;\delta)
\le \mathfrak R_{p,L}(x)
\le
2\ \Omega_p^{[x]^p}(x;T/L),
\label{eq:Omega-by-R}
\end{equation}
Consequently the class \(\mathscr R_p([0,T])\) is unchanged if the dyadic multiresolution in condition~\eqref{eq:dyadic-characterization-condition} is replaced by any fixed \(q\)-adic partition sequence.
\end{proposition}

\begin{proof}
Assume first that \(x\) satisfies Definition~\ref{def:p-roughness-intrinsic}.  Since \(\Tdyad_m=\Pi(h_m,0)\), Equation~\eqref{eq:intrinsic-p-roughness} gives \(x\in V_p(\Tdyad)\) and \([x]^p_{\Tdyad}=A\).  Moreover,
\begin{equation}
\Tdyad_m^{b,r}=\Pi(bh_m,rh_m).
\label{eq:block-grid-is-shifted-grid}
\end{equation}
If \(L\le b\le2^m/L\), then both \(h_m\) and \(bh_m\) are at most \(T/L\).  Hence
\begin{equation}
|\mathcal C_m^p(x;b,r;t)|
\le
2\Omega_p^A(x;T/L),
\label{eq:block-error-by-Omega}
\end{equation}
which proves the right inequality in ~\eqref{eq:Omega-by-R} and \eqref{eq:dyadic-characterization-condition}.

Conversely, assume \eqref{eq:dyadic-characterization-condition} and set \(A=[x]^p_{\Tdyad}\).  Fix \(L\ge2\).  If \(\mathfrak R_{p,L}(x)=\infty\), the left inequality in ~\eqref{eq:Omega-by-R} is automatic, so suppose from now on that
\begin{equation}
\mathfrak R_{p,L}(x)<\infty.
\label{eq:R-finite-fixed-L}
\end{equation}
For \(0<\ell\le T/(4L)\), put
\begin{equation}
N_\ell:=\left\lceil\frac{2T}{\ell}\right\rceil+4.
\label{eq:N-ell}
\end{equation} As
$\omega_x(h)\to0$ and necessarily \(m(\ell)\to\infty\) as \(\ell\downarrow0\) there exists \(m=m(\ell)\) such that
\begin{equation}
h_m
\le
\min\left\{\frac{\ell}{4L},\frac{\ell}{16N_\ell}\right\},
\qquad
N_\ell\,\omega_x(8N_\ell h_m)^p\le\varepsilon(\ell),
\label{eq:microscope-choice}
\end{equation}
where \(\varepsilon(\ell)\downarrow0\). 
Let
\begin{equation}
b:=\left\lfloor\frac{\ell}{h_m}\right\rfloor.
\label{eq:block-approximation-b}
\end{equation}
For the phase, let \(j_0=\lfloor a/h_m\rfloor\), and write its Euclidean division by \(b\) as
\begin{equation}
j_0=q b+r,
\qquad q\in\mathbb Z,
\qquad 0\le r<b.
\label{eq:block-approximation-phase}
\end{equation}
Then
\begin{equation}
|a-(r+qb)h_m|<h_m,
\qquad
|\ell-bh_m|<h_m.
\label{eq:block-approximation-one-step}
\end{equation}
The first inequality in Equation~\eqref{eq:microscope-choice} implies, for all sufficiently small \(\ell\),
\begin{equation}
L\le b\le\frac{2^m}{L},
\qquad
\frac{3\ell}{4}\le bh_m\le\ell.
\label{eq:block-approximation-window}
\end{equation}

We now compare the two \emph{infinite} lattices before truncating them to \([0,T]\).  Associate the point \(a+k\ell\) with
\[
(r+(q+k)b)h_m,
\qquad k\in\mathbb Z.
\]
Whenever \(a+k\ell\in[0,T]\), one has \(|k|\le N_\ell\), and Equation~\eqref{eq:block-approximation-one-step} yields
\begin{equation}
\left|
(a+k\ell)-(r+(q+k)b)h_m
\right|
\le
(1+|k|)h_m
\le
2N_\ell h_m.
\label{eq:block-approximation-drift}
\end{equation}
Moreover, because \(0\le a<\ell\) and \(b=\lfloor\ell/h_m\rfloor\), the quotient in Equation~\eqref{eq:block-approximation-phase} satisfies \(q\in\{0,1\}\).  Together with \(bh_m\ge3\ell/4\), this shows that any block-lattice point \((r+(q+k)b)h_m\in[0,T]\) also has \(|k|\le N_\ell\).  Hence the same displacement bound holds with the roles of the two lattices reversed.  By Equation~\eqref{eq:microscope-choice},
\begin{equation}
2N_\ell h_m\le\frac{\ell}{8},
\label{eq:block-approximation-small-drift}
\end{equation}
and both lattice spacings are at least \(3\ell/4\).  Consequently every point lying farther than \(2N_\ell h_m\) from \(\{0,T\}\) has its matched point inside \((0,T)\).  Any unmatched point of either truncated grid must therefore lie in one of the two boundary strips of width \(2N_\ell h_m\); since this width is smaller than half either lattice spacing, each strip contains at most one unmatched interior point from each grid.  Thus, after deleting at most one interior point adjacent to each endpoint from either grid, \(\Pi(\ell,a)\) and
\[
\Pi(bh_m,rh_m)=\Tdyad_m^{b,r}
\]
have matched cardinality and corresponding endpoints at distance at most \(2N_\ell h_m\).

Condition~\eqref{eq:R-finite-fixed-L}, the stopped-sum criterion \eqref{eq:stopped-convergence}, and the identity
\[
\V^p_{\Tdyad_m^{b,r}}(x;t)
=
\V^p_{\Tdyad_m}(x;t)+\mathcal C_m^p(x;b,r;t)
\]
provide a constant \(M_L<\infty\) such that, for all sufficiently large \(m\),
\begin{equation}
\sup_{\substack{L\le b\le2^m/L\\0\le r<b}}
\sup_{t\in[0,T]}
\V^p_{\Tdyad_m^{b,r}}(x;t)
\le M_L.
\label{eq:block-energy-uniform-bound}
\end{equation}
Indeed, the dyadic energies are uniformly bounded for large \(m\), while the finite-level coarse-graining error supremum is eventually bounded by \(\mathfrak R_{p,L}(x)+1\). Set
\[
\varepsilon_\ell:=2N_\ell h_m.
\]
By Equation~\eqref{eq:block-approximation-window},
\[
\frac{3\ell}{4}\le bh_m\le\ell,
\]
so the mesh sizes of
\(\Pi(\ell,a)\) and \(\mathbb T_m^{b,r}=\Pi(bh_m,rh_m)\)
both lie in \([\ell/2,2\ell]\).

The preceding lattice matching shows that, after truncation to
\([0,T]\), every unmatched interior point of either grid lies in one
of the boundary strips
\[
(0,\varepsilon_\ell)
\qquad\text{or}\qquad
(T-\varepsilon_\ell,T).
\]
Moreover, each strip contains at most one unmatched point from each
grid.  Deleting these points therefore produces subpartitions
\[
\widetilde\lambda\subseteq\Pi(\ell,a),
\qquad
\widetilde\lambda'\subseteq\mathbb T_m^{b,r},
\]
with the same number of intervals, whose endpoints \(0,T\) are fixed
and whose corresponding partition points are at distance at most
\(\varepsilon_\ell\).
By Equation~\eqref{eq:block-approximation-small-drift},
\[
\varepsilon_\ell=2N_\ell h_m\le\frac{\ell}{8},
\]
so in particular the deleted points lie in the boundary regions
\[
(0,2\ell)\cup(T-2\ell,T).
\]
Hence Lemma~\ref{lem:finite-shifted-grid-stability} applies with
\[
\delta=\ell,
\qquad
\varepsilon=\varepsilon_\ell,
\qquad
\lambda=\Pi(\ell,a),
\qquad
\lambda'=\mathbb T_m^{b,r}.
\]  
Equations~\eqref{eq:microscope-choice} and \eqref{eq:block-approximation-drift} give
\begin{equation}
\begin{aligned}
\sup_{0\le a<\ell}\sup_{t\in[0,T]}
\left|
\V^p_{\Pi(\ell,a)}(x;t)-
\V^p_{\Tdyad_{m(\ell)}^{b,r}}(x;t)
\right|
\le C_{p,M_L,T}\Bigg[
&\left(\frac{\varepsilon(\ell)}{N_\ell\ell}\right)^{1/p}
+\frac{\varepsilon(\ell)}{N_\ell\ell}
+\omega_x(4\ell)^p
\Bigg].
\end{aligned}
\label{eq:uniform-grid-dyadic-block-approximation-quant}
\end{equation}
Since \(N_\ell\ell\ge T\), the right-hand side tends to zero; hence
\begin{equation}
\sup_{0\le a<\ell}\sup_t
\left|
\V^p_{\Pi(\ell,a)}(x;t)-
\V^p_{\Tdyad_{m(\ell)}^{b,r}}(x;t)
\right|
=o(1)
\qquad(\ell\downarrow0).
\label{eq:uniform-grid-dyadic-block-approximation}
\end{equation}

Consequently,
\begin{equation}
\begin{aligned}
\sup_{0\le a<\ell}\sup_t
\left|
\V^p_{\Pi(\ell,a)}(x;t)-A(t)
\right|
\le{}&
o(1)
+
\sup_{\substack{L\le b\le2^{m(\ell)}/L\\0\le r<b}}
\sup_t|\mathcal C_{m(\ell)}^p(x;b,r;t)|
\\
&+
\left\|
\V^p_{\Tdyad_{m(\ell)}}(x;\cdot)-A
\right\|_\infty.
\end{aligned}
\label{eq:Omega-block-triangle}
\end{equation}
The last term tends to zero by the stopped-sum criterion \eqref{eq:stopped-convergence}.  Since \(m(\ell)\to\infty\), taking \(\limsup_{\ell\downarrow0}\) in Equation~\eqref{eq:Omega-block-triangle} gives the left inequality~\eqref{eq:Omega-by-R}.  Condition~\eqref{eq:dyadic-block-cancellation} provides arbitrarily large \(L\) for which \(\mathfrak R_{p,L}(x)<\infty\), and letting such \(L\to\infty\) yields Equation~\eqref{eq:intrinsic-p-roughness}.  This proves $x\in \Rough_p([0,T])$ and ~\eqref{eq:intrinsic-equals-dyadic-energy}.

The final assertion follows because Definition \ref{def:p-roughness-intrinsic} contains no reference to the dyadic partition, while the proof of the implication $x\in \Rough_p([0,T])$\(\Rightarrow\) \eqref{eq:dyadic-characterization-condition} applies verbatim to every fixed \(q\)-adic uniform multiresolution.
\end{proof}

\begin{remark}[Controls over all shifted uniform grids]
\label{rem:scale-dependent-microscope}
{\em In Proposition~\ref{prop:microscope-free-characterization} we are  not approximating all shifted uniform grids by one fixed dyadic level.  For each mesh \(\ell\), the dyadic level \(m(\ell)\) is chosen only after \(\Pi(\ell,a)\) has been specified, and is refined until the accumulated phase and mesh error over \(O(T/\ell)\) cells is negligible.  Lemma~\ref{lem:finite-shifted-grid-stability} then compares the arbitrary shifted grid with a regular dyadic blocking at that scale.  Mesoscopic block cancellation compares the latter with the dyadic energy, which converges to \([x]^p_{\Tdyad}\).  This scale-dependent choice of dyadic resolution is the reason a condition checked on one uniform multiresolution controls all shifted uniform grids and is therefore independent of the base.}
\end{remark}

\begin{corollary}[Mesoscopic block invariance]
\label{prop:roughness-block-invariance}
Let \(x\) be \(p\)-rough.  Let \(m_n\to\infty\), and choose integers \(b_n,r_n\) such that
\begin{equation}
b_n\to\infty,
\qquad
b_nh_{m_n}\to0,
\qquad
0\le r_n<b_n.
\label{eq:mesoscopic-block-sequence}
\end{equation}
Then
\begin{equation}
\sup_{t\in[0,T]}
\left|
\V^p_{\Tdyad_{m_n}^{b_n,r_n}}(x;t)-[x]^p(t)
\right|
\longrightarrow0.
\label{eq:roughness-block-invariance}
\end{equation}
In particular,
\begin{equation}
x\in V_p\bigl((\Tdyad_{m_n}^{b_n,r_n})_n\bigr),
\qquad
[x]^p_{(\Tdyad_{m_n}^{b_n,r_n})_n}=[x]^p.
\label{eq:roughness-block-variation}
\end{equation}
\end{corollary}

\begin{proof}
Apply Equation~\eqref{eq:all-shifted-uniform-sequences} to \(\ell_n=b_nh_{m_n}\) and \(a_n=r_nh_{m_n}\).
\end{proof}

\subsection{Elementary structural properties}

\begin{remark}[Scaling]
\label{prop:roughness-scaling}
If \(x\in\mathscr R_p([0,T])\) and \(c\ne0\), then \(cx\in\mathscr R_p([0,T])\), with
\begin{equation}
[cx]^p=|c|^p[x]^p,
\qquad
\Omega_p^{[cx]^p}(cx;\delta)=|c|^p\Omega_p^{[x]^p}(x;\delta).
\label{eq:roughness-scaling}
\end{equation}
Moreover,
\begin{equation}
\mathfrak R_{p,L}(cx)=|c|^p\mathfrak R_{p,L}(x).
\label{eq:roughness-block-scaling}
\end{equation}
\end{remark}
All discrete \(p\)-energies and coarse-graining errors are homogeneous of degree \(p\).

\begin{remark}[Cancellation versus smallness]
\label{rem:cancellation-vs-smallness}
For a block transformation one may also consider the total coarse-graining error mass
\begin{equation}
\overline{\mathcal C}_m^p(x;b,r;t)
:=
\sum_{B}
\left|\Dp\bigl((\Delta_{m,j}x)_{j\in B}\bigr)\right|,
\label{eq:absolute-error-mass}
\end{equation}
where the sum is over the coarse blocks.  The characterization in Proposition~\ref{prop:microscope-free-characterization} requires cancellation of the signed total coarse-graining error, not vanishing of Equation~\eqref{eq:absolute-error-mass}.  In particular, when all fine increments within all blocks have the same sign, convexity gives nonnegative block coarse-graining errors, so no nontrivial cancellation is possible.
\end{remark}

\begin{remark}[The case $p=2$ and  relation with quadratic roughness]
{\em When $p=2$, 
a different criterion for invariance of quadratic variation across balanced partitions was proposed in \cite{ContDas2023}.
This property,  called {\it quadratic roughness} \cite[Definition~3.2]{ContDas2023}, is relative to a prescribed  reference sequence: fine dyadic increments are grouped according to the cells of that sequence and the resulting cross-products are required to cancel.  
By contrast,
Definition~\ref{def:p-roughness-intrinsic} is intrinsic to the path: it tests the  self-averaging across all fine shifted uniform grids.  
At \(p=2\), Equation~\eqref{eq:mesoscopic-lookback-formula} becomes
\begin{equation}
\mathcal C_m^2(x;b,r;t)
=
2\sum_{u_j^m<t}
\bigl(x(u_j^m)-x(\tau_m^{b,r}(u_j^m))\bigr)
\bigl(x(u_{j+1}^m\wedge t)-x(u_j^m)\bigr).
\label{eq:SR2-decorrelation}
\end{equation}
Thus the block characterization of \(2\)-roughness is a mesoscopic decorrelation condition between an increment accumulated from the most recent regular block boundary and the next microscopic increment.

Proposition~\ref{prop:microscope-free-characterization} shows that this is equivalent to regular mesoscopic block cancellation, but it does not imply quadratic roughness along every balanced sequence. Section~\ref{sec:invariance} identifies the perturbations of uniform grids for which \(p\)-roughness implies partition invariance. Balanced partitions are further discussed in Section \ref{subsec:holder}.}
\end{remark}

\subsection{Transformations preserving \(p\)-roughness}
\label{subsec:first-order-stability}

The $p-$roughness property  is stable under transformations whose increments have the same first-order part as a continuous multiple of the original increment.  This gives a common mechanism for smooth changes of coordinates, smoother perturbations, and pathwise integration.

For a shifted uniform grid \(\Pi(\ell,a)=\{0=v_0<\cdots<v_N=T\}\), continuous paths \(x,y\), and \(g\in C([0,T])\), define the first-order remainder energy
\begin{equation}
\mathcal Q_p(y\mid x,g;\ell,a)
:=
\sup_{t\in[0,T]}
\sum_{i=0}^{N-1}
\left|
 y(v_{i+1}\wedge t)-y(v_i\wedge t)
 -g(v_i)\bigl(x(v_{i+1}\wedge t)-x(v_i\wedge t)\bigr)
\right|^p.
\label{eq:first-order-remainder-energy}
\end{equation}

\begin{proposition}[First-order stability of \(p\)-roughness]
\label{prop:first-order-stability}
Let \(p>1\), let \(x\in\mathscr R_p([0,T])\), and write \(\mu^{p,x}\) for its intrinsic energy measure.  Suppose \(y\in C([0,T])\) and \(g\in C([0,T])\) satisfy
\begin{equation}
\lim_{\delta\downarrow0}
\sup_{0<\ell\le\delta}
\sup_{0\le a<\ell}
\mathcal Q_p(y\mid x,g;\ell,a)
=0.
\label{eq:first-order-negligible-remainder}
\end{equation}
Then, with
\begin{equation}
B_g(t):=\int_{[0,t]}|g(s)|^p\,\mu^{p,x}(ds),
\label{eq:first-order-transformed-energy}
\end{equation}
one has
\begin{equation}
\lim_{\delta\downarrow0}
\sup_{0<\ell\le\delta}
\sup_{0\le a<\ell}
\sup_{t\in[0,T]}
\left|
\V^p_{\Pi(\ell,a)}(y;t)-B_g(t)
\right|
=0.
\label{eq:first-order-stability-conclusion}
\end{equation}
Consequently, if \(B_g(T)>0\), then \(y\in\mathscr R_p([0,T])\) and
\begin{equation}
\mu^{p,y}=|g|^p\mu^{p,x},
\qquad
[y]^p(t)=B_g(t).
\label{eq:first-order-energy-transform}
\end{equation}
If \(B_g(T)=0\), the same conclusion holds with zero limiting \(p\)-energy, so only the nondegeneracy requirement in Definition~\ref{def:p-roughness-intrinsic} fails. In particular:
\begin{enumerate}[label=\textup{(\roman*)},leftmargin=2.1em]
\item \emph{Smooth transformations.}  If \(F\in C^1(U)\) on a neighbourhood \(U\) of \(x([0,T])\) and \(y=F\circ x\), then \eqref{eq:first-order-negligible-remainder} holds with
\begin{equation}
 g(t)=F'(x(t)),
\qquad
[y]^p(t)=\int_{[0,t]}|F'(x(s))|^p\,\mu^{p,x}(ds).
\label{eq:smooth-transform-energy}
\end{equation}
Thus \(F\circ x\) is \(p\)-rough whenever the right-hand side at \(T\) is positive; in particular this is automatic if \(F'\) is bounded away from zero on \(x([0,T])\).

\item \emph{Addition of a smoother path.}  If \(y=x+z\) and
\begin{equation}
\lim_{\delta\downarrow0}
\sup_{0<\ell\le\delta}
\sup_{0\le a<\ell}
\sup_{t\in[0,T]}
\sum_i
|z(v_{i+1}\wedge t)-z(v_i\wedge t)|^p
=0,
\label{eq:smoother-perturbation-condition}
\end{equation}
then \(y\in\mathscr R_p([0,T])\) and \(\mu^{p,y}=\mu^{p,x}\).  Condition~\eqref{eq:smoother-perturbation-condition} holds, for example, when \(z\) is continuous of finite variation, or when \(z\in C^\alpha\) for some \(\alpha>1/p\).

\item \emph{Pathwise integration.}  Let \(p\ge2\) be an even integer and \(f\in C^p(\R)\).  The compensated pathwise integral  \cite[Theorem~1.5]{ContPerkowski2019}, applied along any vanishing shifted-uniform sequence, defines the same path
\begin{equation}
Y(t)
:=\int_0^t f'(x(s))\,\mathrm d^{\mathrm{CP}}x(s)
=
 f(x(t))-f(x(0))
 -\frac1{p!}\int_{[0,t]}f^{(p)}(x(s))\,\mu^{p,x}(ds),
\label{eq:CP-integral-intrinsic}
\end{equation}
because \(x\) has the same \(p\)-energy measure along every such sequence.  Then \eqref{eq:first-order-negligible-remainder} holds with \(g=f'\circ x\), and hence
\begin{equation}
[Y]^p(t)
=
\int_{[0,t]}|f'(x(s))|^p\,\mu^{p,x}(ds).
\label{eq:CP-integral-energy}
\end{equation}
Thus the pathwise integral is \(p\)-rough whenever the energy in \eqref{eq:CP-integral-energy} is nonzero.
\end{enumerate}
\end{proposition}

\begin{proof}
Fix \(h:=|g|^p\).  First note that \(p\)-roughness implies the following weighted version of Definition~\ref{def:p-roughness-intrinsic}:
\begin{equation}
\lim_{\delta\downarrow0}
\sup_{0<\ell\le\delta}
\sup_{0\le a<\ell}
\sup_{t\in[0,T]}
\left|
\sum_i h(v_i)
|x(v_{i+1}\wedge t)-x(v_i\wedge t)|^p
-
\int_{[0,t]}h\,d\mu^{p,x}
\right|
=0.
\label{eq:weighted-intrinsic-energy}
\end{equation}
Indeed, the cumulative functions of the discrete measures \(\mu_{\Pi(\ell,a)}^{p,x}\) differ from the stopped energies by at most \(2\omega_x(\ell)^p\), hence converge to \([x]^p\) uniformly in mesh, phase and time.  Approximation of the continuous function \(h\) by step functions on a fixed finite partition then yields \eqref{eq:weighted-intrinsic-energy}; the finitely many cells crossing the step boundaries contribute only \(O(\omega_x(\ell)^p)\).

For a fixed grid and stopping time, write
\begin{equation}
\Delta_i^t y
=
 g(v_i)\Delta_i^t x+r_i^t,
\qquad
R(t):=\sum_i|r_i^t|^p.
\label{eq:first-order-increment-splitting}
\end{equation}
Minkowski's inequality gives
\begin{equation}
\left|
\left(\sum_i|\Delta_i^t y|^p\right)^{1/p}
-
\left(\sum_i|g(v_i)\Delta_i^t x|^p\right)^{1/p}
\right|
\le R(t)^{1/p}.
\label{eq:first-order-minkowski}
\end{equation}
The weighted energies in the second term are uniformly bounded for all sufficiently small meshes by \eqref{eq:weighted-intrinsic-energy}; Equation~\eqref{eq:first-order-negligible-remainder} and \eqref{eq:first-order-minkowski} then imply, using \(|u^p-v^p|\le C_p(u^{p-1}+v^{p-1})|u-v|\), that the difference between the two energies tends to zero uniformly in \(\ell,a,t\).  Combining this with \eqref{eq:weighted-intrinsic-energy} proves \eqref{eq:first-order-stability-conclusion} and hence \eqref{eq:first-order-energy-transform}.

For (i), Taylor's formula at first order gives, uniformly for \(|v-u|\le\ell\),
\begin{equation}
|F(x(v))-F(x(u))-F'(x(u))(x(v)-x(u))|
\le
\omega_{F'}(\omega_x(\ell))\,|x(v)-x(u)|,
\label{eq:smooth-transform-remainder}
\end{equation}
where \(\omega_{F'}\) is the modulus of continuity of \(F'\) on a compact neighbourhood of \(x([0,T])\).  The corresponding remainder energy is therefore bounded by \(\omega_{F'}(\omega_x(\ell))^p\sup_{t}\V^p_{\Pi(\ell,a)}(x;t)\), which tends to zero uniformly.

For (ii), take \(g\equiv1\) and \(r_i^t=\Delta_i^t z\).  If \(z\) has finite variation, then
\begin{equation}
\sum_i|\Delta_i^t z|^p
\le
\omega_z(\ell)^{p-1}\operatorname{Var}_{[0,T]}(z)
\longrightarrow0.
\label{eq:finite-variation-zero-p-energy}
\end{equation}
If \(z\in C^\alpha\), the number of cells is at most \(T/\ell+2\), and
\begin{equation}
\sum_i|\Delta_i^t z|^p
\le
C\left(\frac{T}{\ell}+2\right)\ell^{\alpha p}
=O(\ell^{\alpha p-1}),
\label{eq:holder-zero-p-energy}
\end{equation}
which tends to zero when \(\alpha>1/p\).

For (iii), Equation~\eqref{eq:CP-integral-intrinsic} and first-order Taylor expansion give, for \(u<v\),
\begin{equation}
|Y(v)-Y(u)-f'(x(u))(x(v)-x(u))|
\le
C_f|x(v)-x(u)|^2
+
C_f\mu^{p,x}([u,v]).
\label{eq:CP-first-order-remainder}
\end{equation}
Hence the \(p\)-th powers of the first terms sum to at most
\(
C_f\omega_x(\ell)^p \sup_t\V^p_{\Pi(\ell,a)}(x;t)
\), uniformly in the stopping time.  For the measure term, set
\begin{equation}
 m_x(\ell)
:=
\sup_{\substack{0\le u<v\le T\\v-u\le\ell}}
\mu^{p,x}([u,v]).
\label{eq:energy-measure-modulus}
\end{equation}
Since \(\mu^{p,x}\) is nonatomic, \(m_x(\ell)\to0\), and
\begin{equation}
\sum_i\mu^{p,x}([v_i,v_{i+1}\wedge t])^p
\le
m_x(\ell)^{p-1}\mu^{p,x}([0,T])
\longrightarrow0.
\label{eq:CP-measure-remainder}
\end{equation}
Thus \eqref{eq:first-order-negligible-remainder} holds with \(g=f'\circ x\), completing the proof.
\end{proof}

\section{$p$-roughness via Faber-Schauder representation}
\label{sec:schauder-representation}

Multiresolution representations have been extensively used to characterize regularity properties of functions \cite{jaffard1991,jaffard2006,triebel2006}.
While the definition of the $p$-roughness property
(Definition~\ref{def:p-roughness-intrinsic}) does not refer to a specific   multiresolution representation, the multi-scale nature of cancellations underlying $p$-roughness leads to a natural characterization in terms of multiresolution representations, which were already used in
Proposition~\ref{prop:microscope-free-characterization}.

We use here a Faber--Schauder representation associated with a  refining partition sequence to give a characterization of the $p$-roughness property. The results are stated here for dyadic Faber-Schauder coefficients but similar results can be obtained in terms of other multiresolution representations.

\subsection{Faber-Schauder representation}

We use the notations of \cite[Section~3]{ContDas2022}. To any finitely refining sequence \(\pi\) on \([0,1]\), we can associate Haar functions, denoted by \(\psi_{m,k,i}\), their primitives  \(e_{m,k,i}^{\pi}\), and the reordered level-
\(m\) Schauder family  \(e_{m,k}^{\pi}\).  Every \(x\in C([0,1])\) has then a unique representation
\begin{equation}
x(t)=x(0)+(x(1)-x(0))t+
\sum_{m\ge0}\sum_k\theta_{m,k}^{x,\pi}e_{m,k}^{\pi}(t).
\label{eq:schauder-general-expansion}
\end{equation}
The coefficients \(\theta_{m,k}^{x,\pi}\) are the explicit three-point second differences given in \cite[Theorem~3.8]{ContDas2022}.  

For simplicity of exposition, we will focus here on the dyadic partition \(\Tdyad\) on \([0,1]\), but all results below may be transposed to other refining partitions.  
For $m\geq0$ and $0\leq k<2^m$, define the dyadic
Faber--Schauder functions by
\[
  e^{\mathbb T}_{m,k}(t)
  :=2^{-m/2}\varphi(2^m t-k),\qquad \varphi(u):=\max\{0,\min\{u,1-u\}\},
  \qquad t\in[0,1].
\]
Thus $e^{\mathbb T}_{m,k}$ is the piecewise linear tent
supported on $[k2^{-m},(k+1)2^{-m}]$, with peak height
$2^{-m/2-1}$ at $(2k+1)2^{-m-1}$. Equivalently,
\begin{eqnarray*}
  &e^{\mathbb T}_{m,k}(t)
  =\int_0^t\psi_{m,k}(s)\,ds,\qquad {\rm where}\\
  &\psi_{m,k}(s)
  :=2^{m/2}\left(
    \mathbf 1_{[k2^{-m},(2k+1)2^{-m-1})}(s)
    -
    \mathbf 1_{[(2k+1)2^{-m-1},(k+1)2^{-m})}(s)
  \right).
\end{eqnarray*}
Together with the functions $1$ and $t$, these tents
form the Faber--Schauder basis of $C([0,1])$,  ordered
by increasing level.
For \(0\le k<2^m\),
\begin{equation}
\supp(e_{m,k}^{\Tdyad})
=
\left[\frac{k}{2^m},\frac{k+1}{2^m}\right],
\qquad
t_2^{m,k}=\frac{2k+1}{2^{m+1}},
\label{eq:dyadic-schauder-support}
\end{equation}
and the general coefficient formula of \cite[Theorem~3.8]{ContDas2022} reduces to
\begin{equation}
\theta_{m,k}^{x,\Tdyad}
=
2^{m/2}
\left(
2x\left(\frac{2k+1}{2^{m+1}}\right)
-x\left(\frac{k}{2^m}\right)
-x\left(\frac{k+1}{2^m}\right)
\right).
\label{eq:dyadic-schauder-coefficient}
\end{equation}

For \(0\le \ell<m\) and \(0\le j<2^m\), let
\begin{equation}
\kappa(\ell;m,j)
:=
\left\lfloor\frac{j}{2^{m-\ell}}\right\rfloor
\label{eq:kappa-dyadic}
\end{equation}
be the unique level-\(\ell\) index whose support contains the fine interval \([j2^{-m},(j+1)2^{-m}]\), and define the Haar sign
\begin{equation}
\varepsilon_\ell(m,j)
:=
2^{-\ell/2}
\psi_{\ell,\kappa(\ell;m,j)}(t),
\qquad
t\in(j2^{-m},(j+1)2^{-m}).
\label{eq:haar-sign}
\end{equation}
Then \(\varepsilon_\ell(m,j)\in\{-1,1\}\).

\begin{proposition}[Schauder representation of the critical increment field]
\label{prop:schauder-increment-field}
For every \(m\ge1\) and \(0\le j<2^m\),
\begin{equation}
Z_{m,j}(x)
=
2^{-m(1-1/p)}
\left[
 a
 +
 \sum_{\ell=0}^{m-1}
 2^{\ell/2}\theta_{\ell,\kappa(\ell;m,j)}^{x,\Tdyad}
 \varepsilon_\ell(m,j)
\right].
\label{eq:normalized-increments-schauder}
\end{equation}
\end{proposition}

\begin{proof}
Let \(x_m\) be the Schauder partial sum through level \(m-1\).  The interpolation property of the Schauder system gives \(x_m=x\) on \(\Tdyad_m\).  On the fine interval \((j2^{-m},(j+1)2^{-m})\), \(x_m\) is affine and
\begin{equation}
x_m'(t)
=
x(1)-x(0)
+
\sum_{\ell=0}^{m-1}
\theta_{\ell,\kappa(\ell;m,j)}^{x,\Tdyad}
\psi_{\ell,\kappa(\ell;m,j)}(t).
\label{eq:schauder-partial-derivative}
\end{equation}
Integrating \eqref{eq:schauder-partial-derivative} over the fine interval and using \eqref{eq:haar-sign} gives the formula for \(\Delta_{m,j}x\).  Multiplication by \(2^{m/p}\) yields \eqref{eq:normalized-increments-schauder}.
\end{proof}
\subsection{A Faber--Schauder criterion for \(p\)-roughness}
It is a classical result that H\"older regularity and other fine properties of functions may be identified from scaling behavior of Faber-Schauder coefficients across scales.
The Ciesielski  characterization \cite{Ciesielski1960},
in the normalization above, yields
\begin{equation}
x\in C^{\alpha}([0,1])
\quad\Longleftrightarrow\quad
\sup_{m,k}2^{m(\alpha-1/2)}|\theta_{m,k}^{x,\Tdyad}|<\infty,
\qquad 0<\alpha<1.
\label{eq:ciesielski-criterion}
\end{equation}
See also the non-uniform extensions and generalized \(p\)-variation coefficient criteria in \cite{ContDas2022,DasKim2025}.  In particular, at the critical exponent \(\alpha=1/p\),
\begin{equation}
\sup_{m,k}2^{m(1/p-1/2)}|\theta_{m,k}^{x,\Tdyad}|<\infty
\qquad \Rightarrow\qquad x\in C^{1/p}([0,1]).\label{eq:critical-schauder-bound}
\end{equation}

We now show that $p$-roughness may also be identified from these coefficients.
Define, from the data \(a:=x(1)-x(0)\) and \(\theta=(\theta_{m,k}^{x,\Tdyad})\),
\begin{equation}
\mathcal Z_{m,j}(a,\theta)
:=
2^{-m(1-1/p)}
\left[
 a
 +
 \sum_{\ell=0}^{m-1}
 2^{\ell/2}\theta_{\ell,\kappa(\ell;m,j)}^{x,\Tdyad}
 \varepsilon_\ell(m,j)
\right].
\label{eq:coefficient-Z}
\end{equation}
By Proposition~\ref{prop:schauder-increment-field}, \(\mathcal Z_{m,j}(a,\theta)=Z_{m,j}(x)\).  For a regular \((b,r)\)-blocking let
\begin{equation}
\mathcal P_{m,j}^{b,r}(a,\theta)
:=
\sum_{i=\beta_{b,r}(j)}^{j-1}\mathcal Z_{m,i}(a,\theta).
\label{eq:coefficient-P}
\end{equation}
\begin{proposition}[Faber-Schauder characterization of $p$-roughness]
\label{prop:schauder-roughness-criterion}
Let $p>1$ and $x\in C([0,1])$, with affine increment
$a=x(1)-x(0)$ and dyadic Faber--Schauder coefficients
$\theta=(\theta_{m,k})$. Write $Z_{m,j}=Z_{m,j}(a,\theta)$ and
$P_{m,j}^{b,r}=P_{m,j}^{b,r}(a,\theta)$ for the coefficient fields
defined above. For $m\ge1$, define $A_m$ at dyadic points by
\begin{equation}
 A_m(k2^{-m})
 :=2^{-m}\sum_{j=0}^{k-1}|Z_{m,j}(a,\theta)|^p,
 \qquad 0\le k\le2^m,
 \label{eq:schauder-cumulative-coefficient-energy}
\end{equation}
and extend $A_m$ linearly on each dyadic cell.
Then $x$ is $p$-rough if and only if the following three
coefficient conditions hold:
\begin{enumerate}
\renewcommand{\labelenumi}{(\roman{enumi})}
\item Uniform convergence of cumulative coefficient energies:
\begin{equation}
 \lim_{M\to\infty}\sup_{m,n\ge M}
 \max_{s\in\Tdyad_{\max(m,n)}}|A_m(s)-A_n(s)|=0.
 \label{eq:schauder-coefficient-energy-cauchy}
\end{equation}
\item Nondegeneracy of the terminal coefficient energy:
\begin{equation}
 \liminf_{m\to\infty}
 2^{-m}\sum_{j=0}^{2^m-1}|Z_{m,j}(a,\theta)|^p>0.
 \label{eq:schauder-coefficient-energy-positive}
\end{equation}
\item Mesoscopic cancellation of coefficient interactions:
\begin{equation}
 \lim_{L\to\infty}\limsup_{m\to\infty}
 \sup_{\substack{L\le b\le\lfloor2^m/L\rfloor\\0\le r<b}}
 \sup_{s\in\Tdyad_m}
 \left|
 2^{-m}\sum_{j2^{-m}<s}
 d_p\bigl(P_{m,j}^{b,r}(a,\theta),Z_{m,j}(a,\theta)\bigr)
 \right|=0.
 \label{eq:schauder-roughness-criterion}
\end{equation}
\end{enumerate}
When these conditions hold,
\begin{equation}
 [x]^p(t)=\lim_{m\to\infty}A_m(t)
 \qquad\text{uniformly for }t\in[0,1].
 \label{eq:schauder-intrinsic-energy-limit}
\end{equation}
\end{proposition}

\begin{proof}
Each $A_m$ is continuous and nondecreasing, with $A_m(0)=0$.
Since $A_m-A_n$ is affine on every cell of
$\Tdyad_{\max(m,n)}$,
\[
 \|A_m-A_n\|_\infty
 =\max_{s\in\Tdyad_{\max(m,n)}}|A_m(s)-A_n(s)|.
\]
Thus \eqref{eq:schauder-coefficient-energy-cauchy} is precisely
the uniform Cauchy condition for $(A_m)$.

Set $h_m=2^{-m}$. The exact Schauder increment representation gives
\[
 x((j+1)h_m)-x(jh_m)=h_m^{1/p}Z_{m,j}(a,\theta).
\]
Consequently, $A_m$ agrees with the stopped dyadic $p$-energy
at every dyadic point. For $t\in[kh_m,(k+1)h_m]$, the difference
between either quantity and the completed-cell energy at $kh_m$
lies in $[0,\omega_x(h_m)^p]$. Hence
\begin{equation}
 \sup_{t\in[0,1]}
 \left|A_m(t)-\mathcal V_{\Tdyad_m}^p(x;t)\right|
 \le\omega_x(h_m)^p\longrightarrow0.
 \label{eq:schauder-interpolated-energy-error}
\end{equation}
It follows from the stopped-sum characterization of $V_p$ that
\eqref{eq:schauder-coefficient-energy-cauchy} is equivalent to
$x\in V_p(\Tdyad)$. Indeed, the uniform limit of the $A_m$ is
continuous and nondecreasing, and conversely the stopped dyadic
energies of a path in $V_p(\Tdyad)$ converge uniformly to its
continuous cumulative variation. Under this condition,
\eqref{eq:schauder-coefficient-energy-positive} is exactly
$[x]^p_{\Tdyad}(1)>0$.

At a dyadic stopping time $s$, the exact block decomposition gives
\[
 C_m^p(x;b,r;s)
 =2^{-m}\sum_{j2^{-m}<s}
 d_p\bigl(P_{m,j}^{b,r}(a,\theta),Z_{m,j}(a,\theta)\bigr).
\]
To pass to arbitrary stopping times, let $s\in\Tdyad_m$ be the
largest dyadic point not exceeding $t$. All coarse-block boundaries
belong to $\Tdyad_m$, so passing from $s$ to $t$ changes only the
unfinished fine cell and the unfinished coarse block. Therefore,
for $L\le b\le2^m/L$,
\begin{equation}
 \left|C_m^p(x;b,r;t)-C_m^p(x;b,r;s)\right|
 \le\omega_x(h_m)^p+2\omega_x(bh_m)^p
 \le3\omega_x(1/L)^p.
 \label{eq:schauder-continuous-stopping-error}
\end{equation}
The last bound tends to zero as $L\to\infty$, solely by continuity
of $x$. Thus \eqref{eq:schauder-roughness-criterion} is equivalent
to $\lim_{L\to\infty}\mathfrak R_{p,L}(x)=0$.

The three coefficient conditions are therefore equivalent to
\[
 x\in V_p(\Tdyad),\qquad [x]^p_{\Tdyad}(1)>0,
 \qquad\lim_{L\to\infty}\mathfrak R_{p,L}(x)=0.
\]
Proposition~\ref{prop:microscope-free-characterization} proves
their equivalence with $p$-roughness and identifies the p-th variation as in \eqref{eq:schauder-intrinsic-energy-limit}.
\end{proof}
The Faber-Schauder decomposition does not play any special role here.  Any multiresolution or wavelet system that provides a stable reconstruction of the critical increment field may be used to formulate analogous coefficient criteria.  The Faber--Schauder system is  convenient because its coefficients are explicit second differences of the path and its partial sums interpolate the path at the partition points.

\subsection{A sign-discrepancy criterion for roughness}
\label{sec:phase}
A useful  case is when coefficient magnitudes are fixed and the variability of the class is carried by signs.  Write
\begin{equation}
\theta_{m,k}=c_m\sigma_{m,k},
\qquad
\sigma_{m,k}\in\{-1,1\},
\label{eq:fixed-amplitude-sign-array}
\end{equation}
where the amplitude profile \((c_m)\) is fixed.  Then the normalized increment field in Equation~\eqref{eq:coefficient-Z} is a functional of the sign array \(\sigma=(\sigma_{m,k})\) and the affine increment \(a=x(1)-x(0)\).

For a regular \((b,r)\)-blocking at level \(m\), let \(B_{m,q}^{b,r}\) denote its consecutive full blocks of exactly \(b\) fine increments.  Define the normalized block coarse-graining error observable
\begin{equation}
G_{m,q}^{p;b,r}(\sigma)
:=
\frac1b\,
\Dp\left(
(\mathcal Z_{m,j}(a,\theta))_{j\in B_{m,q}^{b,r}}
\right).
\label{eq:normalized-sign-block-error}
\end{equation}

\begin{proposition}[Power-discrepancy criterion for \(p\)-roughness]
\label{prop:sign-discrepancy}
Assume the coefficient bound in Equation~\eqref{eq:critical-schauder-bound}, \(x\in V_p(\Tdyad)\), and \([x]^p_{\Tdyad}(1)>0\).  Suppose that the coefficients have the fixed-amplitude form~\eqref{eq:fixed-amplitude-sign-array}.  If there exist constants \(C<\infty\) and \(0<\eta\le1\) such that for every \(m\), every \(1\le b\le2^m\), every \(0\le r<b\), and every consecutive family of \(M\ge1\) full \((b,r)\)-blocks,
\begin{equation}
\left|
\sum_{q=q_0}^{q_0+M-1}
G_{m,q}^{p;b,r}(\sigma)
\right|
\le
CM^{1-\eta},
\label{eq:sign-power-discrepancy}
\end{equation}
then \(x\) is \(p\)-rough.  More precisely,
\begin{equation}
\mathfrak R_{p,L}(x)
\le
C'\left(L^{-\eta}+L^{-1}\right),
\qquad L\ge2,
\label{eq:roughness-from-sign-discrepancy}
\end{equation}
for a constant \(C'\) depending on \(p\), the critical H\"older bound and the constant in ~\eqref{eq:sign-power-discrepancy}.
\end{proposition}

\begin{proof}
Equation~\eqref{eq:critical-schauder-bound} gives \(x\in C^{1/p}\).  On a full block \(B\), homogeneity of \(\Dp\) yields
\begin{equation}
\Dp\bigl((\Delta_{m,j}x)_{j\in B}\bigr)
=
2^{-m}\,
\Dp\bigl((\mathcal Z_{m,j})_{j\in B}\bigr)
=
2^{-m}b\,G_{m,q}^{p;b,r}(\sigma).
\label{eq:block-error-sign-normalization}
\end{equation}
For a stopped sum, the complete interior blocks form a consecutive family.  If their number is \(M\), Equation~\eqref{eq:sign-power-discrepancy} gives a total complete-block contribution bounded by
\begin{equation}
C\,2^{-m}bM^{1-\eta}
\le
C'M^{-\eta}.
\label{eq:complete-block-sign-bound}
\end{equation}
In the mesoscopic window \(b\le2^m/L\), if the number \(M\) of complete blocks satisfies \(M\le L\), then \(2^{-m}bM^{1-\eta}\le L^{-\eta}\); if \(M>L\), then \(2^{-m}bM\le1+o(1)\) and the same term is bounded by \(CM^{-\eta}\le CL^{-\eta}\).  Each incomplete block has time length at most \(b2^{-m}\le L^{-1}\); critical \(1/p\)-H\"older regularity and Lemma~\ref{lem:power-increment} bound its coarse-graining error by \(C/L\).  Taking the supremum over the stopping time and phase therefore gives Equation~\eqref{eq:roughness-from-sign-discrepancy}.  Proposition~\ref{prop:microscope-free-characterization} then yields \(p\)-roughness.
\end{proof}


\subsection{$p$-roughness as a fine structure property}

\begin{proposition}[$p$-roughness is a fine structure property]
\label{prop:schauder-tail-roughness}
Let $p>1$, and let
\[
X(t)=X(0)+at+
\sum_{i=0}^{\infty}\sum_{j\in J_i}
\theta_{i,j}e_{i,j}(t),
\qquad t\in[0,T],
\]
be a random Faber--Schauder expansion, where each $J_i$ is finite
and the series converges uniformly almost surely. Define the {\it tail $\sigma-$field }
\[
F_\infty:=\bigcap_{n\ge1}\mathcal F_n\qquad{\rm where}\quad
\mathcal F_n
:=\sigma(\theta_{i,j}:i\ge n,\ j\in J_i).
\]
Then the event $\{X\in\mathscr R_p([0,T])\}$ agrees almost surely
with an event in $\mathcal F_\infty$. Thus $p$-roughness is a fine
structure property: changing finitely many coefficient levels,
 does not affect the $p-$roughness property.

If the random vectors $(\theta_{i,j})_{j\in J_i}$ are independent
across levels $i$, then
\[
\mathbb P\bigl(X\in\mathscr R_p([0,T])\bigr)\in\{0,1\}.
\]
\end{proposition}

\begin{proof}
Adding a continuous finite-variation function preserves
$p$-roughness and its intrinsic energy. Indeed, for such a function
$z$,
\[
\sup_{\ell\le\delta}\sup_{0\le a<\ell}\sup_{t\in[0,T]}
\mathcal V_{\Pi(\ell,a)}^p(z;t)
\le
\omega_z(\delta)^{p-1}\operatorname{Var}_{[0,T]}(z)
\longrightarrow0.
\]
The power-increment inequality and H\"older's inequality therefore
show that, whenever $x$ is $p$-rough, the discrete energies of
$x+z$ and $x$ differ by a quantity tending to zero uniformly in
mesh, phase and time. Applying the same argument with $-z$ gives
\[
x+z\in\mathscr R_p
\quad\Longleftrightarrow\quad
x\in\mathscr R_p.
\]
Let $C$ be the event of uniform convergence of the Schauder
series. Removing finitely many levels does not change convergence,
so $C\in\mathcal F_n$ for every $n$. On $C$, define
\[
X^{\ge n}(t)
:=\sum_{i=n}^{\infty}\sum_{j\in J_i}
\theta_{i,j}e_{i,j}(t),
\]
and set $X^{\ge n}=0$ on $C^c$. This is an
$\mathcal F_n$-measurable random element of $C([0,T])$.
On $C$, the difference $X-X^{\ge n}$ is piecewise linear and
therefore has finite variation. Consequently,
\[
X\in\mathscr R_p
\quad\Longleftrightarrow\quad
X^{\ge n}\in\mathscr R_p.
\]
The class $\mathscr R_p$ is Borel in $C([0,T])$:
its dyadic characterization involves uniform convergence of
continuous stopped-energy functions, positivity of the limiting
terminal energy, and countable limits and suprema of the
mesoscopic coarse-graining errors. Suprema over time may be
taken over a countable dense subset.
Hence the events
\[
E_n:=\{X^{\ge n}\in\mathscr R_p\}
\]
belong to $\mathcal F_n$. They coincide on $C$, and are all false
on $C^c$, since the zero path is not $p$-rough. Their common value
is therefore an event in $\mathcal F_\infty$ agreeing almost
surely with $\{X\in\mathscr R_p\}$.

Under independence across levels, Kolmogorov's zero--one law
gives the final assertion.
\end{proof}

\section{From \(p\)-roughness to partition invariance}
\label{sec:invariance}

\subsection{Perturbations of shifted uniform grids}

Let
\begin{equation}
\widehat\pi_n:=\Pi(\ell_n,a_n),
\qquad
\ell_n\longrightarrow0,
\qquad
0\le a_n<\ell_n,
\label{eq:uniform-reference-sequence}
\end{equation}
and write
\begin{equation}
\widehat\pi_n
=
\{0=\widehat t_0^n<\cdots<\widehat t_{N_n}^n=T\}.
\label{eq:uniform-reference-points}
\end{equation}
A comparison sequence
\begin{equation}
\pi_n
=
\{0=t_0^n<\cdots<t_{N_n}^n=T\}
\label{eq:uniform-perturbed-points}
\end{equation}
with the same cardinality is called a \emph{shifted-uniform perturbation} if
\begin{equation}
\delta_n
:=
\max_{0\le k\le N_n}|t_k^n-\widehat t_k^n|
\longrightarrow0.
\label{eq:uniform-grid-perturbation}
\end{equation}
Denote this geometric class by \(\mathcal A_{\rm unif}\).  No relation such as \(\delta_n=o(\ell_n)\) is built into the class; the operative hypothesis for a given path will be the endpoint condition in Theorem~\ref{thm:p-rough-invariance}.

Since all interior cells of \(\Pi(\ell_n,a_n)\) have length \(\ell_n\), with at most two truncated boundary cells,
\begin{equation}
N_n\le \frac{T}{\ell_n}+2.
\label{eq:regular-block-count}
\end{equation}
In particular, \(\delta_n=o(\ell_n)\) implies \(N_n\delta_n\to0\).

\subsection{Alignment of partition sequences along a path}
\label{sec:coarsening-alignment}

For later comparison with arbitrary reference sequences, let
\begin{equation}
\rho=(\rho_m)_{m\ge1},
\qquad |\rho_m|\to0,
\label{eq:reference-sequence}
\end{equation}
and let \(\pi=(\pi_n)\) have vanishing mesh.  Choose a strictly increasing \(r:\N\to\N\) and set
\begin{equation}
\rho_n^\star:=\rho_{r(n)},
\qquad h_n:=|\rho_n^\star|.
\label{eq:selected-reference-sequence}
\end{equation}

\begin{definition}[Alignment of partition sequences along a path]
\label{def:alignment}
Let \(x\in C([0,T])\).  The subsequence \(\rho_n^\star=\rho_{r(n)}\) is an \emph{alignment of \(\rho\) with \(\pi\) along \(x\)} if
\begin{equation}
\Theta_n(x;\pi,\rho^\star)
:=
N(\pi_n)\omega_x(h_n)^p
\longrightarrow0.
\label{eq:alignment-condition}
\end{equation}
\end{definition}
Alignment is an auxiliary choice.  In particular, Proposition~\ref{thm:necessity} shows that once \(x\in V_p(\pi)\cap V_p(\rho)\) and \([x]^p_\pi=[x]^p_\rho\), the relative coarse-graining error converges to zero uniformly in time along \emph{every} alignment of \(\rho\) with \(\pi\) along \(x\).

Such an alignment always exists: since \(|\rho_m|\to0\) and \(x\) is uniformly continuous, one may choose \(r(n)>r(n-1)\) so that \(\omega_x(|\rho_{r(n)}|)^p\le (nN(\pi_n))^{-1}\).
Write
\begin{equation}
\pi_n=\{0=t_0^n<\cdots<t_{N_n}^n=T\},
\qquad
\rho_n^\star=\{0=s_0^n<\cdots<s_{M_n}^n=T\},
\label{eq:pi-d-points}
\end{equation}
and define
\begin{equation}
\kappa_n(0)=0,
\qquad \kappa_n(N_n)=M_n,
\qquad
\kappa_n(k)=\min\{j:s_j^n\ge t_k^n\}.
\label{eq:matching-indices}
\end{equation}
Then
\begin{equation}
0\le s_{\kappa_n(k)}^n-t_k^n\le h_n.
\label{eq:alignment-endpoint-distance}
\end{equation}
For stopped fine increments set
\begin{equation}
\delta_{j,n}(t)
:=x(s_{j+1}^n\wedge t)-x(s_j^n\wedge t)\quad{\rm and}\qquad
C_{k,n}(t)
:=
\sum_{j=\kappa_n(k)}^{\kappa_n(k+1)-1}\delta_{j,n}(t).
\label{eq:grouped-increment}
\end{equation}
Define the \(k\)-th coarse-block coarse-graining error by
\begin{equation}
\mathfrak D^p_{k,n}(x;t)
:=
|C_{k,n}(t)|^p
-
\sum_{j=\kappa_n(k)}^{\kappa_n(k+1)-1}|\delta_{j,n}(t)|^p,
\label{eq:block-error-kn}
\end{equation}
and the total relative coarse-graining error by
\begin{equation}
 \Def^p_{\pi,\rho^\star;n}(x;t)
:=
\sum_{k=0}^{N_n-1}\mathfrak D^p_{k,n}(x;t)\qquad A_n(t):=\sum_{k=0}^{N_n-1}|C_{k,n}(t)|^p.
\label{eq:total-error}
\end{equation}
Since the index blocks \(\{\kappa_n(k),\ldots,\kappa_n(k+1)-1\}\) partition the fine increments, summing the block coarse-graining errors gives the identity
\begin{equation}
A_n(t)=\V^p_{\rho_n^\star}(x;t)+\Def^p_{\pi,\rho^\star;n}(x;t).
\label{eq:exact-error}
\end{equation}

\begin{definition}[Cancellation of coarse-graining error]
\label{def:relative-error-cancellation}
Assume \(x\in V_p(\rho)\).  We say that \(x\) admits \emph{cancellation of \(p\)-coarse-graining error  along \(\pi\) relative  to \(\rho\)} if there exists an alignment \(\rho^\star\) such that
\begin{equation}
\Def^p_{\pi,\rho^\star;n}(x;t)\longrightarrow0
\qquad\text{for every }t\in[0,T].
\label{eq:relative-error-cancellation}
\end{equation}
We write
\begin{equation}
x\in\RelDef^p_{\pi\mid\rho}.
\label{eq:relative-error-class}
\end{equation}
\end{definition}
This  is a property of the triple \((x,\pi,\rho)\), so not intrinsic to $x$.
\subsection{Endpoint perturbation}

\begin{proposition}[Endpoint-perturbation stability for matched partitions]
\label{prop:endpoint-perturbation}
Let \(p>1\), \(x\in C([0,T])\), and let
\begin{equation}
\lambda_n=\{0=u_0^n<\cdots<u_{N_n}^n=T\},
\qquad
\pi_n=\{0=t_0^n<\cdots<t_{N_n}^n=T\}
\label{eq:endpoint-matched-partitions}
\end{equation}
be two vanishing-mesh partition sequences with the same number of intervals.  Define
\begin{equation}
\delta_n:=\max_{0\le k\le N_n}|t_k^n-u_k^n|.
\label{eq:endpoint-displacement}
\end{equation}
Assume \(x\in V_p(\lambda)\) and
\begin{equation}
N_n\omega_x(\delta_n)^p\longrightarrow0.
\label{eq:endpoint-perturbation-condition}
\end{equation}
Then
\begin{equation}
x\in V_p(\pi),
\qquad
[x]^p_\pi=[x]^p_\lambda.
\label{eq:endpoint-variation-stability}
\end{equation}
\end{proposition}

\begin{proof}
For corresponding increments \(a_{k,n}(t)\) and \(b_{k,n}(t)\),  \(e_{k,n}(t)=b_{k,n}(t)-a_{k,n}(t)\) satisfies
\begin{equation}
|e_{k,n}(t)|\le2\omega_x(\delta_n).
\label{eq:endpoint-matched-error-bound}
\end{equation}
Hence, with \(F_n(t)=\sum_k|e_{k,n}(t)|^p\),
\begin{equation}
\sup_tF_n(t)\le2^pN_n\omega_x(\delta_n)^p\to0.
\label{eq:endpoint-error-energy-zero}
\end{equation}
For fixed \(t\), apply Lemma~\ref{lem:energy-perturbation} to
$
a_k=a_{k,n}(t),
e_k=e_{k,n}(t),
$ so that \(b_{k,n}(t)=a_{k,n}(t)+e_{k,n}(t)\).  Since
\[
\sum_{k=0}^{N_n-1}|a_{k,n}(t)|^p
=
\V^p_{\lambda_n}(x;t),
\qquad
F_n(t):=
\sum_{k=0}^{N_n-1}|e_{k,n}(t)|^p,
\]
we obtain
\begin{equation}
\left|
\V^p_{\pi_n}(x;t)-\V^p_{\lambda_n}(x;t)
\right|
\le
c_p\left(
\V^p_{\lambda_n}(x;t)^{(p-1)/p}
F_n(t)^{1/p}
+
F_n(t)
\right).
\label{eq:matched-partition-energy-perturbation}
\end{equation}
Since \(x\in V_p(\lambda)\), the reference energies are uniformly bounded by the stopped-sum criterion \eqref{eq:stopped-convergence}.  This proves $$
\sup_{t\in[0,T]}
|\V^p_{\pi_n}(x;t)-\V^p_{\lambda_n}(x;t)|
\longrightarrow0,$$ and the stopped-sum criterion \eqref{eq:stopped-convergence} yields \eqref{eq:endpoint-variation-stability}.
\end{proof}

\subsection{Coarse-graining estimate}

For the aligned construction above, define
\begin{equation}
\Delta_{k,n}^{\pi}(t)
:=x(t_{k+1}^n\wedge t)-x(t_k^n\wedge t),
\qquad
e_{k,n}(t):=\Delta_{k,n}^{\pi}(t)-C_{k,n}(t),
\label{eq:coarse-increment-error}
\end{equation}
and
\begin{equation}
E_n(t):=\sum_{k=0}^{N_n-1}|e_{k,n}(t)|^p.
\label{eq:aggregate-endpoint-error}
\end{equation}
By \eqref{eq:alignment-endpoint-distance}, each endpoint contribution to \(e_{k,n}(t)\) is bounded by \(\omega_x(h_n)\).  Hence, uniformly in \(t\),
\begin{equation}
|e_{k,n}(t)|\le2\omega_x(h_n),
\qquad
E_n(t)\le2^p\Theta_n(x;\pi,\rho^\star).
\label{eq:endpoint-mismatch-bound}
\end{equation}
Along an alignment, \(\sup_tE_n(t)\to0\).

\begin{proposition}[Coarse-graining estimate]
\label{prop:coarse-graining-estimate}
For every selected reference subsequence \(\rho_n^\star=\rho_{r(n)}\),
\begin{equation}
\begin{aligned}
&\left|
\V^p_{\pi_n}(x;t)
-
\V^p_{\rho_n^\star}(x;t)
-
\Def^p_{\pi,\rho^\star;n}(x;t)
\right|
\\
&\qquad\le
C_p\left(
A_n(t)^{(p-1)/p}\Theta_n(x;\pi,\rho^\star)^{1/p}
+
\Theta_n(x;\pi,\rho^\star)
\right).
\end{aligned}
\label{eq:coarse-graining-estimate-theta}
\end{equation}
\end{proposition}

\begin{proof}
For fixed \(t\), apply Lemma~\ref{lem:energy-perturbation} with
$
a_k=C_{k,n}(t),
e_k=e_{k,n}(t).$
Since
\[
\Delta^\pi_{k,n}(t)=C_{k,n}(t)+e_{k,n}(t),
\qquad
\sum_{k=0}^{N_n-1}|C_{k,n}(t)|^p=A_n(t),\qquad
\sum_{k=0}^{N_n-1}|e_{k,n}(t)|^p=E_n(t),
\]
Lemma~\ref{lem:energy-perturbation} yields
\[
\left|
\V^p_{\pi_n}(x;t)-A_n(t)
\right|
\le
c_p\left(
A_n(t)^{(p-1)/p}E_n(t)^{1/p}
+
E_n(t)
\right).
\]
Using~\eqref{eq:exact-error}, this becomes
\begin{equation}
\left|
\V^p_{\pi_n}(x;t)
-
\V^p_{\rho_n^\star}(x;t)
-
R^p_{\pi,\rho^\star;n}(x;t)
\right|
\le
c_p\left(
A_n(t)^{(p-1)/p}E_n(t)^{1/p}
+
E_n(t)
\right).
\label{eq:relative-energy-perturbation}
\end{equation}
Finally, Equation~\eqref{eq:endpoint-mismatch-bound} gives
\[
E_n(t)\le 2^p\Theta_n(x;\pi,\rho^\star),
\]
and substitution proves the stated estimate.
\end{proof}

\subsection{Invariance of \(p\)-th variation}

\begin{theorem}[Partition invariance of \(p\)-th variation]
\label{thm:p-rough-invariance}
Let \(p>1\), let \(x\in\mathscr R_p([0,T])\), and let \(\pi\in\mathcal A_{\rm unif}\) be associated with shifted uniform grids \(\widehat\pi_n=\Pi(\ell_n,a_n)\), and endpoint displacement \(\delta_n\) as in Equation~\eqref{eq:uniform-grid-perturbation}.  Assume
\begin{equation}
N(\pi_n)\ \omega_x(\delta_n)^p\longrightarrow0.
\label{eq:p-rough-transfer-modulus}
\end{equation}
Then
\begin{equation}
x\in V_p(\pi),
\qquad
[x]^p_\pi=[x]^p.
\label{eq:p-rough-invariance}
\end{equation}
\end{theorem}

\begin{proof}
By Definition~\ref{def:p-roughness-intrinsic},
\begin{equation}
\sup_{t\in[0,T]}
\left|
\V^p_{\widehat\pi_n}(x;t)-[x]^p(t)
\right|
\longrightarrow0.
\label{eq:uniform-reference-invariance}
\end{equation}
Hence \(x\in V_p(\widehat\pi)\) with \([x]^p_{\widehat\pi}=[x]^p\).  Equation~\eqref{eq:p-rough-transfer-modulus} is exactly the hypothesis of Proposition~\ref{prop:endpoint-perturbation} for the matched pair \((\widehat\pi_n,\pi_n)\).  Applying Proposition~\ref{prop:endpoint-perturbation} proves ~\eqref{eq:p-rough-invariance}.
\end{proof}
\begin{corollary}[Invariance for H\"older paths]
\label{cor:critical-holder-invariance}
Let $p>1$, let
$x\in\mathscr R_p([0,T])\cap C^{1/p}([0,T])$,
and let $\pi\in\mathcal A_{\mathrm{unif}}$ be associated
with shifted uniform grids
$\widehat\pi_n=\Pi(\ell_n,a_n)$.
Write $N_n$ for their common number of intervals and
$\delta_n$ for the maximal displacement of corresponding
partition points. If
\begin{equation}
  N_n\delta_n\longrightarrow0,
\label{eq:critical-holder-geometric-condition}
\end{equation}
then
\[
  x\in V_p(\pi),
  \qquad [x]^p_\pi=[x]^p,
\]
and the stopped $p$-variation sums converge uniformly
on $[0,T]$ to $[x]^p$.
In particular, the displacement condition holds whenever
$\delta_n=o(\ell_n)$.

\noindent
For $T=1$, the same conclusions hold for any
$x\in C([0,1])$ whose Faber--Schauder coefficients satisfy
\[
  \sup_{m\ge0}\max_{0\le k<2^m}
  2^{m(1/p-1/2)}|\theta_{m,k}|<\infty,
\]
together with 
\eqref{eq:schauder-coefficient-energy-cauchy},
\eqref{eq:schauder-coefficient-energy-positive}, and
\eqref{eq:schauder-roughness-criterion}.
In this case,
\[
  [x]^p_\pi(t)=\lim_{m\to\infty}A_m(t)
  \qquad\text{uniformly for }t\in[0,1],
\]
where $A_m$ is defined by
\eqref{eq:schauder-cumulative-coefficient-energy}.
\end{corollary}

\begin{proof}
If $K_x$ is a $1/p$-H\"older seminorm of $x$, then
$ N_n\omega_x(\delta_n)^p
  \le K_x^p N_n\delta_n
  \longrightarrow0.$
Theorem~\ref{thm:p-rough-invariance} therefore applies.
Since a shifted uniform grid satisfies
$N_n\le T/\ell_n+2$, the condition
$\delta_n=o(\ell_n)$ implies $N_n\delta_n\to0$.

For the coefficient formulation, the displayed coefficient
bound implies $x\in C^{1/p}([0,1])$. The three referenced
coefficient conditions imply $x\in\mathscr R_p([0,1])$
by the Faber--Schauder characterization, with intrinsic
energy given by \eqref{eq:schauder-intrinsic-energy-limit}.
The first part of the corollary now applies.
\end{proof}
The following more refined statement applies in particular to sample paths of fractional Brownian motion:
\begin{corollary}[Invariance for critical H\"older paths]
Let $p>1$ and
\[
x\in \mathscr R_p([0,T])\cap C^{1/p-}([0,T]),
\qquad
C^{1/p-}([0,T])
:=\bigcap_{0<\alpha<1/p}C^\alpha([0,T]).
\]
Let $\pi\in\mathcal A_{\mathrm{unif}}$ be associated with
shifted uniform grids $\widehat\pi_n=\Pi(\ell_n,a_n)$.
Write $N_n$ for their common number of intervals and
$\delta_n$ for the maximal displacement of corresponding
partition points.
If
\[
N_n\delta_n^{\,1-\varepsilon}\longrightarrow0,
\qquad {\rm for\ some}\quad 0<\varepsilon<1,
\]
then $
x\in V_p(\pi),$ and $ [x]^p_\pi=[x]^p,$
with uniform convergence of the stopped $p$-variation sums
on $[0,T]$.
In particular, the conclusion holds if $\delta_n=O(\ell_n^{\,1+\eta})$ for some $\eta>0.$
\end{corollary}

\begin{proof}
Choose $\alpha=(1-\varepsilon)/p$. Since
$x\in C^{1/p-}([0,T])$, there exists $K_\alpha<\infty$
such that
\[
N_n\omega_x(\delta_n)^p
\le K_\alpha^p N_n\delta_n^{\alpha p}
=K_\alpha^p N_n\delta_n^{1-\varepsilon}
\longrightarrow0.
\]
Theorem~5.5 therefore applies.
For the final assertion, we use $N_n\le T/\ell_n+2$ and
choose $0<\varepsilon<\eta/(1+\eta)$. Then
\[
N_n\delta_n^{1-\varepsilon}
=O\!\left(
\ell_n^{(1+\eta)(1-\varepsilon)-1}
\right)\longrightarrow0.
\]
\end{proof}

\subsection{Relative coarse-graining error}

For a fixed pair \((\pi,\rho)\), the coarse-graining estimate gives the following  transfer criterion.  
\begin{proposition}[Two-sequence relative coarse-graining error criterion]
\label{thm:main-equivalence}
\label{thm:sufficiency}
\label{thm:necessity}
Let \(x\in C([0,T])\cap V_p(\rho)\).  Then
\begin{equation}
 x\in V_p(\pi),\ [x]^p_\pi=[x]^p_\rho
 \quad\Longleftrightarrow\quad
 x\in\RelDef^p_{\pi\mid\rho}.
\label{eq:main-equivalence-invariance}
\end{equation}
If the left-hand side holds, then for every alignment \(\rho^\star\) of \(\rho\) with \(\pi\) along \(x\),
\begin{equation}
\sup_{t\in[0,T]}|\Def^p_{\pi,\rho^\star;n}(x;t)|\longrightarrow0.
\label{eq:necessity-error}
\end{equation}
\end{proposition}

\begin{proof}
Suppose first that the relative coarse-graining error vanishes along one alignment.  By \eqref{eq:exact-error}, for every fixed \(t\in[0,T]\) the grouped energies \(A_n(t)\) are bounded, because the aligned reference energies converge and the coarse-graining error tends to zero.  Proposition~\ref{prop:coarse-graining-estimate} therefore gives
\begin{equation*}
\V^p_{\pi_n}(x;t)-\V^p_{\rho_n^\star}(x;t)\longrightarrow0
\qquad\text{for every }t\in[0,T].
\end{equation*}
Since \(\V^p_{\rho_n^\star}(x;t)\to[x]^p_\rho(t)\), the stopped-sum criterion \eqref{eq:stopped-convergence} yields \(x\in V_p(\pi)\) and \([x]^p_\pi=[x]^p_\rho\).

Conversely, assume \(x\in V_p(\pi)\) and the two limiting energies agree.  For any alignment, the two discrete energy families converge uniformly to the same limit.  Equation~\eqref{eq:endpoint-mismatch-bound} gives \(E_n\to0\), while \(A_n\) is uniformly bounded by the coarse energies and \(E_n\).  Proposition~\ref{prop:coarse-graining-estimate} therefore yields \eqref{eq:necessity-error}.  Existence of at least one alignment was noted after Definition~\ref{def:alignment}.
\end{proof}

\subsection{Balanced partitions and H\"older-continuous paths}
\label{subsec:holder}
 A partition sequence \(\pi=(\pi_n)_{n\geq 1}\) is called \emph{balanced} \cite{ContDas2023} if
\begin{equation}
\frac{|\pi_n|}{\underline\pi_n}\le B_\pi\qquad{\rm where}\qquad \underline\pi_n
:=
\min_k(t_{k+1}^n-t_k^n)
\label{eq:balanced}
\end{equation}
for some constant \(B_\pi\).  Then
\begin{equation}
N(\pi_n)\le\frac{B_\pi T}{|\pi_n|}.
\label{eq:balanced-counting}
\end{equation}
\begin{proposition}[Alignment criterion for balanced partitions]\label{prop:holder-alignment}
Let $p>1$,   $x\in C^\alpha([0,T])$,  
$\pi=(\pi_n)_{n\ge1}$   a balanced partition sequence, and $\rho=(\rho_m)_{m\ge1}$ a partition sequence
 with $|\rho_m|\to0$.
Let $
\rho_n^\star=\rho_{r(n)}$
be a   subsequence of $\rho$ and
$h_n:=|\rho_n^\star|$ such that
\[
h_n^\beta\le C_{\mathrm{sc}}\,|\pi_n|,\qquad{\rm with} \quad\beta<p\alpha
\]
for sufficiently large $n$, for some constant  $C_{\mathrm{sc}}<\infty$.
Then $\rho^\star$ is an alignment of $\rho$ with $\pi$ along $x$.
\end{proposition}
\begin{proof}
Since $\pi$ is balanced,
\[
N(\pi_n)\le \frac{B_\pi T}{|\pi_n|}.
\]
Moreover, since $x\in C^\alpha([0,T])$,
$
\omega_x(h_n)\le K_\alpha h_n^\alpha.$
Therefore
\[
\Theta_n(x;\pi,\rho^\star)
=
N(\pi_n)\,\omega_x(h_n)^p
\le
\frac{B_\pi T}{|\pi_n|}\,
K_\alpha^p h_n^{p\alpha}.
\]
Using
$
h_n^\beta\le C_{\mathrm{sc}}\,|\pi_n|$
we obtain
\[
\Theta_n(x;\pi,\rho^\star)
\le
B_\pi T K_\alpha^p C_{\mathrm{sc}}\,
h_n^{p\alpha-\beta}.
\]
Since $h_n\to0$, the right-hand side tends to zero whenever
$\beta<p\alpha$. Hence $\rho^\star$ is an alignment of $\rho$ with
$\pi$ along $x$.
\end{proof}

At \(p=2\), quadratic roughness \cite{ContDas2023} requires cancellation of cross-products of dyadic fine increments grouped inside the cells of a prescribed balanced sequence, together with a scale condition of the form \(|\rho_n^\star|^\beta=O(|\pi_n|)\).  The intrinsic condition of Section~\ref{sec:intrinsic-roughness} is different: it tests all regular mesoscopic blockings of a canonical multiresolution and is independent of the later comparison sequence.  
 The   threshold \(\beta<2\alpha\) already observed in \cite{ContDas2023} reappears in Proposition~\ref{prop:holder-alignment}.

For integer \(p\), Avilez \cite{Avilez2021} formulated roughness through mixed monomials in a polynomial expansion.  For even integer \(p=m\),
\begin{equation}
\Dp(a_1,\ldots,a_q)
=
\sum_{\substack{k_1+\cdots+k_q=m\\\text{at least two }k_j>0}}
\binom{m}{k_1,\ldots,k_q}
\prod_{j=1}^q a_j^{k_j}.
\label{eq:multinomial-error}
\end{equation}
The decomposition \eqref{eq:coarse-graining-error-decomposition} packages these interactions into iterated two-increment coarse-graining errors and remains valid for every real \(p>1\).

\section{$p$-roughness of stochastic processes and fractal functions}
\label{sec:examples}
We now give some fundamental examples of stochastic processes whose sample paths (almost-surely) satisfy Definition \ref{def:p-roughness-intrinsic}.
\subsection{Brownian motion: intrinsic quadratic roughness}
\label{subsec:brownian-intrinsic}

The motivation for Definition~\ref{def:p-roughness-intrinsic} is to somehow capture the 'roughness' properties of typical Brownian paths. We now close the circle by showing that, in fact, 
Brownian motion almost-surely satisfies Definition~\ref{def:p-roughness-intrinsic}.  
In fact, in Theorem~\ref{thm:brownian-intrinsic-roughness} we show a stronger  quantitative self-averaging estimate from which \(2\)-roughness follows directly.
Define
\begin{equation}
\Omega_2(W;\delta):=\Omega_2^{A_0}(W;\delta),
\qquad A_0(t):=t.
\label{eq:brownian-Omega-shorthand}
\end{equation}
We will use the following  observation:
\begin{lemma}[Entropy net for shifted uniform grids]
\label{lem:shifted-grid-net}
Let
\begin{equation}
\delta_k:=T2^{-k},
\qquad
\eta_k:=T2^{-3k},
\qquad
s_k:=\frac{\eta_k\delta_k}{4T}.
\label{eq:shifted-grid-net-scales}
\end{equation}
For each \(k\ge1\), there is a finite set \(\mathcal N_k\) of parameter pairs \((\ell,a)\) with
\begin{equation}
\delta_{k+1}\le\ell\le\delta_k,
\qquad
0\le a<\ell,
\label{eq:shifted-grid-shell}
\end{equation}
such that
\begin{equation}
|\mathcal N_k|\le C_T2^{5k},
\label{eq:shifted-grid-net-cardinality}
\end{equation}
and every \((\ell,a)\) satisfying Equation~\eqref{eq:shifted-grid-shell} can be matched with some \((\ell',a')\in\mathcal N_k\) for which
\begin{equation}
|a-a'|\le\eta_k,
\qquad
|\ell-\ell'|\le s_k.
\label{eq:shifted-grid-net-parameter-distance}
\end{equation}
Moreover, for every sufficiently large \(k\), the following holds.
For every \((\ell,a)\) satisfying \eqref{eq:shifted-grid-shell}, one may
choose \((\ell',a')\in\mathcal N_k\) and subpartitions
\[
\widetilde\lambda\subseteq\Pi(\ell,a),
\qquad
\widetilde\lambda'\subseteq\Pi(\ell',a'),
\]
obtained by deleting from each grid at most one interior point in each
of the boundary regions
\[
(0,\delta_k),
\qquad
(T-\delta_k,T),
\]
such that
\[
\widetilde\lambda
=
\{0=\widetilde u_0<\cdots<\widetilde u_N=T\},
\qquad
\widetilde\lambda'
=
\{0=\widetilde u'_0<\cdots<\widetilde u'_N=T\},
\]
and
\begin{equation}
|\widetilde u_j-\widetilde u'_j|
\le 2\eta_k,
\qquad
0\le j\le N.
\label{eq:entropy-net-edited-matching}
\end{equation}
\end{lemma}

\begin{proof}
Take an \(s_k\)-net of \([\delta_{k+1},\delta_k]\) for the mesh variable and, for each mesh value, an \(\eta_k\)-net of the phase interval.  The number of mesh values is at most \(C_T\eta_k^{-1}\), because \(\delta_k/s_k=4T/\eta_k\) by Equation~\eqref{eq:shifted-grid-net-scales}.  The number of phase values for each mesh is at most \(C_T\delta_k\eta_k^{-1}\); multiplying the two bounds gives Equation~\eqref{eq:shifted-grid-net-cardinality}.  If \(a+j\ell\) and \(a'+j\ell'\) are corresponding interior points, then for \(j\le2T/\delta_k+1\),
\begin{equation}
|(a+j\ell)-(a'+j\ell')|
\le
\eta_k+
\left(\frac{2T}{\delta_k}+1\right)s_k
\le2\eta_k
\label{eq:shifted-grid-drift}
\end{equation}
for all sufficiently large \(k\).  
A discrepancy in the number of interior points can occur only in the
two truncated boundary cells. Since both mesh sizes are at most
\(\delta_k\), every such unmatched point lies in
\[
(0,\delta_k)\cup(T-\delta_k,T).
\]
Deleting at most one unmatched interior point from each boundary
region of each grid gives subpartitions with equal cardinality.
The lattice matching above   gives, for every corresponding partition point
\[
|\widetilde u_j-\widetilde u'_j|
\le2\eta_k
\]
 while \(0\) and \(T\) are
matched to themselves. This proves
\eqref{eq:entropy-net-edited-matching}.
\end{proof}

To extend the estimate from the finite net \(\mathcal N_k\) to all shifted uniform grids with \(\delta_{k+1}\le \ell\le\delta_k\), we use the finite-level stability estimate of Lemma~\ref{lem:finite-shifted-grid-stability}.
\begin{theorem}[Brownian motion is \(2\)-rough]
\label{thm:brownian-intrinsic-roughness}
Let \(W\) be a standard Brownian motion on \([0,T]\).  There is an event of probability one on which
\begin{equation}
\Omega_2(W;\delta)
=
O\left(
\sqrt{T\delta\log\frac{eT}{\delta}}
\right)
\qquad{\rm as}\quad \delta\downarrow0,
\label{eq:brownian-shifted-grid-rate}
\end{equation}
using the notation in Equation~\eqref{eq:brownian-Omega-shorthand}.  Brownian paths are therefore almost-surely 2-rough:
\begin{equation}
\mathbb{P}\left( W\in\mathscr R_2([0,T])\ \right)=1
\label{eq:brownian-R2}
\end{equation}
and
\begin{equation}
\mathfrak R_{2,L}(W)
=
O\left(
T\sqrt{\frac{\log(eL)}{L}}
\right)
\qquad{\rm a.s.}\quad{\rm as}\  L\to\infty.
\label{eq:brownian-roughness-rate}
\end{equation}
\end{theorem}

\begin{proof}
Fix a shifted grid \(\Pi(\ell,a)=\{0=v_0<\cdots<v_N=T\}\) with \(\ell\le\delta\).  At its grid points define
\begin{equation}
M_j
:=
\sum_{i=0}^{j-1}
\left(
(W_{v_{i+1}}-W_{v_i})^2-(v_{i+1}-v_i)
\right).
\label{eq:brownian-qv-martingale}
\end{equation}
The summands are independent and centered.  If \(h_i=v_{i+1}-v_i\le\delta\), then, for \(|\lambda|\le(4\delta)^{-1}\),
\begin{equation}
\mathbb E\exp\left(
\lambda\bigl((W_{v_{i+1}}-W_{v_i})^2-h_i\bigr)
\right)
=
\frac{e^{-\lambda h_i}}{\sqrt{1-2\lambda h_i}}
\le
\exp(2\lambda^2h_i^2).
\label{eq:brownian-square-mgf}
\end{equation}
Since \(\sum_i h_i^2\le T\delta\), the corresponding exponential process is a supermartingale.  Doob's maximal inequality, applied with \(\lambda>0\) and with \(-\lambda\), and optimization over \(0<\lambda\le(4\delta)^{-1}\), give
\begin{equation}
\mathbb P\left(
\max_{0\le j\le N}|M_j|>u
\right)
\le
2\exp\left[
-c\min\left
\{
\frac{u^2}{T\delta},
\frac{u}{\delta}
\right\}
\right]
\label{eq:brownian-maximal-bernstein}
\end{equation}
with a universal \(c>0\).

Apply Equation~\eqref{eq:brownian-maximal-bernstein} to every grid in the net \(\mathcal N_k\) of Lemma~\ref{lem:shifted-grid-net}, with \(\delta=\delta_k\) and
\begin{equation}
u_k
:=
A\sqrt{T\delta_k k}.
\label{eq:brownian-net-rate}
\end{equation}
For sufficiently large \(A\), Equations~\eqref{eq:shifted-grid-net-cardinality} and \eqref{eq:brownian-maximal-bernstein} imply
\begin{equation}
\sum_{k=1}^{\infty}
\mathbb P\left(
\max_{(\ell,a)\in\mathcal N_k}
\max_j |M_j|>u_k
\right)
<\infty.
\label{eq:brownian-net-BC}
\end{equation}
Hence, by Borel--Cantelli, almost surely all sufficiently fine net grids satisfy \(\max_j|M_j|\le u_k\).

For an arbitrary stopping time \(t\), let \(v_j\le t<v_{j+1}\).  The completed-cell quadratic energy differs from \(t\) by at most \(|M_j|+\delta_k\), while the stopped energy differs from the completed-cell energy by at most \(\omega_W(\delta_k)^2\).  L\'evy's modulus of continuity gives, almost surely,
\begin{equation}
\omega_W(h)^2
=O\left(h\log\frac{eT}{h}\right).
\label{eq:brownian-Levy-modulus}
\end{equation}
Thus, uniformly over all net grids in the \(k\)-th shell,
\begin{equation}
\sup_t
\left|
\V^2_{\Pi(\ell,a)}(W;t)-t
\right|
\le
u_k+O(k\delta_k).
\label{eq:brownian-net-uniform-time}
\end{equation}
It remains to pass from the net to the whole shell.  Fix
\((\ell,a)\) with
$
\delta_{k+1}\le\ell\le\delta_k,$
and choose its net representative \((\ell',a')\in\mathcal N_k\)
together with the edited subpartitions supplied by
Lemma~\ref{lem:shifted-grid-net}.
We apply Lemma~\ref{lem:finite-shifted-grid-stability} with
\[
\delta=\frac{\delta_k}{2},
\qquad
\varepsilon=2\eta_k.
\]
Indeed,
\[
\frac{\delta}{2}
=
\frac{\delta_k}{4}
\le
\ell,\ell'
\le
\delta_k
=
2\delta,
\]
and the deleted points supplied by Lemma~\ref{lem:shifted-grid-net} lie in
\[
(0,\delta_k)\cup(T-\delta_k,T)
=
(0,2\delta)\cup(T-2\delta,T).
\]
Thus all the geometric hypotheses of
Lemma~\ref{lem:finite-shifted-grid-stability} are satisfied.
Equation~\eqref{eq:brownian-Levy-modulus} and the choice
\(\eta_k=T2^{-3k}\) give
\begin{equation}
\frac{\omega_W(4\eta_k)^2}{\delta_k}
=
O(k2^{-2k}),
\qquad
\omega_W(2\delta_k)^2
=
O(k\delta_k).
\label{eq:brownian-net-transfer}
\end{equation}
Almost surely, the net energies are uniformly bounded for all
sufficiently large \(k\).  Lemma~\ref{lem:finite-shifted-grid-stability}
therefore gives, uniformly over the \(k\)-th shell,
\[
\left|
\V^2_{\Pi(\ell,a)}(W;t)
-
\V^2_{\Pi(\ell',a')}(W;t)
\right|
=
o(u_k).
\]
Hence, on one probability-one event, the supremum over the entire shell \(\delta_{k+1}\le\ell\le\delta_k\) is bounded by \(C(\omega)u_k\) for all sufficiently large \(k\).  The sequence
\(
 u_k= A\sqrt{T\delta_k k}
\)
is eventually decreasing.  Therefore, if \(\delta_{k+1}<\delta\le\delta_k\), taking the supremum over all finer shells gives
\begin{equation}
\Omega_2(W;\delta)
\le
C(\omega)\sup_{j\ge k}u_j
\le
C'(\omega)\sqrt{T\delta\log\frac{eT}{\delta}}.
\label{eq:brownian-shell-to-full-modulus}
\end{equation}
This proves Equation~\eqref{eq:brownian-shifted-grid-rate}.

Equation~\eqref{eq:brownian-shifted-grid-rate} is precisely Definition~\ref{def:p-roughness-intrinsic} with intrinsic energy \([W]^2(t)=t\), and therefore proves Equation~\eqref{eq:brownian-R2}.  The quantitative block estimate follows from Proposition~\ref{prop:microscope-free-characterization}:
\begin{equation}
\mathfrak R_{2,L}(W)
\le
2\Omega_2(W;T/L),
\label{eq:brownian-R-bound-by-Omega}
\end{equation}
which gives Equation~\eqref{eq:brownian-roughness-rate}.
\end{proof}

\begin{corollary}[Partition stability for Brownian paths]
\label{cor:brownian-regular-block-invariance}
Let \(\pi\in\mathcal A_{\rm unif}\) be associated with \(\widehat\pi_n=\Pi(\ell_n,a_n)\) and endpoint displacement \(\delta_n\).  If
\begin{equation}
\frac{
\delta_n\log\!\left(\frac{eT}{\delta_n}\right)
}{\ell_n}
\longrightarrow0,
\label{eq:brownian-endpoint-transfer-rate}
\end{equation}
then, almost surely,
\begin{equation}
W\in V_2(\pi),
\qquad
[W]^2_\pi(t)=t,
\qquad t\in[0,T].
\label{eq:brownian-regular-block-invariance}
\end{equation}
\end{corollary}

\begin{proof}
On the probability-one event of Theorem~\ref{thm:brownian-intrinsic-roughness} and Equation~\eqref{eq:brownian-Levy-modulus}, Equation~\eqref{eq:regular-block-count} gives
\begin{equation}
N_n\omega_W(\delta_n)^2
\le
C(\omega)
\left(\frac{T}{\ell_n}+2\right)
\delta_n\log\!\left(\frac{eT}{\delta_n}\right)
\longrightarrow0.
\label{eq:brownian-main-transfer-condition}
\end{equation}
Theorem~\ref{thm:p-rough-invariance} then yields Equation~\eqref{eq:brownian-regular-block-invariance}.
\end{proof}

Theorem~\ref{thm:brownian-intrinsic-roughness} is stronger in one direction and weaker in another than \cite[Theorem~3.4]{ContDas2023}: it is intrinsic and uniform over all regular mesoscopic block sizes and phases, whereas the Cont--Das theorem treats each prescribed balanced comparison sequence and allows variable block geometry.  Neither statement formally contains the other.

\subsection{Critical roughness of Fractional Brownian motion}
\label{subsec:fbm-p-roughness}

Let \(B^H\) be fractional Brownian motion with Hurst parameter \(H\in(0,1)\), normalized by
\begin{equation}
\mathbb E[(B_t^H-B_s^H)^2]=|t-s|^{2H},
\label{eq:fbm-normalization}
\end{equation}
and set
\begin{equation}
p:=\frac1H,
\qquad
m_p:=\mathbb E|Z|^p,
\qquad
Z\sim N(0,1),
\qquad
\vartheta_H:=H\wedge(1-H).
\label{eq:fbm-critical-p}
\end{equation}

Write
\begin{equation}
A_H(t):=m_pt,
\qquad
\Omega_p(B^H;\delta):=\Omega_p^{A_H}(B^H;\delta).
\label{eq:fbm-Omega-shorthand}
\end{equation}
The key estimate is a concentration bound for the \(p\)-energy on an arbitrary deterministic partition.

\begin{lemma}[Gaussian concentration for critical fBm energy]
\label{lem:fbm-critical-concentration}
Let
\begin{equation}
\lambda=\{0=t_0<t_1<\cdots<t_N=t\}
\label{eq:fbm-deterministic-partition}
\end{equation}
be a deterministic partition of \([0,t]\) with mesh at most \(\delta\), and set
\begin{equation}
F_\lambda
:=
\left(
\sum_{i=0}^{N-1}|B^H_{t_{i+1}}-B^H_{t_i}|^p
\right)^{1/p}.
\label{eq:fbm-energy-norm}
\end{equation}
Then 
\begin{equation}
\mathbb P\left(
|F_\lambda-\mathbb EF_\lambda|>u
\right)
\le
2\exp\left(
-\frac{c_{H,T}u^2}{\delta^{2\vartheta_H}}
\right).
\label{eq:fbm-F-concentration}
\end{equation}
For every \(\varepsilon>0\) there exist \(c_{\varepsilon,H,T}>0\) and \(\delta_0>0\) such that, uniformly over \(t\in[0,T]\) and over all such partitions with \(\delta\le\delta_0\),
\begin{equation}
\mathbb P\left(
\left|
\sum_i|B^H_{t_{i+1}}-B^H_{t_i}|^p-m_pt
\right|>\varepsilon
\right)
\le
2\exp\left(
-c_{\varepsilon,H,T}\delta^{-2\vartheta_H}
\right).
\label{eq:fbm-energy-concentration}
\end{equation}
\end{lemma}

\begin{proof}
Let \(X_i=B^H_{t_{i+1}}-B^H_{t_i}\) and let \(\Gamma\) be the covariance matrix of \(X=(X_i)\).
If \(0<H<1/2\), disjoint fBm increments are negatively correlated.  Moreover,
\begin{equation}
\sum_j\Gamma_{ij}
=
\operatorname{Cov}(X_i,B_t^H)\ge0.
\label{eq:fbm-row-sum-positive}
\end{equation}
Hence
\begin{equation}
\sum_j|\Gamma_{ij}|
\le2\Gamma_{ii}
\le2\delta^{2H},
\label{eq:fbm-cov-op-small-H}
\end{equation}
so \(\|\Gamma\|_{\mathrm{op}}\le2\delta^{2H}\).  Since \(p=1/H>2\), \(\|v\|_p\le\|v\|_2\), and a representation \(X=\Gamma^{1/2}Z\) shows that \(z\mapsto\|\Gamma^{1/2}z\|_p\) is \(\sqrt2\delta^H\)-Lipschitz.  The case \(H=1/2\) is the same calculation with independent increments and Lipschitz constant at most \(\delta^{1/2}\).

If \(1/2<H<1\), let \(D=\operatorname{diag}(t_{i+1}-t_i)\).  The covariance-kernel representation gives, for every vector \(a\),
\begin{equation}
a^\top\Gamma a
=
H(2H-1)
\int_0^t\int_0^t
f_a(u)f_a(v)|u-v|^{2H-2}\,du\,dv,
\label{eq:fbm-cov-kernel}
\end{equation}
where \(f_a=\sum_i a_i\mathbf 1_{(t_i,t_{i+1}]}\).  Since \(|u-v|^{2H-2}\in L^1([-T,T])\), Young's inequality yields
\begin{equation}
\Gamma\le C_{H,T}D
\label{eq:fbm-cov-dominance}
\end{equation}
in the positive-semidefinite order.  Then
\begin{equation}
X=D^{1/2}R^{1/2}Z,
\qquad
R=D^{-1/2}\Gamma D^{-1/2},
\qquad
\|R\|_{\mathrm{op}}\le C_{H,T},
\label{eq:fbm-cov-factorization}
\end{equation}
where $Z\sim N(0,{\rm Id}_N)$. Using \(p<2\), H\"older's inequality gives
\begin{equation}
\|D^{1/2}w\|_p
\le
\|w\|_2
\left(
\sum_i(t_{i+1}-t_i)^{p/(2-p)}
\right)^{(2-p)/(2p)}
\le
C_{H,T}\delta^{1-H}\|w\|_2.
\label{eq:fbm-weighted-lp-bound}
\end{equation}
Thus $F_\lambda= g(Z)$ where $g$ is Lipschitz-continuous with Lipschitz constant  bounded by \(C_{H,T}\delta^{1-H}\).  The  Gaussian concentration inequality then gives Equation~\eqref{eq:fbm-F-concentration}.
At the critical order \(Hp=1\),
\begin{equation}
\mathbb E F_\lambda^p
=
\sum_i\mathbb E|X_i|^p
=
m_p\sum_i(t_{i+1}-t_i)
=
m_pt.
\label{eq:fbm-critical-expectation}
\end{equation}
Integrating Equation~\eqref{eq:fbm-F-concentration} gives, for every fixed \(r\ge1\),
\begin{equation}
\|F_\lambda-\mathbb EF_\lambda\|_{L^r}
\le
C_{r,H,T}\delta^{\vartheta_H}.
\label{eq:fbm-F-moments}
\end{equation}
Since \(\|F_\lambda\|_{L^p}=(m_pt)^{1/p}\le(m_pT)^{1/p}\), Lemma~\ref{lem:power-increment}, H\"older's inequality, and Equation~\eqref{eq:fbm-F-moments} yield the uniform estimate
\begin{equation}
\left|
\mathbb EF_\lambda^p-(\mathbb EF_\lambda)^p
\right|
\le
C_{H,p,T}\left(
\delta^{\vartheta_H}+\delta^{p\vartheta_H}
\right).
\label{eq:fbm-center-power-error}
\end{equation}
Together with Equation~\eqref{eq:fbm-critical-expectation}, this implies, uniformly for \(t\) bounded away from zero,
\begin{equation}
\left|\mathbb EF_\lambda-(m_pt)^{1/p}\right|
\longrightarrow0
\qquad(\delta\downarrow0).
\label{eq:fbm-F-center}
\end{equation}
Hence a fixed deviation of the \(p\)-energy from \(m_pt\) forces a fixed deviation of \(F_\lambda\) from its mean, and Equation~\eqref{eq:fbm-F-concentration} applies.  For small \(t\), choose \(t_0=t_0(\varepsilon)\) so that \(m_pt_0<\varepsilon/4\).  Jensen's inequality gives \(\mathbb EF_\lambda\le(m_pt)^{1/p}\); therefore an upper energy deviation larger than \(\varepsilon\) again forces a fixed positive deviation of \(F_\lambda\) from its mean, while the lower deviation is impossible when \(t\le t_0\).  Combining the two ranges proves Equation~\eqref{eq:fbm-energy-concentration} uniformly in \(t\).
\end{proof}

\begin{theorem}[Fractional Brownian motion is critically \(p\)-rough]
\label{thm:fbm-intrinsic-p-roughness}
Let \(B^H\) be fractional Brownian motion with Hurst parameter \(H\in(0,1)\), and let \(p=1/H\).  There exists an event \(\Omega_H\) of probability one such that, for every sample path on \(\Omega_H\),
\begin{equation}
\lim_{\delta\downarrow0}
\sup_{0<\ell\le\delta}
\sup_{0\le a<\ell}
\sup_{t\in[0,T]}
\left|
\V^p_{\Pi(\ell,a)}(B^H;t)-m_pt
\right|
=0.
\label{eq:fbm-uniform-shifted-grids}
\end{equation}
Consequently,
\begin{equation}
B^H\in\mathscr R_p([0,T])
\qquad\text{almost surely},
\label{eq:fbm-intrinsic-roughness}
\end{equation}
and
\begin{equation}
[B^H]^p_{\Tdyad}(t)=m_pt,
\qquad t\in[0,T].
\label{eq:fbm-dyadic-critical-variation}
\end{equation}
\end{theorem}

\begin{proof}
Using the uniform almost-sure modulus of continuity of fBm  \cite[Theorem 4.1]{qian2019}, there exists a set  of trajectories with  probability one, on which 
\begin{equation}
\omega_{B^H}(h)
\le
C_H h^H\sqrt{\log\frac{eT}{h}}
\label{eq:fbm-modulus}
\end{equation}
holds for all sufficiently small \(h\).  Fix \(\varepsilon>0\) and consider the shell \(\delta_{k+1}\le\ell\le\delta_k\) and the net \(\mathcal N_k\) from Lemma~\ref{lem:shifted-grid-net}.  Each net grid has at most \(C_T2^k\) grid points.  For each of its grid points \(v\), apply ~\eqref{eq:fbm-energy-concentration} to the deterministic subpartition of \([0,v]\).  By ~\eqref{eq:shifted-grid-net-cardinality}, a union bound gives
\begin{equation}
\mathbb P\left(
\max_{(\ell,a)\in\mathcal N_k}
\max_{v\in\Pi(\ell,a)}
\left|
\V^p_{\Pi(\ell,a)}(B^H;v)-m_pv
\right|>\varepsilon
\right)
\le
C_T2^{6k}
\exp\left(-c_{\varepsilon,H,T}2^{2\vartheta_Hk}\right).
\label{eq:fbm-net-union-bound}
\end{equation}
The series on the right-hand side is summable.  For each \(j\in\N\), apply Borel--Cantelli with \(\varepsilon=1/j\), and intersect the resulting countable family of probability-one events.  On this single event the maximum in Equation~\eqref{eq:fbm-net-union-bound} tends to zero.

For a fixed net grid and arbitrary \(t\), let \(v\le t\) be the preceding grid point.  The completed-cell energy at time \(v\) is already controlled, \(|t-v|\le\delta_k\), so the target values differ by at most \(m_p\delta_k\); the stopped energy differs from the completed-cell energy by at most \(\omega_{B^H}(\delta_k)^p\).  Equation~\eqref{eq:fbm-modulus} and \(Hp=1\) give
\begin{equation}
\omega_{B^H}(\delta_k)^p
\le
C_H\delta_k k^{p/2}
\longrightarrow0.
\label{eq:fbm-stopping-error}
\end{equation}
Thus convergence is uniform in \(t\) on every net grid.
We now use Lemma~\ref{lem:finite-shifted-grid-stability}.  Since \(\eta_k=T2^{-3k}\), Equation~\eqref{eq:fbm-modulus} gives
\begin{equation}
\frac{\omega_{B^H}(4\eta_k)^p}{\delta_k}
\le
C_H2^{-2k}k^{p/2}
\longrightarrow0,
\qquad
\omega_{B^H}(2\delta_k)^p
\le
C_H\delta_k k^{p/2}
\longrightarrow0.
\label{eq:fbm-parameter-perturbation}
\end{equation}
We now transfer the convergence from the entropy net to the whole
shell.  For a grid \(\Pi(\ell,a)\) with
$
\delta_{k+1}\le\ell\le\delta_k,$
we choose the net representative and edited subpartitions supplied by
Lemma~\ref{lem:shifted-grid-net}.  As in the Brownian case,
Lemma~\ref{lem:finite-shifted-grid-stability} applies with
\[
\delta=\frac{\delta_k}{2},
\qquad
\varepsilon=2\eta_k.
\]
Indeed, both grid meshes lie in
$
[\delta_k/2,\delta_k]
\subseteq
[\delta/2,2\delta],$
the deleted points lie in
$
(0,\delta_k)\cup(T-\delta_k,T)
=
(0,2\delta)\cup(T-2\delta,T),$
and the corresponding edited partition points are at distance at
most \(2\eta_k\).

Since \(\eta_k=T2^{-3k}\), Equation~\eqref{eq:fbm-modulus} gives
\begin{equation}
\frac{\omega_{B^H}(4\eta_k)^p}{\delta_k}
\le
C_H2^{-2k}k^{p/2}
\longrightarrow0,
\qquad
\omega_{B^H}(2\delta_k)^p
\le
C_H\delta_k k^{p/2}
\longrightarrow0.
\label{eq:fbm-net-transfer}
\end{equation}
The net energies are eventually uniformly bounded, so
Lemma~\ref{lem:finite-shifted-grid-stability} implies
\[
\sup_{\delta_{k+1}\le\ell\le\delta_k}
\sup_{0\le a<\ell}
\sup_{t\in[0,T]}
\left|
\V^p_{\Pi(\ell,a)}(B^H;t)
-
\V^p_{\Pi(\ell',a')}(B^H;t)
\right|
\longrightarrow0.
\]
Together with the uniform convergence on the net, this transfers the
convergence to every shifted uniform grid in the shell.
The preceding countable intersection over \(\varepsilon=1/j\), together with the fact that the shells exhaust all sufficiently small \(\ell\), gives ~\eqref{eq:fbm-uniform-shifted-grids} with probability one.

Equation~\eqref{eq:fbm-uniform-shifted-grids} is exactly Definition~\ref{def:p-roughness-intrinsic} with intrinsic energy \([B^H]^p(t)=m_pt\), and therefore gives  ~\eqref{eq:fbm-intrinsic-roughness}.  Taking \(a=0\) and \(\ell=h_m\) gives Equation~\eqref{eq:fbm-dyadic-critical-variation}.  Proposition~\ref{prop:microscope-free-characterization} also yields
\begin{equation}
\mathfrak R_{p,L}(B^H)
\le
2
\sup_{0<\ell\le T/L}
\sup_{0\le a<\ell}
\sup_t
\left|
\V^p_{\Pi(\ell,a)}(B^H;t)-m_pt
\right|.
\label{eq:fbm-R-bound-by-shifted}
\end{equation}
\end{proof}

 The new point in Theorem~\ref{thm:fbm-intrinsic-p-roughness} is the maximal uniformity in the block mesh and phase, which yields the intrinsic cancellation condition rather than only convergence along one geometric sequence.  The case \(H=1/2\) contains the qualitative Brownian assertion of Theorem~\ref{thm:brownian-intrinsic-roughness}, but the latter provides a sharper explicit rate in Equation~\eqref{eq:brownian-shifted-grid-rate}.

\begin{corollary}[Partition invariance for $p-$th variation of fBm]
\label{cor:fbm-partition-stability}
Let \(p=1/H\), and let \(\pi\in\mathcal A_{\rm unif}\) be associated with \(\widehat\pi_n=\Pi(\ell_n,a_n)\) and endpoint displacement \(\delta_n\).  If
\begin{equation}
\frac{
\delta_n
\left(\log\frac{eT}{\delta_n}\right)^{p/2}
}{\ell_n}
\longrightarrow0,
\label{eq:fbm-partition-transfer-condition}
\end{equation}
then almost surely
\begin{equation}
B^H\in V_p(\pi),
\qquad
[B^H]^p_\pi(t)=m_pt.
\label{eq:fbm-partition-stability}
\end{equation}
\end{corollary}

\begin{proof}
On the probability-one event of Theorem~\ref{thm:fbm-intrinsic-p-roughness} and Equation~\eqref{eq:fbm-modulus}, Equation~\eqref{eq:regular-block-count} gives
\begin{equation}
N_n\omega_{B^H}(\delta_n)^p
\le
C_H
\left(\frac{T}{\ell_n}+2\right)
\delta_n
\left(\log\frac{eT}{\delta_n}\right)^{p/2}
\longrightarrow0.
\label{eq:fbm-main-transfer-condition}
\end{equation}
Applying Theorem~\ref{thm:p-rough-invariance} gives the result.
\end{proof}

\subsection{Faber-Schauder series: partition stability and intrinsic roughness}
\label{sec:randomsign}

We now compare natural-grid variation, stability under small endpoint
perturbations, and intrinsic roughness for explicit fractal series.
Schied \cite{Schied2016}, Mishura and Schied \cite{Mishura-Schied2016,Mishura-Schied2019}, Han and Schied \cite{han2021} and Schied and Zhang \cite{SchiedZhang2020} have proposed various constructions of irregular functions and random processes with prescribed variation
along partitions;
more general constructions on non-dyadic refining partitions are given in
\cite{ContDas2022}. As we will see, these examples may or may not possess intrinsic roughness in the sense of Definition \ref{def:p-roughness-intrinsic}.

\paragraph{Two classes of fractal series.}
The general Schied--Zhang series \cite{SchiedZhang2020} is defined as follows.
Fix an integer \(b\ge2\), let \(\phi:\R\to\R\) be \(1\)-periodic and Lipschitz with \(\phi(k)=0\) for \(k\in\mathbb Z\), and let
\begin{equation}
\frac1b<|\alpha|<1.
\label{eq:SZ-alpha-range}
\end{equation}
Define
\begin{equation}
f(t):=\sum_{m=0}^{\infty}\alpha^m\phi(b^m t),
\qquad t\in[0,1],
\label{eq:SZ-function}
\end{equation}
and set
\begin{equation}
H:=-\log_b|\alpha|,
\qquad
p:=\frac1H.
\label{eq:SZ-critical-exponents}
\end{equation}

The signed Takagi--Landsberg class of \cite{Mishura-Schied2019} is instead
defined, in the dyadic Schauder normalization of
Section~\ref{sec:schauder-representation}, by
\begin{equation}
\mathfrak X^H
:=\left\{
x_\sigma(t)=\sum_{m=0}^{\infty}2^{m(1/2-H)}
\sum_{k=0}^{2^m-1}\sigma_{m,k}e^{\mathbb T}_{m,k}(t):
\ \sigma_{m,k}\in\{-1,1\}
\right\},\qquad 0<H<1.
\label{eq:signed-TL-class}
\end{equation}
The series converges uniformly and satisfies the critical Schauder bound
\eqref{eq:critical-schauder-bound} when \(p=1/H\); in particular,
every member is \(H\)-H\"older. The all-positive member is
\begin{equation}
T_H(t):=\sum_{m=0}^{\infty}2^{-mH}\phi_\triangle(2^m t),
\qquad \phi_\triangle(u):=\operatorname{dist}(u,\mathbb Z).
\label{eq:TL-positive-function}
\end{equation}
Thus \(T_H\) is the specialization of \eqref{eq:SZ-function} with
\(b=2\), \(\alpha=2^{-H}\), and \(\phi=\phi_\triangle\).
The arbitrary location-dependent signs in \eqref{eq:signed-TL-class}
are additional freedom: the signed class is not being identified with
the general periodic-seed series \eqref{eq:SZ-function}.

\paragraph{Natural-grid variation and endpoint stability.}
Let
\begin{equation}
\mathbb T_n^{(b)}:=\{kb^{-n}:0\le k\le b^n\}.
\label{eq:b-adic-partition}
\end{equation}
Under the nondegeneracy assumptions in \cite{SchiedZhang2020},
\begin{equation}
[f]^p_{\mathbb T^{(b)}}(t)=c_pt,
\qquad c_p>0.
\label{eq:SZ-linear-variation}
\end{equation}
For the normalized path
\begin{equation}
x(t):=c_p^{-1/p}f(t),
\label{eq:SZ-normalization}
\end{equation}
one has
\begin{equation}
[x]^p_{\mathbb T^{(b)}}(t)=t.
\label{eq:SZ-unit-variation}
\end{equation}
The standard scale decomposition for the Schied--Zhang series gives \(x\in C^H([0,1])\), with \(H\) as in \eqref{eq:SZ-critical-exponents}; see \cite{SchiedZhang2020}.  Since \(pH=1\), one has \(\omega_x(h)^p\le Ch\).

For the signed class, the natural grid is \(\mathbb T=\mathbb T^{(2)}\).
Let \((\varepsilon_j)_{j\ge1}\) be independent symmetric
\(\{-1,1\}\)-valued random variables, and define
\begin{equation}
q:=2^{H-1},\qquad
S_H:=\sum_{j=1}^{\infty}q^j\varepsilon_j,\qquad
C_p:=\mathbb E|S_H|^p,\qquad p=1/H.
\label{eq:TL-variation-constant}
\end{equation}
Then \(0<C_p<\infty\), and Theorem~2.1 of
\cite{Mishura-Schied2019} gives, for every sign array,
\begin{equation}
[x_\sigma]^p_{\mathbb T}(t)=C_pt,
\qquad x_\sigma\in\mathfrak X^{1/p}.
\label{eq:signed-TL-linear-variation}
\end{equation}
The signs \(\varepsilon_j\) in \eqref{eq:TL-variation-constant}
represent the Bernoulli-convolution constant; the coefficient array
\(\sigma\) in \eqref{eq:signed-TL-linear-variation} need not be random.
For the tent specialization \(T_H\), the constant \(c_p\) in
\eqref{eq:SZ-linear-variation} equals \(C_p\).

The critical H\"older estimate and the endpoint perturbation theorem
give the following stability statement without assuming intrinsic roughness.

\begin{corollary}[Stability of normalized Schied--Zhang paths under perturbed grids]
\label{cor:SZ-perturbed-grid}
Let
\begin{equation}
\pi_n=\{0=t_0^n<\cdots<t_{b^n}^n=1\},
\qquad
\delta_n:=\max_k|t_k^n-kb^{-n}|,
\label{eq:SZ-perturbed-partition}
\end{equation}
and assume
\begin{equation}
\delta_n=o(b^{-n}).
\label{eq:SZ-small-perturbation}
\end{equation}
Then
\begin{equation}
x\in V_p(\pi),
\qquad
[x]^p_\pi(t)=t.
\label{eq:SZ-perturbed-invariance}
\end{equation}
Consequently, the relative coarse-graining error with respect to \(\mathbb T^{(b)}\) vanishes along every alignment.
\end{corollary}

\begin{proof}
The critical H\"older estimate gives \(\omega_x(h)^p\le Ch\).  Therefore
\begin{equation}
b^n\omega_x(\delta_n)^p\le Cb^n\delta_n\to0.
\label{eq:SZ-endpoint-condition}
\end{equation}
Apply Proposition~\ref{prop:endpoint-perturbation} and then Proposition~\ref{thm:necessity}.
\end{proof}

The same proof applies to every normalized signed path
\(C_p^{-1/p}x_\sigma\), with \(b=2\), by
\eqref{eq:signed-TL-linear-variation} and its critical H\"older bound.

\begin{example}[Explicit critical orders]
\label{ex:SZ-p-greater-two}
For \(b=2\) and the tent map, \(\alpha=2^{-1/3}\) gives \(p=3\), while \(\alpha=2^{-2/7}\) gives \(p=7/2\).  More generally \(\alpha=b^{-1/p}\) produces any prescribed critical order \(p>1\) for which the Schied--Zhang nondegeneracy condition \cite{SchiedZhang2020} holds.
\end{example}

\paragraph{Deterministic phase dependence.}
Endpoint stability does not imply stability under shifts comparable to
the mesh. For the all-positive member this distinction is explicit.

\begin{proposition}[Phase dependence of the Takagi--Landsberg function]
\label{prop:TL-phase-dependence}
Let \(p>1\), \(H=1/p\), and let \(T_H\) and \(C_p\) be defined by
\eqref{eq:TL-positive-function} and \eqref{eq:TL-variation-constant}.
Then
\begin{equation}
\lim_{n\to\infty}\V^p_{\Pi(2^{-n},0)}(T_H;1)=C_p,
\qquad
\lim_{n\to\infty}\V^p_{\Pi(2^{-n},2^{-n-1})}(T_H;1)
=\frac{C_p}{2^p-1}.
\label{eq:TL-phase-limits}
\end{equation}
In particular, \(T_H\notin\mathscr R_p([0,1])\).
\end{proposition}

\begin{proof}
The first limit follows from \eqref{eq:signed-TL-linear-variation}.
For the second, put \(h=2^{-n}\) and
\(d_{n,k}=T_H((k+1)h)-T_H(kh)\), \(0\le k<2^n\).
The terms of levels below \(n\) are affine on each dyadic cell.
At all cell midpoints the level-\(n\) term has the same value, and
all higher-level terms vanish. Hence each interior half-shift increment is
\begin{equation}
T_H((k+3/2)h)-T_H((k+1/2)h)
=\frac{d_{n,k}+d_{n,k+1}}2,
\qquad 0\le k\le2^n-2.
\label{eq:TL-half-shift-increment}
\end{equation}
Writing \(b_j(k)\in\{0,1\}\) for the \(j\)-th binary digit of
\(k\), counted from the least significant digit, gives
\begin{equation}
h^{-H}d_{n,k}=\sum_{j=1}^{n}q^j(1-2b_j(k)).
\label{eq:TL-dyadic-digit-increment}
\end{equation}
If \(k\) has exactly \(r\) trailing ones, then addition of one
flips precisely its first \(r+1\) digits. The corresponding terms
in the average in \eqref{eq:TL-half-shift-increment} cancel.
Among \(0\le k\le2^n-2\), there are \(2^{n-r-1}\) such indices,
and their remaining \(n-r-1\) digits run through all binary patterns.
Set \(S_{H,N}=\sum_{j=1}^{N}q^j\varepsilon_j\), with
\(S_{H,0}=0\). Since \(Hp=1\), the sum over interior cells is exactly
\begin{equation}
\sum_{k=0}^{2^n-2}
\left|\frac{d_{n,k}+d_{n,k+1}}2\right|^p
=\sum_{r=0}^{n-1}2^{-(r+1)}q^{p(r+1)}
\mathbb E|S_{H,n-r-1}|^p.
\label{eq:TL-half-shift-finite-energy}
\end{equation}
The two boundary cells contribute \(O(h)\), by the critical
H\"older bound. Also \(|S_{H,N}|\le q/(1-q)\) and
\(S_{H,N}\to S_H\) uniformly in the signs. Dominated convergence
therefore gives
\begin{equation}
\lim_{n\to\infty}\V^p_{\Pi(2^{-n},2^{-n-1})}(T_H;1)
=C_p\sum_{r=0}^{\infty}2^{-(r+1)}q^{p(r+1)}
=C_p\frac{q^p}{2-q^p}
=\frac{C_p}{2^p-1},
\label{eq:TL-half-shift-series}
\end{equation}
where \(q^p=2^{1-p}\). The two limits in
\eqref{eq:TL-phase-limits} are different for every \(p>1\), so
Definition~\ref{def:p-roughness-intrinsic} cannot hold.
\end{proof}

\begin{remark}[Partition stability without intrinsic roughness]
\label{rem:SZ-intrinsic-verification}{\em 
For \(p=3\), \(C_3\simeq0.7469\), whereas the half-shift limit is
exactly \(C_3/7\simeq0.1067\). Thus the Schied--Zhang class is not
contained in \(\mathscr R_p\). This does not contradict
Corollary~\ref{cor:SZ-perturbed-grid}: the half-shift has displacement
\(2^{-n-1}\), which is comparable to the mesh rather than
\(o(2^{-n})\). The example shows that its small-displacement
assumption cannot in general be relaxed to an \(O(2^{-n})\) bound.
The failure also occurs for the case with all-positive coefficients at \(p=2\), while the next theorem shows that independent signs give a different result
at that order.}
\end{remark}

\paragraph{Independent signs, $p=2$.}
Choosing the signs in \(\mathfrak X^{1/2}\) independently gives a
non-Gaussian example of intrinsic quadratic roughness. 
\begin{theorem}[Rademacher Faber--Schauder paths are \(2\)-rough]
\label{thm:rademacher-schauder-rough}
Let \((\xi_{m,k})_{m\ge0,\ 0\le k<2^m}\) be independent Rademacher
random variables,
\[
\mathbb P(\xi_{m,k}=1)=\mathbb P(\xi_{m,k}=-1)=\frac12,
\]
and let \(c\neq0\) and \(a\in\mathbb R\). Consider the random
Faber--Schauder series
\begin{equation}
X(t)
=
X(0)+at
+
c\sum_{m=0}^{\infty}\sum_{k=0}^{2^m-1}
\xi_{m,k}e^{\mathbb T}_{m,k}(t),
\qquad t\in[0,1].
\label{eq:rademacher-schauder-path}
\end{equation}
Then
\[
\mathbb{P}\left( X\in\mathscr R_2([0,1])\right)=1
\qquad{\rm and}\qquad 
\mathbb{P}\left([X]^2(t)=c^2t\text{ for all }t\in[0,1]\right)=1.
\]
More precisely,
\begin{equation}
\Omega_2^{c^2{\rm id}}(X;\delta)
=
O_{\rm a.s.}\left(
\sqrt{\delta\log\frac{e}{\delta}}
\right),
\qquad \delta\downarrow0,
\label{eq:rademacher-schauder-self-averaging}
\end{equation}
and consequently
\begin{equation}
\mathfrak R_{2,L}(X)
=
O_{\rm a.s.}\left(
\sqrt{\frac{\log(eL)}{L}}
\right),
\qquad L\to\infty.
\label{eq:rademacher-schauder-mesoscopic}
\end{equation}
In particular, independent random signs almost surely produce the
mesoscopic cancellation of coarse-graining errors required in
Proposition~\ref{prop:microscope-free-characterization}.
\end{theorem}

\begin{proof}
Since the coefficients in \eqref{eq:rademacher-schauder-path} satisfy
\(\theta_{m,k}=c\xi_{m,k}\), the critical Schauder bound
\eqref{eq:critical-schauder-bound} holds with \(p=2\). Hence the sample
paths are \(1/2\)-H\"older, with a deterministic bound depending only on
\(|a|\) and \(|c|\).
Write
\[
Y(t)
:=
c\sum_{m\ge0}\sum_{k<2^m}
\xi_{m,k}e^{\mathbb T}_{m,k}(t),
\qquad
X(t)-X(0)=at+Y(t).
\]
Since the Haar system together with the constant function is an
orthonormal basis of \(L^2([0,1])\), Parseval's identity gives
\begin{equation}
\mathbb E[Y(s)Y(t)]
=
c^2\bigl(s\wedge t-st\bigr).
\label{eq:rademacher-bridge-covariance}
\end{equation}
Thus \(Y\) has the covariance kernel of a Brownian bridge.
Let
\[
\lambda=\{0=t_0<t_1<\cdots<t_N=t\}
\]
be a deterministic partition of \([0,t]\) with mesh at most \(\delta\),
and put \(h_i=t_{i+1}-t_i\). For
\[
\Delta_iY=Y(t_{i+1})-Y(t_i),
\]
the covariance matrix is
\begin{equation}
\Sigma_\lambda
=
c^2\bigl(\operatorname{diag}(h_i)-hh^\top\bigr),
\qquad h=(h_0,\ldots,h_{N-1})^\top.
\label{eq:rademacher-increment-covariance}
\end{equation}
In particular,
\[
\|\Sigma_\lambda\|_{\rm op}\le c^2\delta,
\qquad
\operatorname{tr}(\Sigma_\lambda^2)
\le c^4\delta,
\]
uniformly over all such partitions.

The vector \((\Delta_iY)_i\) is a linear image of the independent
Rademacher family. For a finite truncation, write this image as
\(A_\lambda\xi\). The quadratic-form matrix
\(A_\lambda^\top A_\lambda\) has the same nonzero eigenvalues as
the truncated increment covariance, which is bounded above by
\(\Sigma_\lambda\). Its operator norm is therefore at most
\(c^2\delta\), and its squared Hilbert--Schmidt norm is at most
\(c^4\delta\). The affine cross term has sub-Gaussian variance
proxy bounded by a constant times
\(h^\top\Sigma_\lambda h\le c^2\delta\sum_i h_i^2\le c^2\delta^2\).
The Hanson--Wright inequality \cite{HansonWright1971}, together with the
corresponding sub-Gaussian bound for the affine cross term, therefore
yields constants \(C,c_0>0\), depending only on \(a\) and \(c\), such
that
\begin{equation}
\mathbb P\left(
\left|
\V_\lambda^2(X;t)
-
\mathbb E \V_\lambda^2(X;t)
\right|>u
\right)
\le
C\exp\left[
-c_0\min\left\{
\frac{u^2}{\delta},
\frac{u}{\delta}
\right\}
\right].
\label{eq:rademacher-hanson-wright}
\end{equation}
The bound passes to the full series by uniform convergence of the
Schauder truncations. Moreover, \eqref{eq:rademacher-increment-covariance} gives
\begin{align}
\mathbb E \V_\lambda^2(X;t)
&=
c^2\left(t-\sum_i h_i^2\right)
+a^2\sum_i h_i^2 =
c^2t+(a^2-c^2)\sum_i h_i^2,
\label{eq:rademacher-mean-energy}
\end{align}
and hence
\begin{equation}
\left|
\mathbb E \V_\lambda^2(X;t)-c^2t
\right|
\le C\delta.
\label{eq:rademacher-mean-error}
\end{equation}
We now apply \eqref{eq:rademacher-hanson-wright} to the entropy net of
Lemma~\ref{lem:shifted-grid-net}, including the grid stopping points, with
\[
u_k=A\sqrt{\delta_k k}.
\]
The number of resulting grid--stopping-time pairs grows at most
exponentially in \(k\), whereas
\eqref{eq:rademacher-hanson-wright} is bounded by
\(C\exp(-c_0A^2k)\). Choosing \(A\) sufficiently large and applying
Borel--Cantelli gives, almost surely, for all sufficiently large \(k\),
\begin{equation}
\sup_{(\ell,a)\in\mathcal N_k}
\sup_{t\in[0,1]}
\left|
\V_{\Pi(\ell,a)}^2(X;t)-c^2t
\right|
\le C(\omega)\sqrt{\delta_k k}.
\label{eq:rademacher-net-bound}
\end{equation}
The \(1/2\)-H\"older bound controls the incomplete terminal cell.
To pass from the entropy net to the whole shell, let
\(\Pi(\ell,a)\) satisfy
$ \delta_{k+1}\le\ell\le\delta_k $
and choose the net representative and edited subpartitions supplied
by Lemma~\ref{lem:shifted-grid-net}.  Lemma~
\ref{lem:finite-shifted-grid-stability} applies with
\[
\delta=\frac{\delta_k}{2},
\qquad
\varepsilon=2\eta_k,
\]
since the two mesh sizes lie in
\([\delta_k/2,\delta_k]\), the deleted points lie in
$(0,\delta_k)\cup(1-\delta_k,1), $
and the corresponding edited partition points are at distance at
most \(2\eta_k\).
Since every sample path is \(1/2\)-H\"older, $
\omega_X(h)^2\le C h.$
Consequently the three error terms in
Lemma~\ref{lem:finite-shifted-grid-stability} are bounded by
\[
C\left[
\left(\frac{\eta_k}{\delta_k}\right)^{1/2}
+
\frac{\eta_k}{\delta_k}
+
\delta_k
\right].
\]
Here
\[
\eta_k=2^{-3k}=\delta_k^3,
\]
so the right-hand side is \(O(\delta_k)\).
The net energies are uniformly bounded by
Equation~\eqref{eq:rademacher-net-bound}; hence the convergence in
that equation extends uniformly to every shifted grid in the
\(k\)-th shell. Taking the supremum over finer shells proves
\eqref{eq:rademacher-schauder-self-averaging}.
Thus Definition~\ref{def:p-roughness-intrinsic} gives
\(X\in\mathscr R_2([0,1])\) and \([X]^2(t)=c^2t\).
Finally Proposition~\ref{prop:microscope-free-characterization} yields
\[
\mathfrak R_{2,L}(X)
\le
2\Omega_2^{c^2{\rm id}}(X;1/L),
\]
which is \eqref{eq:rademacher-schauder-mesoscopic}.
\end{proof}

\begin{remark}[Independent signs beyond the quadratic order]
\label{rem:randomsign-p}
{\em For \(p>1\), consider the critical random-sign series
\begin{equation}
X_p(t):=\sum_{m=0}^{\infty}2^{m(1/2-1/p)}
\sum_{k=0}^{2^m-1}\xi_{m,k}e^{\mathbb T}_{m,k}(t),
\qquad t\in[0,1],
\label{eq:rademacher-critical-p-path}
\end{equation}
with independent symmetric signs. Every realization belongs to
\(\mathfrak X^{1/p}\), so its dyadic \(p\)-th variation is \(C_pt\)
by \eqref{eq:signed-TL-linear-variation}.

The proof of Theorem~\ref{thm:rademacher-schauder-rough} uses the
quadratic structure of the energy and Parseval's identity. For
\(X_2\), corresponding to \(a=0\) and \(c=1\),
\eqref{eq:rademacher-increment-covariance} gives
\(\mathbb E|\Delta_I X_2|^2=|I|-|I|^2\).
There is no corresponding Parseval identity for the critical
\(p\)-th absolute moment when \(p\ne2\). In particular, independence
of the coefficient signs does not by itself establish the uniform
energy limit or the power-discrepancy hypothesis of
Proposition~\ref{prop:sign-discrepancy}.

For \(p=3\), numerical averages over \(200\) independent sign arrays
at level \(n=9\) give the following energies on
\(\Pi(2^{-n},\beta2^{-n})\) at time \(1\):
\begin{center}
\begin{tabular}{lcccc}
 & \(\beta=0\) & \(\beta=\tfrac14\) & \(\beta=\tfrac12\) & \(\beta=\tfrac34\)\\[2pt]\hline
Independent signs & \(0.7466\) & \(1.14\) & \(1.03\) & \(1.14\)
\end{tabular}
\end{center}
For comparison, the deterministic all-positive path at \(n=14\)
gives \(0.7469\), \(0.2229\), \(0.1067\), and \(0.2229\) at
the same four phases. Its dyadic and half-shift limits are proved in
Proposition~\ref{prop:TL-phase-dependence}.

These finite-resolution averages suggest phase dependence for
independent signs at \(p=3\), despite their common dyadic limit.
They do not prove almost-sure failure of intrinsic \(3\)-roughness,
and do not establish a claim for every \(p\ne2\). Such a conclusion
would require an asymptotic analysis of the shifted energies, beyond
the numerical evidence given here.

At \(p=2\), the finite-mesh expectation is described exactly by
\eqref{eq:rademacher-mean-energy}. With \(h=2^{-n}\), the two
boundary cells of a shifted grid give
\[
\mathbb E\V^2_{\Pi(h,\beta h)}(X_2;1)
=1-h+2\beta(1-\beta)h^2,\qquad 0\le\beta<1.
\]
Thus the value is exactly \(1-h\) at phase zero and differs by
only \(O(h^2)\) at other phases; all phases have limit \(1\).
Theorem~\ref{thm:rademacher-schauder-rough} supplies the stronger
almost-sure uniform convergence required for intrinsic roughness.}
\end{remark}

\begin{example}[Independent signs do not imply fourth-order roughness]
\label{ex:rademacher-p4}
Let $(\xi_{m,k})_{m\geq0,\,0\leq k<2^m}$ be independent symmetric
Rademacher random variables, and consider
\[
  X_4(t)
  :=\sum_{m=0}^{\infty}2^{m/4}
       \sum_{k=0}^{2^m-1}\xi_{m,k}e^{\mathbb T}_{m,k}(t),
  \qquad t\in[0,1].
\]
Every realization is $1/4$-H\"older continuous and has dyadic fourth
variation
\[
  [X_4]^4_{\mathbb T}(t)=C_4t,
  \qquad
  C_4:=\frac{13+12\sqrt2}{49}.
\]
Nevertheless,
\[
  \mathbb P\bigl(X_4\in\mathcal R_4([0,1])\bigr)=0.
\]
More precisely, the expected fourth energy along the half-shifted
dyadic grids satisfies
\begin{equation}
  \lim_{n\to\infty}
  \mathbb E\!\left[
    V^4_{\Pi(2^{-n},\,2^{-n-1})}(X_4;1)
  \right]
  =K_4
  :=\frac{6481+2364\sqrt2}{7595}
  >C_4.
  \label{eq:rademacher-p4-half-shift}
\end{equation}

\begin{proof}
Set
\[
  q:=2^{-3/4},\qquad
  a:=q^2=2^{-3/2},\qquad
  v:=\frac{a}{1-a}.
\]
The critical Schauder bound gives a deterministic constant $L$ such
that every realization satisfies
\[
  |X_4(t)-X_4(s)|\leq L|t-s|^{1/4}.
\]
In particular, for every finite partition $\lambda$ of $[0,1]$,
\begin{equation}
  V^4_\lambda(X_4;1)\leq L^4.
  \label{eq:rademacher-p4-energy-bound}
\end{equation}
The dyadic variation formula for the signed Takagi--Landsberg class
identifies its limiting energy constant with
\[
  C_4=\mathbb E|S|^4,
  \qquad S:=\sum_{j=1}^{\infty}q^j\varepsilon_j,
\]
where $(\varepsilon_j)_{j\geq1}$ are independent symmetric
Rademacher variables. For any square-summable real sequence $(c_j)$,
\begin{equation}
  \mathbb E\!\left(\sum_jc_j\varepsilon_j\right)^4
  =3\left(\sum_jc_j^2\right)^2-2\sum_jc_j^4.
  \label{eq:rademacher-fourth-moment}
\end{equation}
Since $q^4=1/8$, this gives
\[
  C_4=3v^2-\frac{2}{7}
     =\frac{13+12\sqrt2}{49}.
\]

Fix $n\geq1$, write $h:=2^{-n}$, and set
\[
  d_{n,k}:=X_4((k+1)h)-X_4(kh),\qquad 0\leq k<2^n.
\]
The terms of levels below $n$ are affine on each dyadic cell; the
level-$n$ tent has height $h^{1/4}/2$, and all higher-level terms
vanish at the cell midpoint. Consequently,
\begin{align*}
  &X_4((k+3/2)h)-X_4((k+1/2)h)=
    \frac{d_{n,k}+d_{n,k+1}}{2}
    +\frac{h^{1/4}}{2}(\xi_{n,k+1}-\xi_{n,k}),
    \qquad 0\leq k\leq2^n-2.
\end{align*}
Suppose that the binary expansion of $k$ has exactly $r$ trailing
ones, where $0\leq r\leq n-1$. In the average of the two adjacent
dyadic increments, the contribution from their common ancestor at
level $n-r-1$ cancels. Contributions from finer levels involve
distinct independent coefficients, while those from coarser levels
coincide. Thus the normalized half-shift increment has the same law as
\[
  Y_{r,n}
  :=\frac12\sum_{j=0}^{r}q^j(\varepsilon_j+\varepsilon'_j)
    +\sum_{j=r+2}^{n}q^j\eta_j,
\]
where all variables on the right are independent symmetric
Rademacher variables, and an empty sum is zero. There are exactly
$2^{n-r-1}$ such indices $k$.

The two boundary cells of the shifted grid contribute $O(h)$, by the
uniform H\"older bound. Hence
\[
  \mathbb E\!\left[V^4_{\Pi(h,h/2)}(X_4;1)\right]
  =\sum_{r=0}^{n-1}2^{-r-1}\mathbb E|Y_{r,n}|^4+O(h).
\]
For each fixed $r$, $Y_{r,n}$ converges uniformly in the signs to
\[
  Y_r
  :=\frac12\sum_{j=0}^{r}q^j(\varepsilon_j+\varepsilon'_j)
    +\sum_{j=r+2}^{\infty}q^j\eta_j.
\]
Moreover, $|Y_{r,n}|\leq(1-q)^{-1}$ uniformly in $r$ and $n$.
Dominated convergence therefore yields
\begin{equation}
  \lim_{n\to\infty}
  \mathbb E\!\left[V^4_{\Pi(2^{-n},2^{-n-1})}(X_4;1)\right]
  =\sum_{r=0}^{\infty}2^{-r-1}\mathbb E|Y_r|^4.
  \label{eq:rademacher-p4-carry-sum}
\end{equation}

The variance of $Y_r$ is
\[
  v_r
  :=\frac12\sum_{j=0}^{r}a^j+\sum_{j=r+2}^{\infty}a^j
   =\frac{1}{2(1-a)}
     +\frac{a(a-1/2)}{1-a}\,a^r.
\]
The sum of the fourth powers of its Rademacher coefficients equals
\[
  \frac18\sum_{j=0}^{r}q^{4j}
  +\sum_{j=r+2}^{\infty}q^{4j}
  =\frac17.
\]
Equation~\eqref{eq:rademacher-fourth-moment} consequently gives
\[
  \mathbb E|Y_r|^4=3v_r^2-\frac27.
\]
Writing
\[
  A:=\frac{1}{2(1-a)},\qquad
  B:=\frac{a(a-1/2)}{1-a},
\]
and summing the geometric series in
\eqref{eq:rademacher-p4-carry-sum}, we obtain
\[
  K_4
  =3\left(A^2+\frac{2AB}{2-a}+\frac{B^2}{2-a^2}\right)-\frac27
  =\frac{6481+2364\sqrt2}{7595}.
\]
The strict inequality also follows directly from
\[
  v_r-v
  =\frac{1/2-a}{1-a}\bigl(1-a^{r+1}\bigr)>0,
\]
which implies $\mathbb E|Y_r|^4>C_4$ for every $r\geq0$.

Finally, changing finitely many coefficient levels adds a
piecewise-linear path and therefore preserves membership in
$\mathcal R_4$. Independence across levels and the tail-event
property of $p$-roughness imply
\[
  \mathbb P\bigl(X_4\in\mathcal R_4([0,1])\bigr)\in\{0,1\}.
\]
If this probability were one, the dyadic fourth variation would
identify the intrinsic energy as $C_4t$. The definition of
$4$-roughness would then force
\[
  V^4_{\Pi(2^{-n},2^{-n-1})}(X_4;1)\longrightarrow C_4
  \qquad\text{almost surely}.
\]
The deterministic bound~\eqref{eq:rademacher-p4-energy-bound} would
imply convergence of the expectations to $C_4$, contradicting
\eqref{eq:rademacher-p4-half-shift}. The probability must therefore
be zero.
\end{proof}
\end{example}

\section{Energy occupation measures and local times}
\label{sec:fourier}
\subsection{Multidimensional paths and scalar \(p\)-energy}
\label{sec:multidimensional-scalar-energy}
The scalar theory may be extended to vector-valued paths  by replacing absolute values
of increments with their Euclidean norms in Definition \ref{def:p-roughness-intrinsic}. We denote
the resulting class by
$\mathscr R^{sc}_p([0,T];\mathbb R^d)$.
Let \(x\in C([0,T],\R^d)\) and define
\begin{equation}
\mu_{\lambda_n}^{p,x}
:=\sum_i |x(u_{i+1}^n)-x(u_i^n)|^p\,\delta_{u_i^n}.
\label{eq:vector-discrete-time-energy}
\end{equation}
We write \(x\in V_{p,\mathrm{sc}}(\lambda;\R^d)\) when these measures converge weakly to a finite nonatomic measure, and denote the cumulative limit by \([x]^p_{\lambda,\mathrm{sc}}\).  The stopped-sum criterion \eqref{eq:stopped-convergence} remains valid without change.
Similarly,  $x\in \mathscr R^{sc}_p([0,T];\mathbb R^d)$ if there exists a continuous
nondecreasing function $[x]^p_{\mathrm{sc}}:[0,T]\mapsto \mathbb{R}_+$ satisfying
$
  [x]^p_{\mathrm{sc}}(0)=0, [x]^p_{\mathrm{sc}}(T)>0
$ and
$$ 
\sup_{0<\ell\le\delta}
\sup_{0\le a<\ell}
\sup_{t\in[0,T]}
\left|
\sum_{\Pi(\ell,a)} \|x(u_{i+1}^n\wedge t )-x(u_i^n\wedge t)\|^p-[x]^p_{sc}(t)
\right|\  \mathop{\longrightarrow}^{\delta \to 0} \ 0.$$
Then $[x]^p_{sc}$ is the intrinsic $p-$energy of $x$.
For \(a_1,\ldots,a_m\in\R^d\), set
\begin{equation}
\mathfrak D_{p,d}(a_1,\ldots,a_m)
:=\left|\sum_{j=1}^m a_j\right|^p-\sum_{j=1}^m|a_j|^p.
\label{eq:vector-error}
\end{equation}
The only analytic input needed in the scalar proof is still available in Euclidean space:
\begin{equation}
\bigl||a+b|^p-|a|^p\bigr|
\le C_p\bigl(|a|^{p-1}|b|+|b|^p\bigr),
\qquad a,b\in\R^d,
\label{eq:vector-power-increment}
\end{equation}
which follows from the fundamental theorem of calculus for \(z\mapsto|z|^p\).  Consequently the alignment, endpoint-perturbation and coarse-graining arguments of Sections~\ref{sec:intrinsic-roughness}--\ref{sec:invariance} carry over with the Euclidean norm and constants independent of the dimension.

\begin{theorem}[Partition invariance of scalar $p-$energy]
\label{thm:vector-main-equivalence}
Let $x\in\mathscr R^{sc}_p([0,T];\mathbb R^d)$ with $p-$energy $[x]^p_{sc}$. Let \(\pi\in\mathcal A_{\rm unif}\) be a matched endpoint perturbation with displacement \(\delta_n\).  If
\begin{equation}
N_n\omega_x(\delta_n)^p\to0,
\label{eq:vector-main-roughness}
\end{equation}
then
\begin{equation}
x\in V_{p,\mathrm{sc}}(\pi;\R^d),
\qquad
[x]^p_{\pi,\mathrm{sc}}=[x]^p_{\mathrm{sc}}.
\label{eq:vector-main-invariance}
\end{equation}
The corresponding fixed-pair relative coarse-graining error criterion is obtained from Proposition~\ref{thm:main-equivalence} by replacing \(\mathfrak D_p\) with \(\mathfrak D_{p,d}\).
\end{theorem}

\begin{proof}
The endpoint perturbation proof uses only \eqref{eq:vector-power-increment} and H\"older's inequality; the relative statement uses the same grouped-increment identity.  Thus the scalar arguments apply word for word.
\end{proof}

At \(p=2\), this scalar energy contains only the trace of the matrix-valued quadratic variation.  Indeed, if
\begin{equation}
Q_{\lambda_n}^x:=\sum_i \Delta_i x(\Delta_i x)^\top\delta_{u_i^n}
\label{eq:vector-matrix-qv}
\end{equation}
converges entrywise to \(Q_\lambda^x\), then
\begin{equation}
\mu_\lambda^{2,x}=\operatorname{tr}Q_\lambda^x.
\label{eq:vector-trace-qv}
\end{equation}
Hence scalar-energy invariance does not imply invariance of the full quadratic-variation matrix.  We will re-examine this issue in Section \ref{sec:applications}.



\subsection{Energy occupation measures}
The scalar \(p\)-th variation may be used to define a weighted occupation measure, which retains the spatial positions at which the $p-$th variation  accumulates. We call this the {\it energy occupation measure}.
The energy occupation measure is defined exactly as in the scalar case by pushing \(\mu^{p,x}\) forward under \(t\mapsto x(t)\).

For a partition \(\lambda_n=\{u_i^n\}\) and \(t\in[0,T]\), define
\begin{equation}
\nu_{\lambda_n,t}^{p,x}
:=
\sum_i
|x(u_{i+1}^n\wedge t)-x(u_i^n\wedge t)|^p
\,\delta_{x(u_i^n\wedge t)}.
\label{eq:discrete-energy-occupation}
\end{equation}
For a finite signed Borel measure $\nu$ on $\mathbb{R}^d$,
we use the Fourier-transform convention
\[
  \widehat{\nu}(\xi)
  := \int_{\mathbb{R}^d} e^{i\xi\cdot y}\,\nu(dy),
  \qquad \xi\in\mathbb{R}^d.
\]
At zero frequency,
\begin{equation}
\widehat\nu_{\lambda_n,t}^{p,x}(0)=\V^p_{\lambda_n}(x;t).
\label{eq:occupation-zero-frequency}
\end{equation}
If \(x\in V_p(\lambda)\), set
\begin{equation}
\nu_t^{p,x;\lambda}
:=
x_\#\bigl(\mu_\lambda^{p,x}|_{[0,t]}\bigr).
\label{eq:limiting-energy-occupation-general}
\end{equation}

\begin{proposition}[Occupation-measure convergence]
\label{prop:occupation-convergence}
If \(x\in V_p(\lambda)\), then for every \(t\in[0,T]\),
\begin{equation}
\nu_{\lambda_n,t}^{p,x}\Longrightarrow\nu_t^{p,x;\lambda},
\label{eq:occupation-convergence}
\end{equation}
and hence \(\widehat\nu_{\lambda_n,t}^{p,x}(\xi)\to\widehat\nu_t^{p,x;\lambda}(\xi)\) for every fixed \(\xi\).
\end{proposition}

\begin{proof}
Apart from the single cell containing \(t\), Equation~\eqref{eq:discrete-energy-occupation} is the push-forward of the discrete time-energy measure under the continuous map \(s\mapsto x(s)\).  The unfinished-cell error is bounded by \(C\omega_x(|\lambda_n|)^p\), and weak convergence of \(\mu_{\lambda_n}^{p,x}\) gives the result.
\end{proof}

For the remainder of the section the reference sequence is \(\rho\), and we write \(\nu_t^{p,x}:=\nu_t^{p,x;\rho}\).

\subsection{Fourier-localized coarse-graining error}

For an alignment \(\rho^\star\) of \(\rho\) with \(\pi\), define the signed atomic coarse-graining error measure
\begin{equation}
\eta_{\pi,\rho^\star;n}^{p,x}(t)
:=
\sum_{k=0}^{N(\pi_n)-1}
\mathfrak D_{k,n}^p(x;t)\,\delta_{x(t_k^n)}.
\label{eq:fourier-error-measure}
\end{equation}
Its Fourier transform satisfies
\begin{equation}
\widehat\eta_{\pi,\rho^\star;n}^{p,x}(t,0)
=
\Def^p_{\pi,\rho^\star;n}(x;t).
\label{eq:fourier-error-zero}
\end{equation}

The scalar coarse-graining estimate admits the following Fourier version.  Write \(\Theta_n=\Theta_n(x;\pi,\rho^\star)\).  If \(A_n(t)\) is the grouped energy of Section~\ref{sec:invariance}, then
\begin{equation}
\begin{aligned}
&\left|
\widehat\nu_{\pi_n,t}^{p,x}(\xi)
-
\widehat\nu_{\rho_n^\star,t}^{p,x}(\xi)
-
\widehat\eta_{\pi,\rho^\star;n}^{p,x}(t,\xi)
\right|
\\
&\qquad\le
C_p\bigl(A_n(t)^{(p-1)/p}\Theta_n^{1/p}+\Theta_n\bigr)
+
|\xi|\,\omega_x(|\pi_n|)\V^p_{\rho_n^\star}(x;t).
\end{aligned}
\label{eq:fourier-coarse-graining-theta}
\end{equation}
Indeed, the first term is exactly the endpoint error from Proposition~\ref{prop:coarse-graining-estimate}; the additional term comes from relocating each fine energy atom from \(x(s_j^n)\) to its coarse representative \(x(t_k^n)\), using \(|e^{ia}-e^{ib}|\le|a-b|\).

\begin{theorem}[Cancellation of scalar coarse-graining errors]
\label{thm:fourier-equivalence}
Let \(x\in C([0,T])\cap V_p(\rho)\).  For any partition sequence \(\pi\) with vanishing mesh, the following are equivalent:
\begin{enumerate}[label=(\roman*)]
\item \(x\in\RelDef^p_{\pi\mid\rho}\);
\item \(x\in V_p(\pi)\) and \([x]^p_\pi=[x]^p_\rho\);
\item for every \(t\), \(\nu_{\pi_n,t}^{p,x}\Longrightarrow\nu_t^{p,x}\);
\item 
\  $ \widehat\eta_{\pi,\rho^\star;n}^{p,x}(t,\xi)\longrightarrow0$
 for some alignment \(\rho^\star\) and every $(t,\xi).$
 \item 
\  $ \widehat\eta_{\pi,\rho^\star;n}^{p,x}(t,\xi)\longrightarrow0$
 for every alignment \(\rho^\star\) and  every $(t,\xi)$.
\end{enumerate}
\end{theorem}

\begin{proof}
The equivalence of (i) and (ii) is Proposition~\ref{thm:main-equivalence}.  Equality of the time-energy measures in (ii), followed by Proposition~\ref{prop:occupation-convergence}, gives (iii); taking total masses gives the converse.  Under these conditions both discrete occupation transforms converge to the same limit.  Equation~\eqref{eq:fourier-coarse-graining-theta}, the alignment condition and boundedness of \(A_n\) then give (iv).  Finally, \eqref{eq:fourier-error-zero} gives (iv)\(\Rightarrow\)(i).
\end{proof}
Thus Fourier cancellation at fixed frequency  is not a stronger  roughness notion: it is an equivalent representation of the same relative partition-stability class.

\subsection{Quantitative Fourier windows and energy irregularity}
\label{subsec:energy-irregularity}

A finite atomic coarse-graining error measure cannot satisfy a useful global Fourier decay estimate.  The natural finite-level notion is therefore mesoscopic.  Let \(\Lambda_n\uparrow\infty\), \(\zeta_n\downarrow0\), and \(\varrho>0\).  We say that the alignment has mesoscopic Fourier \((p,\varrho)\)-coarse-graining error cancellation if
\begin{equation}
\sup_{t\in[0,T]}
\sup_{|\xi|\le\Lambda_n}
(1+|\xi|)^\varrho
\left|\widehat\eta_{\pi,\rho^\star;n}^{p,x}(t,\xi)\right|
\le\zeta_n.
\label{eq:mesoscopic-fourier-error}
\end{equation}

For \(0\le s<t\le T\), define the intrinsic energy occupation transform
\begin{equation}
\Phi_{s,t}^{x,p}(\xi)
:=
\int_{(s,t]}e^{i\xi\cdot x(r)}\,\mu_\rho^{p,x}(dr).
\label{eq:energy-occupation-transform}
\end{equation}
We call \(x\) \(p\)-energy \((\gamma,\varrho)\)-irregular when the energy is nontrivial and
\begin{equation}
|\Phi_{s,t}^{x,p}(\xi)|
\le
C|t-s|^\gamma(1+|\xi|)^{-\varrho}
\label{eq:energy-irregularity}
\end{equation}
for every \(s<t\) and \(\xi\).

\begin{theorem}[Mesoscopic transfer of energy irregularity]
\label{thm:mesoscopic-to-irregularity}
Assume \(x\in C([0,T])\cap V_p(\rho)\), let \(\rho^\star\) be an alignment with \(\pi\), and suppose \eqref{eq:mesoscopic-fourier-error} holds.  If
\begin{equation}
(1+\Lambda_n)^\varrho\Theta_n^{1/p}\to0,
\qquad
(1+\Lambda_n)^\varrho\Lambda_n\omega_x(|\pi_n|)\to0,
\label{eq:mesoscopic-window-scales}
\end{equation}
then any uniform discrete occupation estimate of the form
\begin{equation}
(1+|\xi|)^\varrho
\left|\widehat\nu_{\lambda_n,t}^{p,x}(\xi)-
\widehat\nu_{\lambda_n,s}^{p,x}(\xi)\right|
\le C|t-s|^\gamma+o(1),
\qquad |\xi|\le\Lambda_n,
\label{eq:mesoscopic-occupation-estimate}
\end{equation}
valid for either \(\lambda_n=\pi_n\) or \(\lambda_n=\rho_n^\star\), transfers to the other family and passes to the limit.  In particular, if \eqref{eq:mesoscopic-occupation-estimate} holds for one of the two discrete occupation families and \([x]^p_\rho(T)>0\), then \(x\) is \(p\)-energy \((\gamma,\varrho)\)-irregular in the sense of \eqref{eq:energy-irregularity}.
\end{theorem}

\begin{proof}
At \(\xi=0\), \eqref{eq:mesoscopic-fourier-error} gives uniform scalar coarse-graining error convergence, hence boundedness of \(A_n\) by \eqref{eq:exact-error}.  Multiplying \eqref{eq:fourier-coarse-graining-theta} by \((1+|\xi|)^\varrho\) and taking \(|\xi|\le\Lambda_n\) gives an error tending to zero by \eqref{eq:mesoscopic-window-scales}.  This transfers \eqref{eq:mesoscopic-occupation-estimate} between the two discretizations.  Proposition~\ref{prop:occupation-convergence} along the aligned reference sequence then gives \eqref{eq:energy-irregularity} for each fixed \(\xi\).
\end{proof}

The mesoscopic coarse-graining error cancellation controls stability under coarse graining, whereas energy irregularity is an additional spatial cancellation property.

\begin{theorem}[H\"older--Sobolev regularity of the \(p\)-energy local time]
\label{thm:holder-sobolev-local-time}
Let \(x:[0,T]\to\R^d\) satisfy \eqref{eq:energy-irregularity} for some \(0<\gamma\le1\) and \(\varrho>d/2\).  Then \(\nu_t^{p,x}\) has an \(L^2(\R^d)\)-density \(L_t^{p,x}\), and for every
\(r<\varrho-d/2\),
\begin{equation}
\|L_t^{p,x}-L_s^{p,x}\|_{H^r(\R^d)}
\le C_r|t-s|^\gamma.
\label{eq:local-time-holder-sobolev}
\end{equation}
If \(\varrho>d\), Sobolev embedding yields a spatially continuous version.  Partition invariance of the energy occupation measure implies invariance of this density up to Lebesgue-null sets.
\end{theorem}

\begin{proof}
By \eqref{eq:energy-irregularity}, \(|\widehat\nu_t^{p,x}(\xi)|\lesssim(1+|\xi|)^{-\varrho}\), which belongs to \(L^2(\R^d)\) when \(\varrho>d/2\).  Plancherel gives the density.  The Fourier transform of \(L_t^{p,x}-L_s^{p,x}\) is \(\Phi_{s,t}^{x,p}\), and the Fourier characterization of \(H^r\) gives \eqref{eq:local-time-holder-sobolev} exactly when \(2r-2\varrho<-d\).  The spatial continuity statement follows from Sobolev embedding; the final invariance statement follows from partition invariance of the energy occupation measure.
\end{proof}
 
\begin{remark}[Relation with $\rho$ -irregularity]
\label{rem:cg-besov-obstruction}
Catellier and Gubinelli~\cite{CatellierGubinelli2016} call a path
$\rho $-irregular if, for some $\gamma>1/2$,
\[
  \left|\int_s^t e^{i\xi\cdot x(r)}\,dr\right|
  \le C|t-s|^\gamma(1+|\xi|)^{-\rho}
  \qquad\text{for all }s<t\text{ and }\xi.
\]
This property concerns Fourier decay of the (unweighted) occupation
measure, whereas $p$-roughness concerns stability of discrete
$p$-energies across sampling grids: the energy-weighted condition above
uses $d[x]^p(r)$ in place of $dr$.  When $[x]^p(t)=ct$, the energy-weighted condition \eqref{eq:energy-irregularity}  is equivalent to the $\rho-$irregularity \cite{CatellierGubinelli2016,GaleatiGubinelli2024} with the same exponents.

  \(p\)-roughness and \(\rho \)-irregularity
control different aspects of the path, but  impose compatible
constraints on its time regularity.
Recall (Lemma \ref{lem:p-rough-besov}) that  \(p\)-roughness implies 
\begin{equation}
x\in B^{1/p}_{1,\infty}([0,T]).
\label{eq:p-rough-besov}
\end{equation}
On the other hand
 $\rho$-irregularity implies \cite[Proposition~3.8(c)]{galeati2023}
\begin{equation}
x\notin B^\alpha_{1,\infty}([0,T])
\qquad\text{for every}\qquad
\alpha>
\frac{1-\gamma}{\rho};
\label{eq:cg-besov-obstruction}
\end{equation}
Consequently, if a path is simultaneously \(p\)-rough and
\((\rho,\gamma)\)-irregular, then
\[
\frac1p\le\frac{1-\gamma}{\rho},
\qquad\text{equivalently}\qquad
\rho\le p(1-\gamma).
\]
Since \(\gamma>1/2\) in the usual definition of \(\rho\)-irregularity,
this yields \(\rho<p/2\).
\end{remark}

\subsection{Relation with higher-order pathwise local time}
\label{subsec:energy-higher-order-local-time}

For even integers $p\ge2$, the energy occupation measure is directly
related to the order-$p$ local time defined in~\cite{ContPerkowski2019}. For a partition
$\lambda_n=\{u_i^n\}$ and $x\in V_p(\lambda)\cap C([0,T],\mathbb{R})$ let
$a_i^n(t)=x(u_i^n\wedge t)$, $b_i^n(t)=x(u_{i+1}^n\wedge t)$,
and $I(u,v)=(\min\{u,v\},\max\{u,v\}]$. The discrete level-crossing
local time, in the normalization of \cite[Definition~3.1]{ContPerkowski2019}, is
\begin{equation}
 \ell_{\lambda_n,t}^{p,x}(a)
 :=\sum_i\mathbf 1_{I(a_i^n(t),b_i^n(t))}(a)
                 |b_i^n(t)-a|^{p-1}.
 \label{eq:energy-CP-discrete-local-time}
\end{equation}
Each crossing kernel has integral $|v-u|^p/p$. Thus, whereas
$\nu_{\lambda_n,t}^{p,x}$ places each increment's energy at its
starting point, $p\ell_{\lambda_n,t}^{p,x}(a)\,da$ distributes that
energy over the   span of the increment.

Let $g$ be continuous on $K=x([0,T])$, with modulus of continuity 
$\omega_g$. Then
\begin{equation}
 \left|
 p\int_{\mathbb R}g(a)\ell_{\lambda_n,t}^{p,x}(a)\,da
 -\int_{\mathbb R}g(a)\nu_{\lambda_n,t}^{p,x}(da)
 \right|
 \le
 \omega_g\bigl(\omega_x(|\lambda_n|)\bigr)
 \V_{\lambda_n}^p(x;t).
 \label{eq:energy-CP-test-function-bound}
\end{equation}
Indeed, on each crossing interval the difference between $g(a)$
and its value at the starting point is bounded by the displayed
modulus, and integration of the kernel gives its energy divided by
$p$. If $x\in V_p(\lambda)$, the right-hand side tends to zero
uniformly in $t$. The convergence of the occupation measure 
therefore implies, for every $t$
\begin{equation}
 p\ell_{\lambda_n,t}^{p,x}(a)\,da
 \Longrightarrow\nu_t^{p,x;\lambda}.
 \label{eq:energy-CP-measure-convergence}
\end{equation}

Consequently, whenever the  $p$-local time
$\ell_t^{p,x;\lambda}$ exists, it satisfies the occupation identity
\begin{equation}
 p\int_{\mathbb R}g(a)\ell_t^{p,x;\lambda}(a)\,da
 =\int_{[0,t]}g(x(s))\,\mu_\lambda^{p,x}(ds).
 \label{eq:energy-CP-occupation-identity}
\end{equation}
In particular, if $\mu_\lambda^{p,x}$ is the common energy measure
considered above and $\nu_t^{p,x}(da)=L_t^{p,x}(a)\,da$, then
$L_t^{p,x}=p\ell_t^{p,x;\lambda}$ almost everywhere. Existing
order-$p$ local times along two sequences with the same limiting
$p$-energy measure therefore coincide almost everywhere in space,
at each time.

Existence of the occupation density alone does not establish the
weak $L^q$ convergence and weak continuity in time required by this definition. The Fourier--Sobolev criterion above
identifies the possible limit; an additional bound on the discrete
crossing local times is needed to obtain the stronger convergence.
For example, when $L^{p,x}$ is continuous in time with values in
$L^2$, the bound
$\sup_n\sup_{t\in[0,T]}\|\ell_{\lambda_n,t}^{p,x}\|_{L^2}<\infty$
upgrades \eqref{eq:energy-CP-measure-convergence} to weak $L^2$
convergence to $L_t^{p,x}/p$ at every $t$.

The normalization also agrees with the higher-order Tanaka formula of \cite[Thm 3.2]{ContPerkowski2019}:
for  $f\in W^{p,2}_{\rm loc}$, the local-time correction becomes
\begin{equation}
 \frac1{(p-1)!}\int_{\mathbb R}f^{(p)}(a)
                    \ell_t^{p,x;\lambda}(a)\,da
 =\frac1{p!}\int_{[0,t]}f^{(p)}(x(s))\,\mu_\lambda^{p,x}(ds).
 \label{eq:energy-CP-Tanaka-normalization}
\end{equation}
This connects the local time to the $p$-th variation term ("It\^o term") appearing in the change of variable formula of \cite{ContPerkowski2019}.
 The energy occupation measure itself
remains meaningful for any $p>1$. We will examine the case $p\notin \mathbb{N}$ in the next section.
\section{Application to pathwise calculus}\label{sec:applications}

\subsection{Robust formulation of higher-order F\"ollmer--It\^o calculus}
\label{sec:robust-fractional-calculus}

Let \(p=m+\alpha\), with \(m=\lfloor p\rfloor\) and \(0<\alpha<1\), and let \(\mathcal E_p\) denote the zero-remainder class of \cite[Theorem~2.12]{ContJin2024}.  For a partition \(\lambda_n=\{u_i^n\}\), write
\begin{equation}
\mathcal I_{\lambda_n}^{p,f}(x;t)
:=
\sum_i\sum_{j=1}^m
\frac{f^{(j)}(x(u_i^n\wedge t))}{j!}
\bigl(x(u_{i+1}^n\wedge t)-x(u_i^n\wedge t)\bigr)^j.
\label{eq:CJ-compensated-sum}
\end{equation}
The additional pathwise hypothesis in the zero-remainder theorem is the level-set condition
\begin{equation}
\int_{[0,T]}\mathbf 1_{\{x(s)=a\}}\,\mu_\lambda^{p,x}(ds)=0
\qquad\text{for every }a\in\R.
\label{eq:CJ-level-set-condition}
\end{equation}
Thus $\mu^{p,x}_\lambda(\{t, x(t)=a\})=0$ whenever $\nu_T^{p,x;\lambda}$ is absolutely continuous.

Under \(x\in V_p(\lambda)\), \eqref{eq:CJ-level-set-condition}, and \(f\in\mathcal E_p\), \cite[Theorem~2.12]{ContJin2024} gives
\begin{equation}
\lim_{n\to\infty}\mathcal I_{\lambda_n}^{p,f}(x;t)
=
f(x(t))-f(x(0)).
\label{eq:CJ-zero-remainder-formula}
\end{equation}

\begin{corollary}[Partition-robust zero-remainder calculus]
\label{cor:intrinsic-fractional-calculus}
Let \(p>1\) be noninteger.

\begin{enumerate}[label=(\roman*)]
\item If \(x\in V_p(\rho)\cap V_p(\pi)\), \([x]^p_\rho=[x]^p_\pi\), and the common energy measure satisfies \eqref{eq:CJ-level-set-condition}, then for every \(f\in\mathcal E_p\) the compensated sums along \(\rho\) and \(\pi\) converge to the same limit \(f(x(t))-f(x(0))\).
\item If \(x\in\mathscr R_p([0,T])\), \(\pi\in\mathcal A_{\rm unif}\), the intrinsic energy measure \(\mu^{p,x}\) satisfies \eqref{eq:CJ-level-set-condition} and 
\begin{equation}
N_n\omega_x(\delta_n)^p\to0,
\label{eq:intrinsic-calculus-endpoint-condition}
\end{equation}
then the same conclusion holds along \(\pi\).
\end{enumerate}
\end{corollary}

\begin{proof}
Equality of the cumulative \(p\)-variation functions is equality of the corresponding nonatomic time-energy measures, so the level-set condition is transported unchanged.  Apply \cite[Theorem~2.12]{ContJin2024}.  Part~(ii) follows from Theorem~\ref{thm:p-rough-invariance}.
\end{proof}

For even integer \(p\), the same argument applied to the higher-order F\"ollmer formula of Cont--Perkowski \cite{ContPerkowski2019} shows that the compensated integral and its \(p\)-energy correction are invariant whenever the limiting \(p\)-energy measure is invariant.  We do not restate that partitionwise formula here.

\subsection{Nonzero fractional remainders and marked energy measures}

Outside the zero-remainder class, Cont--Jin associate the fractional remainder with a two-point marking.  We retain only the measure-theoretic object needed for partition stability.  Let \(P_f^p\) be a Cont--Jin marking with state space \(X_f\).  For a partition \(\lambda_n\), define
\begin{equation}
\Xi_{\lambda_n}^{p,f,x}
:=
\sum_i
|x(u_{i+1}^n)-x(u_i^n)|^p
\,\delta_{u_i^n}\otimes
\delta_{P_f^p(x(u_i^n),x(u_{i+1}^n))},
\label{eq:CJ-marked-energy-measure}
\end{equation}
with zero increments omitted.  Its first marginal is exactly the discrete time-energy measure \(\mu_{\lambda_n}^{p,x}\).

\begin{proposition}[Transfer under marked-energy invariance]
\label{prop:CJ-marked-robustness}
Suppose \(x\in C([0,T])\cap V_p(\rho)\) and
\begin{equation}
\Xi_{\rho_n}^{p,f,x}\Longrightarrow\Xi^{p,f,x},
\qquad
\Xi_{\pi_n}^{p,f,x}\Longrightarrow\Xi^{p,f,x}
\label{eq:CJ-marked-invariance}
\end{equation}
for the same finite limit.  Then \(x\in V_p(\pi)\) and \(\mu_\pi^{p,x}=\mu_\rho^{p,x}\).  Moreover, any continuous remainder observable \(\widehat G_f^p\) may be integrated against the two discrete marked measures with the same limit whenever either \(\widehat G_f^p\) is bounded or the family is uniformly integrable with respect to \(|\widehat G_f^p|\).
\end{proposition}

\begin{proof}
Projection of \eqref{eq:CJ-marked-invariance} onto the time coordinate gives the same weak limit for the two scalar time-energy measures.  Since this common first marginal is nonatomic, \(\{t\}\times X_f\) has zero mass under the limiting marked measure.  Hence the Portmanteau theorem applies to the stopped test function
\(
(s,z)\mapsto \mathbf 1_{[0,t]}(s)\widehat G_f^p(z)
\)
when \(\widehat G_f^p\) is bounded and continuous.  The uniformly-integrable case follows by truncation.
\end{proof}

Thus common marked-measure convergence is the substantive extra assumption: scalar partition invariance alone determines only the first marginal of the remainder measure.

A useful sufficient condition reduces marked invariance back to scalar energy invariance.

\begin{proposition}[Diagonal criterion for marked-energy invariance]
\label{prop:CJ-diagonal-marking}
Let \(K=x([0,T])\).  Suppose there is a continuous map \(P_{f,0}^p:K\to X_f\) such that
\begin{equation}
\sup_{\substack{a,b\in K\\0<|a-b|\le\varepsilon}}
 d_{X_f}\bigl(P_f^p(a,b),P_{f,0}^p(a)\bigr)
\longrightarrow0
\qquad(\varepsilon\downarrow0).
\label{eq:CJ-mark-diagonal-continuity}
\end{equation}
If \(x\in V_p(\lambda)\), then
\begin{equation}
\Xi_{\lambda_n}^{p,f,x}
\Longrightarrow
\bigl(t\mapsto(t,P_{f,0}^p(x(t)))\bigr)_\#\mu_\lambda^{p,x}.
\label{eq:CJ-mark-diagonal-limit}
\end{equation}
Thus, scalar \(p\)-energy invariance implies marked-energy invariance for markings satisfying \eqref{eq:CJ-mark-diagonal-continuity}.
\end{proposition}

\begin{proof}
Vanishing mesh and continuity of \(x\) make the marking error uniform over the partition cells.  Since \(P_{f,0}^p(K)\) is compact, the relevant continuity is uniform near the limiting diagonal marks.  Testing against bounded continuous functions then reduces \eqref{eq:CJ-mark-diagonal-limit} to weak convergence of the scalar time-energy measures.
\end{proof}

\subsection{Intrinsic $p$-th variation tensor }
Another viewpoint is to examine the $p$-roughness of one-dimensional projections of a vector-valued function $x\in C([0,T];\mathbb R^d)$.
In the case where $p$ is an even integer,  we can associate  an intrinsic $p-$variation tensor  to $x$ provided that it has  $p$-rough projections.  This  $p$-variation tensor is the object which appears in the higher-order change of variable formula \cite{ContPerkowski2019}.

\begin{proposition}[Intrinsic tensor variation by polarization]
\label{prop:intrinsic-tensor-variation}
Let $p=2k\geq 2$ be an even integer,  $x\in C([0,T];\mathbb R^d)$ and
\[
  N=\dim\operatorname{Sym}^p(\mathbb R^d)
   =\binom{d+p-1}{p}.
\]
For a partition $\pi=\{0=t_0<\cdots<t_M=T\}$, define
\[
  Q_\pi^{(p)}(x;t)
  :=
  \sum_{j=0}^{M-1}
  \bigl(x(t_{j+1}\wedge t)-x(t_j\wedge t)\bigr)^{\otimes p}.
\]
Assume  there exist vectors
$v_1,\ldots,v_N\in\mathbb R^d$ such that
$
  \{v_r^{\otimes p}:1\leq r\leq N\}
  \quad\text{is a basis of }
  \operatorname{Sym}^p(\mathbb R^d),$
and $A_r\in C([0,T],\mathbb{R}_+)$ such that scalar projections $x_r=x.v_r$ satisfy
\[ 
\Omega_p^{A_r}(x_r;\delta)
=
\sup_{0<\ell\le\delta}
\sup_{0\le a<\ell}
\sup_{t\in[0,T]}
\left|
\V^p_{\Pi(\ell,a)}(x_r;t)-A_r(t)
\right|\mathop{\longrightarrow}^{\delta \to 0}0
\quad{\rm with}\quad  A_r(0)=0.
\]
The functions $A_r$ are allowed to vanish.
Then there exists a unique continuous function of bounded
variation
\[
  Q^{(p)}\in C([0,T],\operatorname{Sym}^p(\mathbb R^d)\ )
  \qquad Q^{(p)}(0)=0,
\]
such that, for any  tensor norm,
\[
  \lim_{\delta\downarrow0}
  \sup_{0<\ell\leq\delta}
  \sup_{0\leq a<\ell}
  \sup_{t\in[0,T]}
  \left\|
    Q_{\Pi(\ell,a)}^{(p)}(x;t)-Q^{(p)}(t)
  \right\|
  =0
\]
and the following properties hold:
\begin{enumerate}
\item
Every linear projection $v\cdot x$ has a continuous
limiting $p$-energy, uniformly over shifted uniform grids
and stopping times, given by
\[
  A_v(t):=Q^{(p)}(t)[v,\ldots,v].
\]
In particular, $A_{v_r}=A_r$.

\item
For any $v_1',\ldots,v_p'\in\mathbb R^d$,
\[
  Q^{(p)}(t)[v_1',\ldots,v_p']
  =
  \frac1{p!}
  \sum_{J\subseteq\{1,\ldots,p\}}
  (-1)^{p-|J|}
  A_{\sum_{j\in J}v_j'}(t),
\]
where $A_0=0$.

\item
For $0\leq s<t\leq T$, the increment
$Q^{(p)}(t)-Q^{(p)}(s)$ belongs to the closed convex cone
\[
  \mathcal K_p
  :=
  \overline{\operatorname{cone}
  \{z^{\otimes p}:z\in\mathbb R^d\}}.
\]
Consequently, every $A_v$ is nondecreasing.

\item
The Euclidean scalar $p$-energy of $x$  converges
uniformly over shifted uniform grids and stopping times.
Its limit is
\[
  A_{\mathrm{sc}}(t)
  :=
  \sum_{i_1,\ldots,i_k=1}^d
  Q^{(p)}(t)
  [e_{i_1},e_{i_1},\ldots,e_{i_k},e_{i_k}],
\]
where $(e_1,\ldots,e_d)$ is the canonical basis.
If $A_r(T)>0$ for at least one $r$, i.e $x_r\in \Rough_p([0,T],\mathbb{R})$ then
\[
  x\in\mathcal R_{p}^{\mathrm{sc}}([0,T];\mathbb R^d),
  \qquad [x]_{p,\mathrm{sc}}=A_{\mathrm{sc}}.
\]
\end{enumerate}
\end{proposition}

\begin{proof}
Equip $\operatorname{Sym}^p(\mathbb R^d)$ with its
canonical Euclidean inner product, so that
\[
  S[v,\ldots,v]=\langle S,v^{\otimes p}\rangle.
\]
The linear map
\[
  \mathsf E:
  \operatorname{Sym}^p(\mathbb R^d)\longrightarrow\mathbb R^N,
  \qquad
  \mathsf E(S)
  :=
  \bigl(S[v_r,\ldots,v_r]\bigr)_{r=1}^N,
\]
is an isomorphism by the choice of the vectors $v_r$.
Such a choice always exists: if a symmetric tensor $S$
is orthogonal to every $v^{\otimes p}$, then the
homogeneous polynomial $v\mapsto S[v,\ldots,v]$
vanishes identically, and hence $S=0$.
Since $p$ is even, for every partition $\pi$,
\[
  Q_\pi^{(p)}(x;t)[v,\ldots,v]
  =
  \sum_{j=0}^{M-1}
  \left|
    v\cdot\bigl(x(t_{j+1}\wedge t)-x(t_j\wedge t)\bigr)
  \right|^p
  =
  V_\pi^p(v\cdot x;t).
\]
Define
\[
  Q^{(p)}(t)
  :=
  \mathsf E^{-1}\bigl(A_1(t),\ldots,A_N(t)\bigr).
\]
For a constant $C_{\mathsf E}$ depending only on the
chosen directions and tensor norm,
\[
  \left\|Q_\pi^{(p)}(x;t)-Q^{(p)}(t)\right\|
  \leq
  C_{\mathsf E}
  \max_{1\leq r\leq N}
  \left|V_\pi^p(v_r\cdot x;t)-A_r(t)\right|.
\]
The assumed uniform convergence proves the 
tensor convergence, as well as uniqueness.
Continuity and $Q^{(p)}(0)=0$ follow from the definition.
Furthermore,
\[
  \operatorname{Var}_{[0,T]}(Q^{(p)})
  \leq
  C\sum_{r=1}^N A_r(T)<\infty,
\]
because the functions $A_r$ are nondecreasing.
For any $v\in\mathbb R^d$, contraction against
$v^{\otimes p}$ is a continuous linear functional.
Applying it to the tensor convergence yields the
uniform convergence of the scalar energies to $A_v$.
The formula in (ii) is the polarization identity for
the symmetric $p$-linear form $Q^{(p)}(t)$.

To prove (iii), introduce the completed-cell tensor sum
\[
  \overline Q_\pi^{(p)}(x;t)
  :=
  \sum_{j:\,t_{j+1}\leq t}
  \bigl(x(t_{j+1})-x(t_j)\bigr)^{\otimes p}.
\]
Denoting by $ \omega_x$ the modulus of continuity of $x$ we have
\[
  \sup_{t\in[0,T]}
  \left\|
    Q_\pi^{(p)}(x;t)
    -\overline Q_\pi^{(p)}(x;t)
  \right\|
  \leq C\,\omega_x(|\pi|)^p.
\]
Thus the completed-cell sums have the same uniform
limit. For every $s<t$,
\[
  \overline Q_\pi^{(p)}(x;t)
  -\overline Q_\pi^{(p)}(x;s)
  \in\mathcal K_p.
\]
Passing to the limit and using closedness of
$\mathcal K_p$ proves (iii).
For even $p$, contraction of any element of
$\mathcal K_p$ against $v^{\otimes p}$ is nonnegative;
hence $A_v(t)-A_v(s)\geq0$.
Define the linear functional
\[
  \mathsf L_p(S)
  :=
  \sum_{i_1,\ldots,i_k=1}^d
  S[e_{i_1},e_{i_1},\ldots,e_{i_k},e_{i_k}].
\]
Then for every $z\in\mathbb R^d$,
\[
  \mathsf L_p(z^{\otimes p})
  =
  \left(\sum_{i=1}^d z_i^2\right)^k
  =
  \|z\|^p.
\]
Composing   the tensor convergence with $\mathsf L_p$
therefore proves the  convergence of the
Euclidean scalar energies, with limit
$A_{\mathrm{sc}}=\mathsf L_p(Q^{(p)})$.
This limit is continuous, non-decreasing, and vanishes
at zero. Moreover,  we can pass to the limit in the inequality
\[
  V_\pi^p(v_r\cdot x;t)
  \leq
  \|v_r\|^p V_{\pi,\mathrm{sc}}^p(x;t)
\]
to obtain
$  A_r(t)\leq\|v_r\|^p A_{\mathrm{sc}}(t).$
Thus if $A_r(T)>0$ for some $r$ then
$A_{\mathrm{sc}}(T)>0$.
\end{proof}

\section{Renormalization group interpretation}
\label{sec:renormalization}

The coarse-graining mechanism of Sections~\ref{sec:pth-variation}--\ref{sec:invariance}
is  formulated through block maps on increment fields, which relates to the real-space renormalization group of Kadanoff \cite{Kadanoff1966} and Wilson and Kogut \cite{WilsonKogut1974}.
The $p$-roughness property   expresses  uniform asymptotic cancellation of the coarse-graining error along the renormalization flow.
 
The  class $\mathscr R_p^+$ of {\it strictly $p$-rough} paths defined in \eqref{def.strictlyrough}, whose
intrinsic energy is positive on every initial interval, is invariant
under the  critical scaling semigroup  
$x\mapsto\lambda^{1/p}x(\cdot/\lambda)$. Its energy transforms covariantly, with linear
profiles as the nonzero fixed energy profiles. When the initial energy
density exists and is finite and positive, the rescaled energy converges
to a linear profile. These are statements about energy profiles, not
convergence or exact self-similarity of the paths.  

\subsection{Coarse-graining as a renormalization group flow}
\label{rg:sec-flow}

Let $\lambda=\{0=s_0<\cdots<s_M=T\}$ be a finite partition and
$\pi=\{0=t_0<\cdots<t_N=T\}\subseteq\lambda$ a coarsening. For $t\in[0,T]$
the stopped fine increment field is
$\delta^\lambda_j(x;t):=x(s_{j+1}\wedge t)-x(s_j\wedge t)$, $0\le j<M$, and
$J_k^{\pi\leftarrow\lambda}:=\{j:[s_j,s_{j+1}]\subseteq[t_k,t_{k+1}]\}$; these
index sets partition $\{0,\dots,M-1\}$ because $\pi\subseteq\lambda$.

\begin{definition}[Block map, energy, drift]\label{rg:def-block}
For an increment field $a=(a_j)_{0\le j<M}$ indexed by $\lambda$, define the
\emph{block map} $\cB_{\pi\leftarrow\lambda}a:=\bigl(\sum_{j\in J_k^{\pi\leftarrow\lambda}}a_j\bigr)_{0\le k<N}$,
the \emph{$p$-energy} $\cE_p(a):=\sum_j|a_j|^p$, and, for $x\in C([0,T])$, the
\emph{energy drift}
\begin{equation}
\Delta_p(\pi,\lambda;x;t):=\V^p_\pi(x;t)-\V^p_\lambda(x;t).
\label{rg:eq-drift}
\end{equation}
\end{definition}

Telescoping gives $\cB_{\pi\leftarrow\lambda}\delta^\lambda(x;t)=\delta^\pi(x;t)$,
so the drift is the change of the observable $\cE_p$ under one block step:
$\Delta_p(\pi,\lambda;x;t)=\cE_p(\cB_{\pi\leftarrow\lambda}\delta^\lambda(x;t))-\cE_p(\delta^\lambda(x;t))$.

\begin{lemma}[Exact block flow]\label{rg:lem-flow}
Let $\pi\subseteq\mu\subseteq\lambda$ be finite partitions of $[0,T]$,
$x\in C([0,T])$ and $t\in[0,T]$. Then
\begin{enumerate}[label=(\roman*)]
\item $\cB_{\pi\leftarrow\lambda}=\cB_{\pi\leftarrow\mu}\circ\cB_{\mu\leftarrow\lambda}$
on increment fields indexed by $\lambda$;
\item $\Delta_p(\pi,\lambda;x;t)=\sum_{k=0}^{N-1}\Dp\bigl((\delta^\lambda_j(x;t))_{j\in J_k^{\pi\leftarrow\lambda}}\bigr)$,
and each block term is a sum of two-increment interactions $d_p(P,\delta)$
by the exact identity
\[
D_p(a_1,\ldots,a_q)=\sum_{j=2}^{q}d_p(a_1+\cdots+a_{j-1},a_j);
\]
\item $\Delta_p(\pi,\lambda;x;t)=\Delta_p(\pi,\mu;x;t)+\Delta_p(\mu,\lambda;x;t)$.
\end{enumerate}
\end{lemma}

\begin{proof}
(i) The $\lambda$-intervals inside a $\pi$-cell are partitioned by the
$\mu$-intervals inside that cell; summing over the former within the
latter gives the sum over the former. (ii) The sets
$J_k^{\pi\leftarrow\lambda}$ partition the fine indices, so
$\cE_p(\cB_{\pi\leftarrow\lambda}\delta^\lambda)-\cE_p(\delta^\lambda)
=\sum_k\bigl(|\sum_{j\in J_k}\delta_j|^p-\sum_{j\in J_k}|\delta_j|^p\bigr)$,
and the displayed identity expands each term. (iii) Add and
subtract $\V^p_\mu(x;t)$.
\end{proof}
Nested block maps thus form a discrete coarse-graining semigroup, and the
$p$-energy drift is an additive cocycle over it whose value is exactly the
block coarse-graining error of Section~\ref{sec:pth-variation}.
\subsection{$p$-roughness as fixed-point stability of the energy}
\label{rg:sec-fixed}
$p$-roughness may be interpreted as a fixed-point statement for the observable $\cE_p$ under the flow of
Lemma~\ref{rg:lem-flow}:   a block step changes the stopped
$p$-energy by an amount tending to zero.

\emph{Intrinsic fixed points.} For the dyadic multiresolution, the regular
$(b,r)$-blocking $\mathbb T_m^{b,r}\subseteq\mathbb T_m$ is one block step,
and its drift is
$C_m^p(x;b,r;t)=\Delta_p(\mathbb T_m^{b,r},\mathbb T_m;x;t)$.
Proposition~\ref{prop:microscope-free-characterization} says that
$x\in\mathscr R_p([0,T])$ if and only if
\[
 x\in V_p(\mathbb T),\qquad [x]^p_{\mathbb T}(T)>0,
 \qquad\lim_{L\to\infty}\mathfrak R_{p,L}(x)=0.
\]
The last condition tests all regular block sizes and phases in the
mesoscopic window $L\le b\le\lfloor2^m/L\rfloor$, uniformly in stopping
time. The continuous, nonzero reference energy is essential: constant
paths also have zero drift. The resulting characterization is unchanged
if the dyadic reference is replaced by any fixed $q$-adic uniform
multiresolution.

\emph{Relative fixed points.} Let $x\in C([0,T])\cap V_p(\rho)$, and let
$\rho^\star=(\rho_n^\star)$ be an alignment of $\rho$ with a comparison
sequence $\pi$. Using the alignment indices $\kappa_n(k)$ of Section~\ref{sec:invariance},
form the exact coarsening
$\widetilde\pi_n:=\{s^n_{\kappa_n(k)}:0\le k\le N_n\}\subseteq\rho_n^\star$
with repeated points deleted. Its stopped energy is the grouped energy
$A_n(t)$, since its increments are the grouped increments
$C_{k,n}(t)$ of \eqref{eq:grouped-increment}, with any empty groups omitted.
Adding and subtracting this energy gives
\begin{equation}
\begin{split}
\V^p_{\pi_n}(x;t)-\V^p_{\rho_n^\star}(x;t)
={}&\underbrace{\Delta_p(\widetilde\pi_n,\rho_n^\star;x;t)}_
 {=\ R^p_{\pi,\rho^\star;n}(x;t)} +\underbrace{\bigl(\V^p_{\pi_n}(x;t)-\V^p_{\widetilde\pi_n}(x;t)\bigr)}_
 {\text{endpoint perturbation}}.
\end{split}
\label{rg:eq-decomposition}
\end{equation}
By \eqref{eq:coarse-graining-estimate-theta}, the absolute value of the
second term is bounded by
$C_p(A_n(t)^{(p-1)/p}\Theta_n^{1/p}+\Theta_n)$, where
$\Theta_n=\Theta_n(x;\pi,\rho^\star)$. It tends to zero along an alignment
whenever $A_n(t)$ is bounded; the conclusion is uniform in time if these
energies are uniformly bounded in time.
Proposition~\ref{thm:necessity} identifies vanishing relative drift along
one alignment with
\[
 x\in\mathcal D^p_{\pi\mid\rho}
 \quad\Longleftrightarrow\quad
 x\in V_p(\pi),\qquad [x]^p_\pi=[x]^p_\rho,
\]
and shows that the coarse-graining error then vanishes uniformly in time along every
alignment. For a family $\mathfrak A$ of comparison sequences, the
corresponding relative class is
$\bigcap_{\pi\in\mathfrak A}\mathcal D^p_{\pi\mid\rho}$.
Theorem~\ref{thm:fourier-equivalence} gives the occupation-measure version
of this equivalence. The mesoscopic Fourier cancellation condition
introduced in   Section \ref{sec:fourier} controls the coarse-graining on growing
frequency windows.

The critical zoom below is a separate one-parameter semigroup on paths.
It rescales time and amplitude, and induces an explicit action on the
limiting energy profiles.

\subsection{The critical scaling flow}
\label{rg:sec-scaling}

For $\lambda\ge1$ and $x\in C([0,T])$ define the \emph{critical zoom}
\begin{equation}
(R_\lambda x)(t):=\lambda^{1/p}\,x(t/\lambda),\qquad t\in[0,T],
\label{rg:eq-zoom}
\end{equation}
which uses only $x|_{[0,T/\lambda]}$, and for a continuous nondecreasing
$A$ with $A(0)=0$ put $A_\lambda(t):=\lambda A(t/\lambda)$. For $\beta\in(0,1]$
write $K_\beta(z):=\sup_{s\ne t}|z(t)-z(s)|/|t-s|^\beta$.

\begin{theorem}[$p$-roughness and the critical scaling flow]\label{rg:thm-scaling}
Let $p>1$ and $T>0$.
\begin{enumerate}[label=(\roman*)]
\item \textup{(Exact covariance.)} $R_\lambda R_\mu=R_{\lambda\mu}$, and for
every $x\in C([0,T])$, every shifted uniform grid $\Pi(\ell,a)$ on $[0,T]$
and every $t\in[0,T]$,
\begin{equation}
\V^p_{\Pi(\ell,a)}(R_\lambda x;t)=\lambda\,\V^p_{\Pi(\ell/\lambda,\,a/\lambda)}(x;t/\lambda).
\label{rg:eq-covariance}
\end{equation}
Consequently, for every continuous nondecreasing $A$ with $A(0)=0$,
\begin{equation}
\Omega^{A_\lambda}_p(R_\lambda x;\delta)\le\lambda\,\Omega^{A}_p(x;\delta/\lambda),
\qquad 0<\delta\le T.
\label{rg:eq-Omega-covariance}
\end{equation}

\item \textup{(Invariance of $\mathscr R_p^+$.)} If
$x\in\mathscr R_p^+([0,T])$ with intrinsic energy $A$, then, for every
$\lambda\ge1$,
\[
 R_\lambda x\in\mathscr R_p^+([0,T]),\qquad [R_\lambda x]^p=A_\lambda.
\]
More generally, for $x\in\mathscr R_p([0,T])$ and a fixed $\lambda\ge1$,
$R_\lambda x\in\mathscr R_p([0,T])$ if and only if $A(T/\lambda)>0$,
and in that case the same energy identity holds.

\item \textup{(Energy fixed points.)} $A_\lambda=A$ for all $\lambda\ge1$ if
and only if $A(t)=A(T)\,t/T$. Hence the $p$-rough paths whose intrinsic
energy is invariant under the  flow are exactly those with linear
intrinsic energy.

\item \textup{(Attraction of energy profiles.)} Let
$x\in\mathscr R_p([0,T])$ with intrinsic energy $A$. If
$\alpha:=\lim_{s\downarrow0}A(s)/s\in(0,\infty)$, then
$x\in\mathscr R_p^+([0,T])$, $R_\lambda x\in\mathscr R_p^+([0,T])$
for every $\lambda\ge1$, and
\begin{equation}
\sup_{t\in[0,T]}\bigl|[R_\lambda x]^p(t)-\alpha t\bigr|\longrightarrow0
\qquad(\lambda\to\infty).
\label{rg:eq-attraction}
\end{equation}

\item \textup{(Uniqueness of the critical exponent.)} For $\gamma>0$ let
$(R^{(\gamma)}_\lambda x)(t):=\lambda^{\gamma}x(t/\lambda)$. Under the hypotheses
of \textup{(iv)}, $R^{(\gamma)}_\lambda x\in\mathscr R_p([0,T])$ with intrinsic
energy $\lambda^{\gamma p-1}A_\lambda$, which converges uniformly on $[0,T]$ to
$\alpha t$ if $\gamma=1/p$, to $0$ if $\gamma<1/p$, and diverges at every
$t>0$ if $\gamma>1/p$.

\item \textup{(Relevance of perturbations.)} Let $x\in\mathscr R_p([0,T])$ and  $z\in C([0,T])$.
\begin{enumerate}[label=(\alph*)]
\item If $K_\beta(z)<\infty$ for some $\beta\in(1/p,1]$, then for all
$\lambda\ge1$ and $\delta\in(0,T]$,
\begin{equation}
\sup_{0<\ell\le\delta}\ \sup_{0\le a<\ell}\ \sup_{t\in[0,T]}
\V^p_{\Pi(\ell,a)}(R_\lambda z;t)\ \le\ K_\beta(z)^p\,(T+2\delta)\,\delta^{\beta p-1}\,\lambda^{1-\beta p},
\label{rg:eq-irrelevant}
\end{equation}
and $x+z\in\mathscr R_p([0,T])$ with $[x+z]^p=[x]^p$.
\item If $z\in\mathscr R_q([0,T])$ with $q>p$, then
$x+z\in\mathscr R_q([0,T])$ with $[x+z]^q=[z]^q$, and
$\inf_{0\le a<\ell}\V^{p'}_{\Pi(\ell,a)}(x+z;T)\to\infty$ as $\ell\downarrow0$
for every $1<p'<q$; in particular $x+z\notin\mathscr R_p([0,T])$.
\end{enumerate}
\end{enumerate}
\end{theorem}

The invariant path class in (ii) is $\mathscr R_p^+$, while the fixed
objects in (iii) are energy profiles in the original time coordinate.
The attraction in (iv) asserts uniform convergence of energies, not of
paths. The uniqueness assertion in (v) uses the positive finite initial
energy density in (iv). If instead $A(s)\sim cs^\kappa$ as $s\downarrow0$, with $c,\kappa>0$,
then $\lambda^{\gamma p}A(t/\lambda)$ has a finite nonzero limit for
$t>0$ when $\gamma=\kappa/p$.
In (vi), smoother perturbations are irrelevant for the limiting
$p$-energy, and $q$-rough perturbations with $q>p$ dominate. Perturbations
at the same order admit no general addition rule: $x+(-x)=0$ whereas
$x+x$ has energy $2^p[x]^p$. These uses of relevance concern the energy
observable.

\begin{proof}
\emph{(i).} The semigroup property is immediate from the definition.
Fix $\Pi(\ell,a)$ on $[0,T]$ and $t\in[0,T]$. Dividing every point by
$\lambda$ maps the cells of $\Pi(\ell,a)$, stopped at $t$, onto the cells
of the grid $\{0,T/\lambda\}\cup\{(a+k\ell)/\lambda\in(0,T/\lambda)\}$, stopped
at $t/\lambda\le T/\lambda$. Each increment of $R_\lambda x$ is
$\lambda^{1/p}$ times the corresponding increment of $x$; hence
$\V^p_{\Pi(\ell,a)}(R_\lambda x;t)$ equals $\lambda$ times the stopped
$p$-energy of $x$ along that grid at $t/\lambda$. The grid
$\Pi(\ell/\lambda,a/\lambda)$ on $[0,T]$ has the same points in $(0,T/\lambda)$;
its cells beyond $T/\lambda$ contribute nothing when stopped at
$t/\lambda\le T/\lambda$, and its cell $[u_i,u_{i+1}]$ with
$u_i<T/\lambda<u_{i+1}$, if any, contributes $|x(t/\lambda)-x(u_i)|^p$ when
$u_i<t/\lambda$, exactly as the last cell $[u_i,T/\lambda]$ of the grid on
$[0,T/\lambda]$ does. This proves \eqref{rg:eq-covariance}. Taking
$0<\ell\le\delta$, $0\le a<\ell$, $t\in[0,T]$ on the left corresponds to
$0<\ell'\le\delta/\lambda$, $0\le a'<\ell'$, $t'\in[0,T/\lambda]$ on the right,
and $\lambda A(t/\lambda)=A_\lambda(t)$; enlarging the range of $t'$ to $[0,T]$
gives \eqref{rg:eq-Omega-covariance}.

\emph{(ii).} For each fixed $\lambda$, the covariance bound implies
$\Omega_p^{A_\lambda}(R_\lambda x;\delta)\to0$ as $\delta\downarrow0$.
If $x\in\mathscr R_p^+$, then
$A_\lambda(t)=\lambda A(t/\lambda)>0$ for every $t>0$, which proves
invariance of $\mathscr R_p^+$ and the energy identity.
For a general $x\in\mathscr R_p$, nondegeneracy of $A_\lambda$ is
exactly $A(T/\lambda)>0$. If this fails, monotonicity gives
$A_\lambda\equiv0$, so the rescaled path has zero limiting energy and
cannot belong to $\mathscr R_p$.

\emph{(iii).} If $A_\lambda=A$ for all $\lambda\ge1$ then
$A(T/\lambda)=A(T)/\lambda$, i.e.\ $A(s)=A(T)s/T$ for $0<s\le T$. Conversely a
linear $A$ satisfies $\lambda A(t/\lambda)=A(t)$.

\emph{(iv).} Since $A$ is nondecreasing and $A(s)>0$ for small $s>0$,
$A(T/\lambda)>0$ for every $\lambda$, so (ii) applies. For $t\in(0,T]$,
\[
\bigl|A_\lambda(t)-\alpha t\bigr|=t\,\Bigl|\frac{A(t/\lambda)}{t/\lambda}-\alpha\Bigr|
\le T\sup_{0<s\le T/\lambda}\Bigl|\frac{A(s)}{s}-\alpha\Bigr|\longrightarrow0 .
\]

\emph{(v).} $R^{(\gamma)}_\lambda x=\lambda^{\gamma-1/p}R_\lambda x$, so by (ii)
and homogeneity of the discrete $p$-energy it is $p$-rough
with energy $\lambda^{(\gamma-1/p)p}A_\lambda=\lambda^{\gamma p-1}A_\lambda$, and (iv)
gives $A_\lambda(t)=\alpha t+o(1)$ uniformly on $[0,T]$.

\emph{(vi)(a).} By \eqref{rg:eq-covariance},
$\V^p_{\Pi(\ell,a)}(R_\lambda z;t)=\lambda\V^p_{\Pi(\ell/\lambda,a/\lambda)}(z;t/\lambda)$.
The stopped sum on the right has at most $(t/\lambda)/(\ell/\lambda)+2\le(T+2\ell)/\ell$
nonempty cells, each of length at most $\ell/\lambda$, so each increment is
bounded by $K_\beta(z)(\ell/\lambda)^\beta$. Hence the right-hand side is at
most $\lambda\,\frac{T+2\ell}{\ell}K_\beta(z)^p(\ell/\lambda)^{\beta p}
=K_\beta(z)^p(T+2\ell)\ell^{\beta p-1}\lambda^{1-\beta p}$, and $\beta p>1$ gives
\eqref{rg:eq-irrelevant}. At $\lambda=1$, the bound tends to zero as $\delta\downarrow0$.
The stability under perturbations with vanishing energy proved in
Section~3 therefore gives $x+z\in\mathscr R_p$ with $[x+z]^p=A$.
If $x\in\mathscr R_p^+$, the same energy identity also gives
$x+z\in\mathscr R_p^+$.

\emph{(vi)(b).} For the stopped increments $a_i$ of $x$ along $\Pi(\ell,a)$,
$\ell\le\delta$,
\[
\sum_i|a_i|^q\le\max_i|a_i|^{q-p}\sum_i|a_i|^p
\le\omega_x(\delta)^{q-p}\bigl(A(T)+\Omega^A_p(x;\delta)\bigr)\longrightarrow0,
\]
by the definition of $\Omega_p^A$. Thus $x$ has vanishing $q$-energy
uniformly in mesh, phase and time. Applying the same perturbation
stability result with exponent $q$ to the path $z$ gives
$x+z\in\mathscr R_q$ with energy $A_z$. For $y:=x+z$ and $1<p'<q$,
$\V^q_{\Pi(\ell,a)}(y;T)\le\omega_y(\ell)^{q-p'}\V^{p'}_{\Pi(\ell,a)}(y;T)$ and
$\V^q_{\Pi(\ell,a)}(y;T)\ge A_z(T)-\Omega^{A_z}_q(y;\ell)$, whence
$\V^{p'}_{\Pi(\ell,a)}(y;T)\ge\bigl(A_z(T)-\Omega^{A_z}_q(y;\ell)\bigr)/\omega_y(\ell)^{q-p'}\to\infty$
uniformly in $a$; here $\omega_y(\ell)>0$ because $A_z(T)>0$ forces $y$ to
be nonconstant. A path whose $p'$-energies diverge is not in
$\mathscr R_{p'}$.
\end{proof}

\subsection{Linear energy profiles as RG fixed points}
\label{rg:sec-examples}

Brownian motion (Theorem~\ref{thm:brownian-intrinsic-roughness}), fractional
Brownian motion at the critical order
(Theorem~\ref{thm:fbm-intrinsic-p-roughness}), and the Rademacher
Faber--Schauder paths (Theorem~\ref{thm:rademacher-schauder-rough}) have
linear $p$-energy. Their sample paths belong almost surely to $\mathscr R_p^+$
at their respective critical orders and their energy profiles are fixed
by the scaling semigroup. For Brownian motion and the Rademacher series, the quantitative
estimates in Section~\ref{sec:examples} give mesoscopic coarse-graining error of order
$O(\sqrt{\log(eL)/L})$ for fixed $T$. The fractional Brownian theorem
provides qualitative convergence at every critical order.

The Takagi--Landsberg function $T_H$ illustrates why an aligned
coarse-graining scheme is insufficient. On $[0,1]$, set $p=1/H$ and
choose $m>n$, $b=2^{m-n}$. Then
\[
 \mathbb T_m^{b,0}=\Pi(2^{-n},0),\qquad
 \mathbb T_m^{b,b/2}=\Pi(2^{-n},2^{-n-1}).
\]
Proposition~\ref{prop:TL-phase-dependence} gives $p-$energies
$C_p$ and $C_p/(2^p-1)$, respectively, while the fine dyadic energy
converges to $C_p$. Therefore the coarse-graining error at the terminal time \(T\) tends to zero for the aligned blocking, whereas for the half-shifted blocking it converges to
\[
 \frac{C_p}{2^p-1}-C_p=-\frac{C_p(2^p-2)}{2^p-1}<0.
\]
Thus $T_H\notin\mathscr R_p$. By contrast,
Corollary~\ref{cor:SZ-perturbed-grid} gives stability under perturbations
with $\delta_n=o(b^{-n})$. In the dyadic counterexample the displacement
is of order $2^{-n}$; deleting the final interior point of the
half-shifted grid restores the required cardinality and changes the
terminal energy by only $O(2^{-n})$. 

For $p=2$, independent Rademacher signs produce intrinsic
roughness almost surely. This conclusion uses both covariance and the
concentration and continuity estimates in
Theorem~\ref{thm:rademacher-schauder-rough}; it does not follow from
independence and unit variance alone. For example, let independent
centered coefficients satisfy
\[
 \mathbb P(\theta_{m,k}=2^m)=\mathbb P(\theta_{m,k}=-2^m)=2^{-2m-1},
 \qquad\mathbb P(\theta_{m,k}=0)=1-2^{-2m}.
\]
Each has variance one, but
$\sum_m\sum_{k<2^m}\mathbb P(\theta_{m,k}\ne0)=\sum_m2^{-m}<\infty$.
Almost surely the series has only finitely many nonzero coefficients,
so it defines a piecewise-linear path with zero quadratic energy.
For \(p=4\), Example \ref{ex:rademacher-p4} proves almost-sure failure of intrinsic roughness for the critical independent-sign series. For \(p=3\), Remark \ref{rem:randomsign-p} provides numerical evidence of phase dependence.

The block-map interpretation  characterizes $p$-roughness by vanishing
mesoscopic coarse-graining error together with a nonzero continuous reference
energy. The scaling    preserves $\mathscr R_p^+$ and describes the
transformation of its energy profiles; paths with positive  initial energy
density   are attracted to a nonzero linear profile.

\end{document}